\documentclass[a4paper, 12pt, oneside]{article}

\usepackage[utf8]{inputenc}
\usepackage{cmap}

\usepackage[english,russian]{babel}
\usepackage[warn]{mathtext}
\usepackage[T2A]{fontenc}
\usepackage{microtype}

\usepackage[a4paper, margin=2cm]{geometry}

\usepackage[unicode=true, colorlinks=true, linkcolor=blue, citecolor=Green]{hyperref}
\usepackage[usenames,dvipsnames]{color}        

\usepackage{amssymb,amsfonts,amsmath,amsthm,mathtools}

\usepackage{mathrsfs}                    

\usepackage[inline]{enumitem}            
\usepackage{indentfirst}                 
\usepackage{titlesec}

\renewcommand{\le}{\leqslant}
\renewcommand{\ge}{\geqslant}

\DeclareMathOperator{\Dom}{Dom}
\DeclareMathOperator{\meas}{meas}
\DeclareMathOperator{\rank}{rank}
\DeclareMathOperator{\Ran}{Ran}
\DeclareMathOperator{\Ker}{Ker}
\DeclareMathOperator{\clos}{clos}
\DeclareMathOperator{\esssup}{ess-sup}
\DeclareMathOperator{\const}{const}

\titleformat{\part}[hang]{ \large\bfseries}{Глава~\thepart.}{0.5ex}{\centering}[]
\titleformat{\section}[hang]{ \scshape\bfseries}{\S\thesection.}{0.5ex}{\centering}[]
\titleformat{\subsection}[runin]{\bfseries}{\thesubsection.}{0.5ex}{}[.]

\numberwithin{equation}{section}

\theoremstyle{plain}
\newtheorem{thrm}{\scshape Теорема}[section]
\newtheorem{proposition}[thrm]{\scshape Предложение}
\newtheorem{lemma}[thrm]{\scshape Лемма}
\newtheorem{corollary}[thrm]{\scshape Следствие}
\newtheorem{remark}[thrm]{\scshape Замечание}
\newtheorem{condition}[thrm]{\scshape Условие}

\theoremstyle{definition}
\newtheorem{example}[thrm]{\scshape Пример}

\newtheoremstyle{break}        
{}{}%
{\itshape}{}%
{\bfseries}{}
{\newline}{}
\theoremstyle{break}

\theoremstyle{break}

\providecommand{\keywords}[1]{\textbf{{Ключевые слова:}} #1.}

\title{Операторные оценки погрешности при усреднении нестационарного уравнения Шрёдингера:\\зависимость от времени\footnote{Работа выполнена при поддержке гранта РНФ \textnumero 17-11-01069.}}
\date{}
\author{Дородный~М.~А.\footnote{Санкт-Петербургский государственный университет, Россия, 199034, Санкт-Петербург, Университетская~наб., д.~7/9; e-mail: \texttt{mdorodni@yandex.ru}.}}

\begin{document}
	
\clubpenalty = 10000
\widowpenalty = 10000

\hyphenation{мат-ри-чно-знач-ная}
\hyphenation{пе-рио-ди-чес-кие}
\hyphenation{пе-рио-ди-чес-кую}
\hyphenation{пе-рио-ди-чес-кое}
   
\maketitle

\begin{abstract}
	\noindent
	В $L_2 (\mathbb{R}^d; \mathbb{C}^n)$ рассматривается самосопряжённый матричный эллиптический дифференциальный оператор $\mathcal{A}_\varepsilon$ второго порядка с периодическими коэффициентами, зависящими от $\mathbf{x}/\varepsilon$. Для операторной экспоненты $e^{-i \tau \mathcal{A}_\varepsilon}$, $\tau \in \mathbb{R}$, при малом $\varepsilon$ получены аппроксимации по ($H^s \to L_2$)-операторной норме при подходящем $s$. Обсуждается точность полученных оценок в зависимости от $\tau$. Результаты применяются к вопросу о поведении решения $\mathbf{u}_\varepsilon$ задачи Коши для нестационарного уравнения типа Шрёдингера $i \partial_{\tau} \mathbf{u}_\varepsilon =  \mathcal{A}_\varepsilon \mathbf{u}_\varepsilon + \mathbf{F}$.
\end{abstract}


\part*{Введение}
Работа относится к теории усреднения (гомогенизации) периодических дифференциальных операторов (ДО). Задачам усреднения в пределе малого периода посвящена обширная литература; укажем в первую очередь книги~\cite{BeLP,BaPa,ZhKO}. Один из методов изучения задач усреднения в $\mathbb{R}^d$~--- это спектральный метод, основанный на теории Флоке--Блоха. См., например,~\cite[глава~4]{BeLP}, \cite[глава~2]{ZhKO}, \cite{Se1981}, \cite{COrVa}, \cite{APi}.

\subsection{Класс операторов}
Пусть $\Gamma$~--- решётка в $\mathbb{R}^d$, $\Omega$~--- элементарная ячейка решётки~$\Gamma$. Для $\Gamma$-периодических функций в $\mathbb{R}^d$ используется обозначение $\varphi^{\varepsilon}(\mathbf{x}) \coloneqq \varphi(\varepsilon^{-1}\mathbf{x})$, $\varepsilon>0$. В $L_2 (\mathbb{R}^d; \mathbb{C}^n)$ рассматриваются самосопряжённые эллиптические матричные ДО второго порядка следующего вида
\begin{equation}
\label{intro_A_eps}
\mathcal{A}_{\varepsilon} = (f^{\varepsilon}(\mathbf{x}))^* b(\mathbf{D})^* g^{\varepsilon}(\mathbf{x}) b(\mathbf{D}) f^{\varepsilon}(\mathbf{x}).
\end{equation}
Здесь $b(\mathbf{D})$~--- однородный матричный ДО первого порядка с постоянными коэффициентами. Предполагается, что символ  $b(\boldsymbol{\xi})$~--- ($m \times n$)-матрица ранга $n$ (считаем, что $m \ge n$). Далее, $g(\mathbf{x})$~--- $\Gamma$-периодическая ограниченная и положительно определённая ($m \times m$)-матрица-функция, $f(\mathbf{x})$~--- $\Gamma$-периодическая ограниченная вместе со своей обратной ($n \times n$)-матрица-функция.

Целесообразно первоначально изучать более узкий класс операторов вида
\begin{equation}
\label{intro_hatA_eps}
\widehat{\mathcal{A}}_{\varepsilon} = b(\mathbf{D})^* g^{\varepsilon}(\mathbf{x}) b(\mathbf{D}),
\end{equation}
отвечающий случаю $f = \mathbf{1}_n$. Многие операторы математической физики допускают запись в виде~(\ref{intro_A_eps}) или~(\ref{intro_hatA_eps}), см., например,~\cite[гл.~4]{BSu2005-2}. Простейший пример~--- оператор акустики $\widehat{\mathcal{A}}_{\varepsilon} = \mathbf{D}^* g^{\varepsilon}(\mathbf{x}) \mathbf{D} = -\operatorname{div} g^{\varepsilon}(\mathbf{x}) \nabla$.

\subsection{Обзор} 
В 2001 году М.~Ш.~Бирманом и Т.~А.~Суслиной (см.~\cite{BSu2001}) был предложен и развит теоретико-операторный подход к задачам усреднения в $\mathbb{R}^d$ (вариант спектрального метода), основанный на масштабном преобразовании, теории Флоке--Блоха и аналитической теории возмущений. 
С помощью этого подхода были получены так называемые~\emph{операторные оценки погрешности} в задачах усреднения.

В случае усреднения эллиптических и параболических задач этот подход  достаточно подробно разработан: укажем, в частности, работы~\cite{BSu2003, BSu2005, BSu2005-2, Su2004, Su2007, VSu2011, VSu2012}.

Другой подход  к получению операторных оценок погрешности (\textquotedblleft метод сдвига\textquotedblright) для эллиптических и параболических задач  был предложен В.~В.~Жиковым и С.~Е.~Пастуховой в работах \cite{Zh2006, ZhPas2005, ZhPas2006}.  См. также обзор~\cite{ZhPas2016}.

Операторные оценки погрешности для нестационарного уравнения типа Шрёдингера и гиперболических уравнений изучены в меньшей степени. Им были посвящены работы~\cite{BSu2008, Su2016, Su2017, DSu2018, M2017, M2019a}, а также~\cite{D2019} и \cite{M2019b}, где рассматривался более широкий класс операторов, включающих младшие члены. В операторных терминах речь идёт о поведении при малом $\varepsilon$ оператор-функций $e^{-i \tau \widehat{\mathcal{A}}_{\varepsilon}}$ и $\cos(\tau \widehat{\mathcal{A}}_{\varepsilon}^{1/2})$, $\widehat{\mathcal{A}}_{\varepsilon}^{-1/2} \sin(\tau \widehat{\mathcal{A}}_{\varepsilon}^{1/2})$,  где $\tau \in \mathbb{R}$. Остановимся подробнее на результатах для нестационарного уравнения типа Шрёдингера.
В~\cite{BSu2008} была доказана оценка 
\begin{equation}
\label{intro_exp_est_1}
\| e^{-i \tau \widehat{\mathcal{A}}_{\varepsilon}} - e^{-i \tau \widehat{\mathcal{A}}^0} \|_{H^3 (\mathbb{R}^d) \to L_2 (\mathbb{R}^d)} \le C (1 + |\tau|)\varepsilon.
\end{equation}
Здесь $\widehat{\mathcal{A}}^0 = b(\mathbf{D})^* g^0 b(\mathbf{D})$~--- \emph{эффективный оператор} c постоянной \emph{эффективной матрицей} $g^0$. Затем в работе~\cite{Su2017} (см. также~\cite{Su2016}) была подтверждена точность этой оценки относительно типа операторной нормы. С другой стороны, были найдены достаточные условия (которые формулируются в терминах спектральных характеристик оператора на краю спектра), позволяющие усилить результат и получить оценку
\begin{equation}
\label{intro_exp_est_2}
\| e^{-i \tau \widehat{\mathcal{A}}_{\varepsilon}} - e^{-i \tau \widehat{\mathcal{A}}^0} \|_{H^{2} (\mathbb{R}^d) \to L_2 (\mathbb{R}^d)} \le \mathcal{C}(1+|\tau|) \varepsilon.
\end{equation}

\subsection{Основные результаты работы}
Настоящая работа посвящена оценкам погрешности при аппроксимации операторной экспоненты; особое внимание уделяется зависимости этих оценок от времени. Мы показываем, что множитель $(1+ |\tau|)$ в оценке~(\ref{intro_exp_est_1}) нельзя улучшить (заменить на $(1 + |\tau|^\alpha)$ с $\alpha <1$) в общей ситуации. С другой стороны, мы доказываем, что оценку~(\ref{intro_exp_est_2}) (справедливую при дополнительных предположениях) можно улучшить:
\begin{equation*}
\| e^{-i \tau \widehat{\mathcal{A}}_{\varepsilon}} - e^{-i \tau \widehat{\mathcal{A}}^0} \|_{H^{2} (\mathbb{R}^d) \to L_2 (\mathbb{R}^d)} \le \check{\mathcal{C}}(1+|\tau|^{1/2}) \varepsilon.
\end{equation*}
Этот результат позволяет получать сходимость решений с квалифицированной оценкой погрешности при больших временах, а именно, при $\tau = O(\varepsilon^{-\alpha})$ с $\alpha < 2$.
Аналоги этих результатов получены и для более общего оператора~(\ref{intro_A_eps}). При этом оказывается, что удобно изучать
оператор  $f^\varepsilon e^{-i \tau \mathcal{A}_{\varepsilon}} (f^\varepsilon)^{-1}$ (операторную экспоненту, окаймлённую быстро осциллирующими множителями). 

Результаты, полученные в операторных терминах, применяются затем к вопросу о поведении решения $\mathbf{u}_\varepsilon (\mathbf{x}, \tau)$, $\mathbf{x} \in \mathbb{R}^d$, $\tau \in \mathbb{R}$, следующей задачи
\begin{equation}
\label{intro_Ahat_Cauchy_problem}
\left\{
\begin{aligned}
&i \frac{\partial \mathbf{u}_\varepsilon (\mathbf{x}, \tau)}{\partial \tau} =  (\widehat{\mathcal{A}}_{\varepsilon} \mathbf{u}_\varepsilon) (\mathbf{x}, \tau) + \mathbf{F} (\mathbf{x}, \tau), \\
& \mathbf{u}_\varepsilon (\mathbf{x}, 0) = \boldsymbol{\phi} (\mathbf{x}), 
\end{aligned}
\right.
\end{equation}
а также более общей задачи с оператором $\mathcal{A}_{\varepsilon}$.

\subsection{Метод}
Результаты получены с помощью теоретико-операторного подхода. Масштабное преобразование сводит изучение разности экспонент под знаком нормы~(\ref{intro_exp_est_1}) к изучению разности $e^{-i \tau \varepsilon^{-2} \widehat{\mathcal{A}}} - e^{-i \tau \varepsilon^{-2} \widehat{\mathcal{A}}^0}$, где $\widehat{\mathcal{A}} = b(\mathbf{D})^* g(\mathbf{x}) b(\mathbf{D})$. Затем с помощью унитарного преобразования Гельфанда оператор~$\widehat{\mathcal{A}}$ раскладывается в прямой интеграл по зависящим от квазиимпульса $\mathbf{k}$ операторам $\widehat{\mathcal{A}}(\mathbf{k})$, действующим в пространстве $L_2 (\Omega; \mathbb{C}^n)$. Следуя~\cite{BSu2003}, мы выделяем одномерный параметр $t = |\mathbf{k}|$ и рассматриваем семейство $\widehat{\mathcal{A}}(\mathbf{k})$ как квадратичный пучок по отношению к параметру $t$. При этом часть построений удаётся провести в рамках абстрактной теории операторов.  В абстрактной схеме изучается действующее в некотором гильбертовом пространстве $\mathfrak{H}$ операторное семейство $A(t)$, допускающее факторизацию вида $A(t) = X(t)^* X(t)$, где $X(t) = X_0 + t X_1$.

\subsection{Структура статьи}
Работа состоит из трёх глав. В гл.~\ref{abstr_part} (\S\S1--3) содержится необходимый абстрактный теоретико-операторный материал. В гл.~\ref{L2_operators_part} (\S\S4--9) изучаются периодические ДО, действующие в $L_2(\mathbb{R}^d;\mathbb{C}^n)$. Гл.~\ref{main_results_part} (\S\S10--11) посвящена задачам усреднения для нестационарного уравнения типа Шрёдингера. В~\S\ref{main_results_exp_section} получены основные результаты работы в операторных терминах. Затем в~\S\ref{main_results_Cauchy_section} эти результаты применяются к усреднению для задачи Коши~(\ref{intro_Ahat_Cauchy_problem}), а также более общей задачи с оператором $\mathcal{A}_{\varepsilon}$. 

\subsection{Обозначения}
Пусть $\mathfrak{H}$ и $\mathfrak{H}_{*}$ --- комплексные сепарабельные гильбертовы пространства. Символы $(\cdot, \cdot)_{\mathfrak{H}}$ и $ \| \cdot \|_{\mathfrak{H}}$ означают скалярное произведение и норму в $\mathfrak{H}$. Символ $\| \cdot \|_{\mathfrak{H} \to \mathfrak{H}_*}$ означает норму ограниченного оператора из $\mathfrak{H}$ в $\mathfrak{H}_{*}$. Иногда мы опускаем индексы, если это не ведёт к смешениям. Через $I = I_{\mathfrak{H}}$ обозначается тождественный оператор в $\mathfrak{H}$. Если $A \colon \mathfrak{H} \to \mathfrak{H}_*$~--- линейный оператор, то через $\Dom A$ и $\Ker A$ обозначаются область определения и ядро $A$, соответственно. Если $\mathfrak{N}$~--- подпространство в $\mathfrak{H}$, то $\mathfrak{N}^{\perp} \coloneqq \mathfrak{H} \ominus \mathfrak{N}$. Если $P$~--- ортогональный проектор пространства $\mathfrak{H}$ на $\mathfrak{N}$, то $P^{\perp}$~--- ортогональный проектор $\mathfrak{H}$ на $\mathfrak{N}^{\perp}$. 

Символы $\left< \cdot, \cdot \right>$ и $| \cdot |$ означают стандартные скалярное произведение и норму в $\mathbb{C}^n$; $\boldsymbol{1}_n$~--- единичная ($n \times n$)-матрица. Для $(m \times n)$-матрицы $a$ символ  $a^*$ означает эрмитово сопряжённую матрицу.
Далее, $\mathbf{x} = (x_1, \ldots , x_d) \in \mathbb{R}^d$, $i D_j = \frac{\partial}{\partial x_j}$, $j = 1,\ldots, d$, $\mathbf{D} = -i \nabla = (D_1, \ldots, D_d)$. 

Классы $L_p$ функций cо значениями в $\mathbb{C}^n$, заданных в области $\mathcal{O} \subset \mathbb{R}^d$, обозначаются через $L_p (\mathcal{O}; \mathbb{C}^n)$, $1 \le p \le \infty$. Классы Соболева $\mathbb{C}^n$-значных функций в области $\mathcal{O} \subset \mathbb{R}^d$ порядка $s$ с индексом суммирования $p$ обозначаются $W^s_p (\mathcal{O}; \mathbb{C}^n)$. При $p=2$ используем обозначения $H^s (\mathcal{O}; \mathbb{C}^n)$, $s \in \mathbb{R}$. При $n=1$ пишем просто $L_p (\mathcal{O})$, $W^s_p (\mathcal{O})$, $H^s (\mathcal{O})$ и т.д., но иногда мы применяем такие обозначения и для пространств векторнозначных и матричнозначных функций.

Через $C$, $c$, $\mathcal{C}$, $\mathfrak{C}$ (возможно, с индексами и значками) обозначаются различные оценочные постоянные.

\subsection{Благодарности}
Автор выражает благодарность Т.~А.~Суслиной за полезные обсуждения и ценные советы.

\part{Абстрактная теоретико-операторная схема}
\label{abstr_part}

\section{Квадратичные операторные семейства}

\subsection{Операторы $X(t)$ и $A(t)$}  
\label{abstr_X_A_section}
Пусть $\mathfrak{H}$ и $\mathfrak{H}_{*}$~--- комплексные сепарабельные гильбертовы пространства. Предположим, что $X_{0} \colon \mathfrak{H} \to \mathfrak{H}_{*}$~--- плотно определённый и замкнутый оператор, а $X_{1} \colon \mathfrak{H} \to \mathfrak{H}_{*}$~--- ограниченный оператор. Введём замкнутый на $\Dom X_0$ оператор $X(t) \coloneqq X_0 + t X_1$, $t \in \mathbb{R}$. Рассмотрим семейство самосопряжённых операторов $A(t) \coloneqq X(t)^*  X(t)$ в $\mathfrak{H}$. Оператор $A(t)$ порождается замкнутой квадратичной формой $\| X(t) u \|^{2}_{\mathfrak{H}_*}$, $u \in \Dom X_0$. Введём обозначения $A_0 \coloneqq A(0)$, $\mathfrak{N} \coloneqq \Ker  A_0 = \Ker X_0$, $\mathfrak{N}_{*} \coloneqq \Ker X^*_0$. Предполагается выполненным следующее условие.
\begin{condition}
	Точка $\lambda_0 = 0$~--- изолированная точка спектра оператора $A_0$ и $0 < n \coloneqq \dim \mathfrak{N} < \infty$, $n \le n_* \coloneqq \dim  \mathfrak{N}_* \le \infty$.
\end{condition}

Обозначим через $d^0$ \emph{расстояние от точки $\lambda_0 = 0$ до остального спектра оператора $A_0$}. Пусть $P$ и $P_*$~--- ортопроекторы в $\mathfrak{H}$ на $\mathfrak{N}$ и в  $\mathfrak{H}_*$ на  $\mathfrak{N}_*$, соответственно. Обозначим через $F(t;[a, b])$ спектральный проектор оператора $A(t)$ для промежутка $[a,b]$ и положим $\mathfrak{F} (t;[a,b]) \coloneqq F(t;[a, b]) \mathfrak{H}$. \emph{Фиксируем число $\delta > 0$ такое, что $8 \delta < d^0$}. Будем писать $F(t)$ вместо $F(t; [0,\delta])$ и $\mathfrak{F} (t)$ вместо $\mathfrak{F}(t; [0, \delta])$. Выберем число $t^0 > 0$ так, чтобы
\begin{equation}
\label{abstr_t0}
t^0 \le \delta^{1/2} \|X_1\|^{-1}.
\end{equation}
Как показано в~\cite[гл.~1, предложение~1.2]{BSu2003}, $F(t; [0,\delta]) = F(t;[0, 3 \delta])$ и $\rank F(t; [0,\delta]) = n$ при $|t| \le t^0$.

\subsection{Операторы $Z$, $R$ и $S$}
Следуя~\cite[гл.~1, \S1]{BSu2003} и~\cite[\S1]{BSu2005}, введём операторы, которые возникают при рассмотрениях в духе теории возмущений.

Пусть $\omega \in \mathfrak{N}$ и пусть $\psi = \psi(\omega) \in \Dom X_0 \cap \mathfrak{N}^{\perp}$~--- (слабое) решение уравнения 
\begin{equation*}
X^*_0 (X_0 \psi + X_1 \omega) = 0.
\end{equation*}
Введём оператор $Z \colon \mathfrak{H} \to \mathfrak{H}$ по формуле $Zu = \psi (P u)$, $u \in \mathfrak{H}$. Далее, определим оператор $R \coloneqq X_0 Z + X_1 \colon \mathfrak{N}  \to \mathfrak{N}_*$. Оператор $R$ также может быть определён формулой $R= P_*X_1 |_{\mathfrak{N}}$. Следуя \cite[гл.~1, п.~1.3]{BSu2003}, назовём оператор $S \coloneqq R^* R \colon \mathfrak{N} \to \mathfrak{N}$ \emph{спектральным ростком} семейства $A(t)$ при $t=0$. Для ростка справедливо также представление $S = P X^*_1 P_* X_1 |_{\mathfrak{N}}$. Спектральный росток называется \emph{невырожденным}, если $\Ker S = \{0\}$.

\subsection{Операторы $Z_2$ и $R_2$}
\label{abstr_Z2_R2_section}
Нам потребуется ввести операторы $Z_2$ и $R_2$, определённые в~\cite[гл.~1, \S1]{VSu2011}.

Пусть $\omega \in \mathfrak{N}$ и пусть $\phi = \phi(\omega) \in \Dom X_0 \cap \mathfrak{N}^{\perp}$~--- (слабое) решение уравнения 
\begin{equation*}
X^*_0 (X_0 \phi + X_1 Z \omega) = -P^{\perp} X_1^* R \omega.
\end{equation*}
Правая часть этого уравнения принадлежит $\mathfrak{N}^{\perp} = \Ran X_0^*$, поэтому выполнено условие разрешимости.
Определим оператор $Z_2 \colon \mathfrak{H} \to \mathfrak{H}$ по формуле $Z_2 u = \phi (P u)$, $u \in \mathfrak{H}$. Наконец, введём оператор $R_2 \coloneqq X_0 Z_2 + X_1 Z \colon \mathfrak{N}  \to \mathfrak{H}_*$.

\subsection{Аналитические ветви собственных значений и собственных векторов оператора $A(t)$} Согласно общей аналитической теории возмущений (см.~\cite{K}), при $|t| \le t^0$ существуют вещественно аналитические функции $\lambda_l (t)$ (ветви собственных значений) и вещественно аналитические $\mathfrak{H}$-значные функции $\varphi_l (t)$ (ветви собственных векторов), такие что $A(t) \varphi_l(t) = \lambda_l (t) \varphi_l(t)$, $l = 1, \ldots, n$, причём набор $\varphi_l (t)$, $l = 1, \ldots, n$, образует \emph{ортонормированный базис} в $\mathfrak{F}(t)$. Более того, для \emph{достаточно малого} $t_*$ (где $0 < t_* \le t^0$) при $|t| \le t_*$ имеют место сходящиеся степенные разложения
\begin{align}
\label{abstr_A(t)_eigenvalues_series}
\lambda_l(t) &= \gamma_l t^2 + \mu_l t^3 + \nu_l t^4 + \ldots, & &\gamma_l \ge 0, \; \mu_l, \nu_l \in \mathbb{R}, \qquad  l = 1, \ldots, n, \\ 
\label{abstr_A(t)_eigenvectors_series}
\varphi_l (t) &= \omega_l + t \psi_l^{(1)} + t \psi_l^{(2)} + \ldots, & &l = 1, \ldots, n.
\end{align}
При этом элементы $\omega_l =  \varphi_l (0), \, l = 1, \ldots, n,$ образуют ортонормированный базис в подпространстве $\mathfrak{N}$.
В~\cite[гл.~1, \S1]{BSu2003} и~\cite[\S1]{BSu2005} было установлено, что $\widetilde{\omega}_l = \psi_l^{(1)} - Z \omega_l \in \mathfrak{N}$, 
\begin{alignat}{2}
\label{abstr_S_eigenvectors}
S \omega_l = \gamma_l \omega_l&, \qquad  &&l = 1, \ldots, n,\\
\label{abstr_tilde_omega_rel}
(\widetilde{\omega}_j, \omega_k ) + (\omega_j, \widetilde{\omega}_k) = 0&, \qquad &&j, k = 1,\ldots, n.
\end{alignat}
Таким образом, \emph{числа $\gamma_l$ и элементы $\omega_l$, определённые в~\emph{(\ref{abstr_A(t)_eigenvalues_series})} и~\emph{(\ref{abstr_A(t)_eigenvectors_series})}, являются собственными для ростка $S$}. Справедливы представления $P = \sum_{l=1}^{n} (\cdot, \omega_l) \omega_l$, $SP = \sum_{l=1}^{n} \gamma_l (\cdot, \omega_l) \omega_l$.

\subsection{Пороговые аппроксимации}
Следующие утверждения были получены в~\cite[гл.~1, теоремы~4.1 и~4.3]{BSu2003} и~\cite[теорема~4.1]{BSu2005}. Договоримся ниже через $\beta_j$ обозначать абсолютные константы (значения которых допускают явный контроль), причём считаем $\beta_j \ge 1$.
\begin{thrm}[\cite{BSu2003}] 
	В условиях п.~\emph{\ref{abstr_X_A_section}} при $|t| \le t^0$ справедливы оценки
	\begin{align}
	\label{abstr_F(t)_threshold}
	\| F(t) - P \| &\le C_1 |t|, & C_1 &= \beta_1 \delta^{-1/2} \| X_1 \|,\\
	\notag
	\| A(t)F(t) - t^2 SP \| &\le C_2 |t|^3, & C_2 &= \beta_2 \delta^{-1/2}\| X_1 \|^3.
	\end{align}
\end{thrm}
\begin{thrm}[\cite{BSu2005}]
	\label{abstr_threshold_approx_thrm_2} 
	В условиях п.~\emph{\ref{abstr_X_A_section}} при $|t| \le t^0$ справедливо представление
	\begin{equation*}
	A(t) F(t) = t^2 SP + t^3 K + \Xi (t), \qquad \| \Xi (t) \| \le C_3 t^4, \qquad C_3 = \beta_3 \delta^{-1}\| X_1 \|^4.
	\end{equation*}
	Здесь оператор $K$ допускает представление $K = K_0 + N = K_0 + N_0 + N_*$, где $K_0$ переводит $\mathfrak{N}$ в $\mathfrak{N}^{\perp}$ и $\mathfrak{N}^{\perp}$ в $\mathfrak{N}$, а $N = N_0 + N_*$ переводит $\mathfrak{N}$ в себя и $\mathfrak{N}^{\perp}$ в $\{ 0 \}$. В терминах коэффициентов степенных разложений операторы $K_0$, $N_0$, $N_*$ имеют вид $K_0 = \sum_{l=1}^{n} \gamma_l \left( (\cdot, Z \omega_l) \omega_l + (\cdot, \omega_l) Z \omega_l \right)$, 
	\begin{equation}
	\label{abstr_N0_N*}
	N_0 = \sum_{l=1}^{n} \mu_l (\cdot, \omega_l) \omega_l, \qquad N_* = \sum_{l=1}^{n} \gamma_l \left( (\cdot, \widetilde{\omega}_l) \omega_l + (\cdot, \omega_l) \widetilde{\omega}_l\right).
	\end{equation}
	В инвариантных терминах справедливы представления $K_0 = Z S P + S P Z^*$, $
	N = Z^*X_1^* R P + (RP)^* X_1 Z$.
\end{thrm}

\begin{remark}
	\label{abstr_N_remark}	
	\begin{enumerate*}[label=\emph{\arabic*$^{\circ}.$}, ref=\arabic*$^{\circ}$]
		\item Если $Z = 0$, то $K_0 = 0$, $N = 0$ и $K = 0$.
		\item В базисе $\{\omega_l\}_{l=1}^n$ операторы $N$, $N_0$, $N_*$ \emph{(}суженные на подпространство $\mathfrak{N}$\emph{)} задаются матрицами размера $n \times n$. При этом оператор $N_0$ диагонален $(N_0 \omega_j, \omega_k ) = \mu_j \delta_{jk}$, $j, k = 1, \ldots ,n$. Матричные элементы оператора $N_*$ имеют вид $(N_* \omega_j, \omega_k) = \gamma_k (\omega_j, \widetilde{\omega}_k) + \gamma_j (\widetilde{\omega}_j, \omega_k ) = ( \gamma_j - \gamma_k)(\widetilde{\omega}_j, \omega_k )$, $j, k = 1,\ldots, n$.	Видно, что диагональные элементы для $N_*$ обращаются в ноль. Более того, $(N_* \omega_j, \omega_k) = 0$ если  $\gamma_j = \gamma_k$.
		\item Если $n = 1$, то $N_* = 0$ и $N = N_0$.
	\end{enumerate*}
\end{remark}

\subsection{Условие невырожденности}
Ниже мы предполагаем выполненным следующее дополнительное условие.
\begin{condition}
	\label{abstr_nondegeneracy_cond}
	Существует константа $c_* > 0$ такая, что $A(t) \ge c_* t^2 I$ при $|t| \le t^0$.
\end{condition}
Из условия~\ref{abstr_nondegeneracy_cond} следует, что $\lambda_l (t) \ge c_* t^2$, $l = 1, \ldots, n$, при $|t| \le t^0$. В силу~(\ref{abstr_A(t)_eigenvalues_series}) это влечёт $\gamma_l \ge c_* > 0$, $l= 1, \ldots, n$, т.~е. невырожденность спектрального ростка:
\begin{equation}
\label{abstr_S_nondegeneracy}
S \ge c_* I_{\mathfrak{N}}.
\end{equation}

\subsection{Разбиение собственных значений оператора $A(t)$ на кластеры}
\label{abstr_cluster_section}

Материал этого пункта заимствован из~\cite[раздел~2]{Su2017}. Он содержателен при $n \ge 2$.

Предположим, что выполнено условие~\ref{abstr_nondegeneracy_cond}. Сейчас нам будет удобно изменить обозначения, отслеживая кратности собственных значений оператора $S$. Пусть $p$~--- число различных собственных значений ростка $S$. Будем считать, что они занумерованы в порядке возрастания и обозначим их через $\gamma^{\circ}_j$, $j = 1, \ldots, p$. Пусть $k_1, \ldots, k_p$~--- их кратности (разумеется, $k_1 + \ldots + k_p = n$).
Введём обозначения для собственных подпространств: $\mathfrak{N}_j = \Ker(S - \gamma^{\circ}_j I_\mathfrak{N})$, $j = 1 ,\ldots, p$. Тогда $\mathfrak{N} = \sum_{j=1}^{p} \oplus \mathfrak{N}_j$.
Пусть $P_j$~--- ортопроектор пространства $\mathfrak{H}$ на $\mathfrak{N}_j$. Тогда $P = \sum_{j=1}^{p} P_j$, $P_j P_l = 0$ при $j \ne l$.

Мы разбиваем первые $n$ собственных значений оператора $A(t)$ на $p$ кластеров при $|t| \le t^0$; $j$-ый кластер состоит из собственных значений $\lambda_l (t)$, $l= i,\ldots, i+k_j-1$, где $i = i(j) = k_1+\ldots+k_{j-1}+1$. 

Для каждой пары индексов $(j, l)$, $1 \le j,l \le p$, $j \ne l$, введём обозначение
\begin{equation*}
c^{\circ}_{jl} \coloneqq \min \{c_*, n^{-1} |\gamma^{\circ}_l - \gamma^{\circ}_j|\}.
\end{equation*}
Ясно, что найдётся номер $i_0 = i_0 (j,l)$, где $j \le i_0 \le l-1$ при $j < l$ и $l \le i_0 \le j-1$ при $l < j$, такой, что $\gamma^{\circ}_{i_0 + 1} -  \gamma^{\circ}_{i_0} \ge c^{\circ}_{jl}$. Это означает, что на промежутке между $\gamma^{\circ}_j$ и $\gamma^{\circ}_l$ в спектре оператора $S$ имеется лакуна длины не меньше $c^{\circ}_{jl}$. Возможно, выбор $i_0$ неоднозначен, в этом случае договоримся брать наименьшее возможное $i_0$ (для определённости). Далее, выберем число $t^{00}_{jl} \le t^0$ так, чтобы 
\begin{equation*}
t^{00}_{jl} \le (4C_2)^{-1} c^{\circ}_{jl} = (4 \beta_2)^{-1} \delta^{1/2} \|X_1\|^{-3 } c^{\circ}_{jl}.
\end{equation*}
Пусть $\Delta^{(1)}_{jl} \coloneqq [\gamma^{\circ}_1 - c^{\circ}_{jl}/4, \gamma^{\circ}_{i_0} + c^{\circ}_{jl}/4]$ и $\Delta^{(2)}_{jl} \coloneqq [\gamma^{\circ}_{i_0+1} - c^{\circ}_{jl}/4, \gamma^{\circ}_p + c^{\circ}_{jl}/4]$. Отрезки $\Delta^{(1)}_{jl}$ и $\Delta^{(2)}_{jl}$ не пересекаются и отделены друг от друга на расстояние, не меньшее $c^{\circ}_{jl}/2$. Как показано в~\cite[раздел~2]{Su2017}, при $|t| \le  t^{00}_{jl}$ у оператора $A(t)$ в промежутке $t^2 \Delta^{(1)}_{jl}$ ровно $k_1 + \ldots + k_{i_0}$ собственных значений (с учётом кратностей) и ровно $k_{i_0+1} + \ldots + k_p$ собственных значений в промежутке $t^2 \Delta^{(2)}_{jl}$.
\begin{remark}
	Операторы $N_{0}$ и $N_{*}$ допускают инвариантное представление: 
	\begin{equation}
	\label{abstr_N0_N*_invar_repr}
	N_{0} = \sum_{j=1}^{p} P_j N P_j, \qquad N_{*} = \sum_{\substack{1 \le l,j \le p \\ j \ne l}} P_l N P_j.
	\end{equation}
\end{remark}

\subsection{Коэффициенты $\nu_l$, $l=1, \ldots,n$}
\label{abstr_nu_section}
Нам потребуется указать связь коэффициентов $\nu_l$, $l=1, \ldots,n$, с некоторой задачей на собственные числа. 

В~\cite[(1.34), (1.37)]{VSu2011} было установлено, что
\begin{align}
\notag
\psi_l^{(2)} - Z \widetilde{\omega}_l - Z_2 \omega_l \eqqcolon \widetilde{\omega}^{(2)}_l \in \mathfrak{N}&, & &l = 1, \ldots, n,\\
\label{abstr_tilde_omega^(2)_rel}
(\widetilde{\omega}^{(2)}_l, \omega_k) + (Z \omega_l, Z \omega_k) + (\widetilde{\omega}_l, \widetilde{\omega}_k) + (\omega_l, \widetilde{\omega}^{(2)}_k) = 0&, & &l, k = 1, \ldots, n.
\end{align}  
Далее, из~\cite[(2.47), формула после~(2.46)]{VSu2011} получаем
\begin{multline}
\label{abstr_nu_rel_1}
(N_1 \omega_l, \omega_k) - \mu_l (\widetilde{\omega}_l, \omega_k) - \mu_k (\omega_l, \widetilde{\omega}_k) - \gamma_l (\widetilde{\omega}^{(2)}_l, \omega_k) - \gamma_k (\omega_l, \widetilde{\omega}^{(2)}_k) - (S \widetilde{\omega}_l, \widetilde{\omega}_k) = \nu_l \delta_{lk},\\
l, k = 1, \ldots, n,
\end{multline}
где $N_1 = N_1^0 - Z^*Z S P - S P Z^*Z$, $N_1^0 = Z_2^* X_1^* R P + (RP)^* X_1 Z_2 + R_2^* R_2 P$. 

Пусть $\gamma_q^{\circ}$~--- собственное число задачи~(\ref{abstr_S_eigenvectors}) кратности $k_q$ (т.~е. $\gamma_i = \ldots = \gamma_{i+k_q-1}$, $i = i(q) = k_1+\ldots+k_{q-1}+1$). Рассмотрим задачу на собственные значения (см. замечание~\ref{abstr_N_remark}) 
\begin{equation}
\label{abstr_N_eigenvalues}
P_q N \omega_l = \mu_l \omega_l, \qquad l = i, \ldots, i+k_q-1.
\end{equation}
Будем считать, что $\mu_l$, $l = i, \ldots, i+k_q-1$, занумерованы в порядке неубывания. Обозначим количество различных собственных значений через $p'(q)$ и обозначим их кратности через $k_{1,q}, \ldots, k_{p'(q),q}$ (разумеется, $k_{1,q} + \ldots + k_{p'(q),q} = k_q$). Переобозначим различные собственные значения через $\mu^{\circ}_{j,q}$, $j = 1, \ldots, p'(q)$ и введём следующие обозначения для собственных подпространств: $\mathfrak{N}_{j,q} = \Ker(P_qN|_{\mathfrak{N}_q} - \mu^{\circ}_{j,q} I_{\mathfrak{N}_q})$, $j = 1 ,\ldots, p'(q)$. Тогда $\mathfrak{N}_q = \sum_{j=1}^{p'(q)} \oplus \mathfrak{N}_{j,q}$. Пусть $P_{j,q}$~--- ортопроектор пространства $\mathfrak{H}$ на $\mathfrak{N}_{j,q}$. Тогда $P_q = \sum_{j=1}^{p'(q)} P_{j,q}$ и $P_{j,q} P_{r,q} = 0$  при $j \ne r$.  

Пусть $\mu^{\circ}_{q',q}$~--- $k_{q',q}$-кратное собственное значение задачи~(\ref{abstr_N_eigenvalues}): $\mu_{i'} = \ldots = \mu_{i'+k_{q',q}-1}$, где $i' = i'(q',q) =i(q)+k_{1,q}+\ldots+k_{q'-1,q}$. Используя соотношения~(\ref{abstr_tilde_omega_rel}), (\ref{abstr_tilde_omega^(2)_rel}) и учитывая, что $\gamma_l = \gamma_k = \gamma^{\circ}_q$, $\mu_l = \mu_k = \mu^{\circ}_{q',q}$, $l,k = i', \ldots, i'+k_{q',q}-1$, из~(\ref{abstr_nu_rel_1}) получаем
\begin{equation}
\label{abstr_nu_rel_2}
(N_1 \omega_l, \omega_k) +  \gamma_l (Z \omega_l, Z \omega_k) + \gamma_l (\widetilde{\omega}_l, \widetilde{\omega}_k) - (S \widetilde{\omega}_l, \widetilde{\omega}_k) = \nu_l \delta_{lk}, \qquad l,k = i', \ldots, i'+k_{q',q}-1.
\end{equation}
Далее, в силу замечания~\ref{abstr_N_remark} имеем
\begin{multline*}
\gamma_l (\widetilde{\omega}_l, \widetilde{\omega}_k) - (S \widetilde{\omega}_l, \widetilde{\omega}_k) =  \sum_{l'=1}^{n} (\gamma_l - \gamma_{l'}) (\widetilde{\omega}_l,\omega_{l'}) (\omega_{l'}, \widetilde{\omega}_k) =\\= \sum_{\substack{l'\in\{1,\ldots,n\}\\l' \ne i, \ldots, i+k_q-1}} \frac{(N \omega_{l}, \omega_{l'}) ( \omega_{l'}, N\omega_{k})}{\gamma^{\circ}_q - \gamma_{l'}}   =  \sum_{\substack{j\in\{1,\ldots,p\} \\	j \ne q}} \frac{( P_{j} N\omega_l, N\omega_k )}{\gamma^{\circ}_q - \gamma^{\circ}_{j}}  \eqqcolon \mathfrak{n}_0^{(q',q)}[\omega_l,\omega_k],\\
l,k = i', \ldots, i'+k_{q',q}-1.
\end{multline*}
Уравнения~(\ref{abstr_nu_rel_2}) можно трактовать как задачу на собственные числа для оператора $\mathcal{N}^{(q',q)}$: 
\begin{equation}
\label{abstr_srcN^q_eigenvalues}
\mathcal{N}^{(q',q)} \omega_l = \nu_l \omega_l, \qquad l = i', \ldots, i'+k_{q',q}-1,
\end{equation}
 где 
\begin{equation*}
\mathcal{N}^{(q',q)} \coloneqq P_{q',q} \left. \left( N_1^0 - \frac{1}{2} Z^*Z S P - \frac{1}{2} S P Z^*Z \right)\right|_{\mathfrak{N}_{q',q}} + \mathcal{N}^{(q',q)}_0,
\end{equation*}
а $\mathcal{N}^{(q',q)}_0$~--- оператор в $\mathfrak{N}_{q',q}$, порождённый формой $\mathfrak{n}_0^{(q',q)}[\cdot,\cdot]$.

Отметим, что в случае, когда $N_0 = 0$ (что равносильно $\mu_l = 0$ для всех $l=1, \ldots,n$) выполнено $\mathfrak{N}_{1,q} = \mathfrak{N}_{q}$, $q=1,\ldots,p$. Тогда вместо $\mathcal{N}^{(1,q)}$ мы будем писать $\mathcal{N}^{(q)}$. 

\section{Аппроксимация оператора $e^{-i \tau \varepsilon^{-2} A(t)}P$}
\label{abstr_exp_section}

\subsection{Аппроксимация оператора $e^{-i \tau \varepsilon^{-2} A(t)}P$}
Пусть $\varepsilon > 0$. Исследуем поведение оператора $e^{-i \tau \varepsilon^{-2} A ( t)}$ при малом $\varepsilon$. Домножим этот оператор на \textquotedblleft сглаживающий множитель\textquotedblright \ $\varepsilon^s (t^2 + \varepsilon^2)^{-s/2}P$, где $s > 0$. (Термин объясняется тем, что в приложениях к ДО такое домножение переходит в сглаживание.) Наша цель~--- получить аппроксимацию сглаженного оператора с оценкой погрешности порядка $O (\varepsilon)$ при минимально возможном $s$.

Следующие утверждения были доказаны в~\cite[теорема~2.1]{BSu2008} и~\cite[следствия~3.3, 3.5]{Su2017}.
\begin{thrm}[\cite{BSu2008}]
	\label{abstr_exp_general_thrm_wo_eps}
	При $\tau \in \mathbb{R}$ и $|t| \le t^0$ справедлива оценка
	\begin{equation}
	\label{abstr_exp_general_est_wo_eps}
	\| e^{-i \tau A(t)} P - e^{-i \tau t^2 SP} P \| \le  2C_1 |t| + C_2 |\tau| |t|^3.
	\end{equation}
\end{thrm}
\begin{thrm}[\cite{Su2017}]
	\label{abstr_exp_enchcd_thrm_wo_eps_1}
	Пусть $N=0$. Тогда при $\tau \in \mathbb{R}$ и $|t| \le t^0$ справедлива оценка
	\begin{equation} 
	\label{abstr_exp_enchcd_est_wo_eps_1}     
	\| e^{-i \tau A(t)} P - e^{-i \tau t^2 SP} P \| \le 2C_1 |t| + C_{4} | \tau | t^4. 
	\end{equation}
	Константа $C_4$ полиномиально зависит от величин $\delta^{-1/2}$, $\|X_1\|$.
\end{thrm}
\begin{thrm}[\cite{Su2017}]
	\label{abstr_exp_enchcd_thrm_wo_eps_2}
	Пусть $N_0 = 0$. Тогда при $\tau \in \mathbb{R}$ и $|t| \le t^{00}$ справедлива оценка
	\begin{equation*}     
	\|  e^{-i \tau A(t)} P - e^{-i \tau t^2 SP} P \| \le C_{5} |t| + C_{6} | \tau | t^4. 
	\end{equation*}
	Здесь $t^{00}$ подчинено условию 
	\begin{equation}
	\label{abstr_t00}
	t^{00} \le  (4 \beta_2)^{-1} \delta^{1/2} \|X_1\|^{-3} c^{\circ},
	\end{equation}
	где 
	\begin{equation}
	\label{abstr_c^circ}
	c^{\circ} \coloneqq \min_{(j,l) \in \mathcal{Z}} c^{\circ}_{jl}, \qquad \mathcal{Z} \coloneqq \{(j,l) \colon 1 \le j,l \le p, j \ne l, P_jNP_l \ne 0\}.
	\end{equation} 
	Константы $C_5$, $C_6$ полиномиально зависят от  $\delta^{-1/2}$, $\|X_1\|$, $n$, $(c^{\circ})^{-1}$.
\end{thrm}

Применим сформулированные теоремы. Начнём с теоремы~\ref{abstr_exp_general_thrm_wo_eps}. Пусть $|t| \le t^0$. В силу~(\ref{abstr_exp_general_est_wo_eps}) (с заменой $\tau$ на $\varepsilon^{-2} \tau$) 
\begin{multline*}
\bigl\| e^{-i \tau \varepsilon^{-2} A (t)} P - e^{-i \tau \varepsilon^{-2} t^2 SP} P \bigr\| \varepsilon^3 (t^2 + \varepsilon^2)^{-3/2} \le \\ \le (2C_{1}|t| + C_{2} \varepsilon^{-2}  |\tau| |t|^3) \varepsilon^3 (t^2 + \varepsilon^2)^{-3/2} \le (C_{1} + C_{2}|\tau|) \varepsilon.
\end{multline*}
Здесь пришлось взять $s = 3$. Мы приходим к следующему результату, полученному в~\cite[теорема~2.6]{BSu2008}.
\begin{thrm}[\cite{BSu2008}]
	\label{abstr_exp_general_thrm}
	При $\tau \in \mathbb{R}$ и $|t| \le t^0$ справедлива оценка
	\begin{equation*}
	\bigl\| e^{-i \tau \varepsilon^{-2} A (t)} P - e^{-i \tau \varepsilon^{-2} t^2 SP} P \bigr\| \varepsilon^3 (t^2 + \varepsilon^2)^{-3/2} \le  (C_1 + C_2 |\tau|) \varepsilon.
	\end{equation*}
	Постоянные $C_1$, $C_2$ контролируются через многочлены от величин $\delta^{-1/2}$, $\|X_1\|$.
\end{thrm}

Теорема~\ref{abstr_exp_enchcd_thrm_wo_eps_1} позволяет усилить результат теоремы~\ref{abstr_exp_general_thrm} в случае, когда $N = 0$.
\begin{thrm}
	\label{abstr_exp_enchcd_thrm_1}
	Пусть $N = 0$. Тогда при $\tau \in \mathbb{R}$ и $|t| \le t^0$ справедлива оценка
	\begin{equation} 
	\label{abstr_exp_enchcd_est_1}   
	\bigl\| e^{-i \tau \varepsilon^{-2} A (t)} P - e^{-i \tau \varepsilon^{-2} t^2 S P}P \bigr\| \varepsilon^{2} (t^2 + \varepsilon^2)^{-1} \le  (C_1 + C'_4 |\tau|^{1/2}) \varepsilon. 
	\end{equation}    
	Постоянные $C_1$, $C'_4$ контролируются через многочлены от величин $\delta^{-1/2}$, $\|X_1\|$.
\end{thrm}
\begin{proof}
	Заметим, что при $|t| \ge \varepsilon^{1/2}/|\tau|^{1/4}$ выполнено 
	\begin{equation*}
	\frac{\varepsilon^{2}}{t^2 + \varepsilon^2} \le
	\frac{\varepsilon^{2}}{\left(\tfrac{\varepsilon^{1/2}}{|\tau|^{1/4}}\right)^2 + \varepsilon^2} = \frac{\varepsilon |\tau|^{1/2}}{1 + \varepsilon |\tau|^{1/2} } \le \varepsilon |\tau|^{1/2},
	\end{equation*}
	поэтому левая часть в~(\ref{abstr_exp_enchcd_est_1}) не превосходит $2 |\tau|^{1/2} \varepsilon$.
	Таким образом, достаточно считать, что $|t| < \varepsilon^{1/2}/|\tau|^{1/4}$. Воспользуемся неравеством~(\ref{abstr_exp_enchcd_est_wo_eps_1}) с заменой $\tau$ на $\varepsilon^{-2} \tau$. Тогда при $|t| < \varepsilon^{1/2}/|\tau|^{1/4}$ получаем оценку
	\begin{multline*}
	\bigl\| e^{-i \tau \varepsilon^{-2} A(t)} P - e^{-i \tau \varepsilon^{-2} t^2 S P} P\bigr\| \varepsilon^2 (t^2 + \varepsilon^2)^{-1} \le  \left( 2 C_1 |t| + C_4 \varepsilon^{-2} |\tau| t^4 \right) \varepsilon^{2} (t^2 + \varepsilon^2)^{-1} \le \\ \le C_1 \varepsilon + C_4 |\tau| t^2 \le C_1 \varepsilon + C_4 |\tau|^{1/2} \varepsilon.
	\end{multline*}
	В результате получаем оценку~(\ref{abstr_exp_enchcd_est_1}) с постоянной $C'_4= \max\{2,C_4\}$.
\end{proof}

Аналогично, применение теоремы~\ref{abstr_exp_enchcd_thrm_wo_eps_2} позволяет получить следующий результат.
\begin{thrm}
	\label{abstr_exp_enchcd_thrm_2}
	Пусть $N_0 = 0$. Тогда при $\tau \in \mathbb{R}$ и $|t| \le t^{00}$ справедлива оценка
	\begin{equation*} 
	\bigl\| e^{-i \tau \varepsilon^{-2} A (t)} P - e^{-i \tau \varepsilon^{-2} t^2 S P}P \bigr\| \varepsilon^{2} (t^2 + \varepsilon^2)^{-1} \le  (C_5 + C'_6 |\tau|^{1/2}) \varepsilon. 
	\end{equation*}    
	Здесь $t^{00}$ подчинено условию~\emph{(\ref{abstr_t00})}, константы $C_5$, $C'_6$ контролируются через многочлены от величин $\delta^{-1/2}$, $\|X_1\|$, $n$, $(c^{\circ})^{-1}$.
\end{thrm}

\begin{remark}
	Теоремы~\emph{\ref{abstr_exp_enchcd_thrm_1}} и~\emph{\ref{abstr_exp_enchcd_thrm_2}} усиливают результаты теорем~\emph{4.2} и~\emph{4.3} из~\emph{\cite{Su2017}} в отношении зависимости оценок от $\tau$. 
\end{remark}

\subsection{Точность результатов относительно сглаживающего множителя}
Покажем, что полученные результаты точны относительно сглаживающего множителя. Следующая теорема, доказанная в~\cite[теорема~4.4]{Su2017}, подтверждает точность теоремы~\ref{abstr_exp_general_thrm}.
\begin{thrm}[\cite{Su2017}]
	\label{abstr_exp_smooth_shrp_thrm_1}
	Пусть $N_0 \ne 0$. Пусть $\tau \ne 0$ и $0 \le s < 3$. Тогда не существует такой константы $C(\tau) >0$, чтобы оценка
	\begin{equation*}
	\bigl\| e^{-i \tau \varepsilon^{-2} A (t)} P - e^{-i \tau \varepsilon^{-2} t^2 S P} P \bigr\| \varepsilon^s (t^2 + \varepsilon^2)^{-s/2} \le  C(\tau) \varepsilon
	\end{equation*}
	выполнялась для всех достаточно малых $|t|$ и $\varepsilon$.
\end{thrm}
Далее, подтвердим точность теорем~\ref{abstr_exp_enchcd_thrm_1}, \ref{abstr_exp_enchcd_thrm_2}.
\begin{thrm}
	\label{abstr_exp_smooth_shrp_thrm_2}
	Пусть $N_0 = 0$ и $\mathcal{N}^{(q)} \ne 0$ для некоторого $q \in \{1, \ldots, p\}$. Пусть $\tau \ne 0$ и $0 \le s < 2$. Тогда не существует такой константы $C(\tau) >0$, чтобы оценка
	\begin{equation}
	\label{abstr_exp_smooth_shrp_est_2}
	\bigl\| e^{-i \tau \varepsilon^{-2} A (t)} P - e^{-i \tau \varepsilon^{-2} t^2 S P} P \bigr\| \varepsilon^s (t^2 + \varepsilon^2)^{-s/2} \le  C(\tau) \varepsilon
	\end{equation}
	выполнялась для всех достаточно малых $|t|$ и $\varepsilon$.
\end{thrm}
\begin{proof}
	Начнём с предварительных замечаний. Поскольку $F(t)^{\perp}P = (P - F(t))P$, то из~(\ref{abstr_F(t)_threshold}) вытекает оценка
	\begin{equation}
	\label{abstr_smooth_shrp_Fperp_est}
	\| e^{-i \tau \varepsilon^{-2} A (t)} F(t)^{\perp} P \| \varepsilon (t^2 + \varepsilon^2)^{-1/2} \le C_1 |t| \varepsilon (t^2 + \varepsilon^2)^{-1/2} \le C_1 \varepsilon, \qquad |t| \le t^0.
	\end{equation}
	Далее, при $ |t| \le t^0$ имеем:
	\begin{equation}
	\label{abstr_smooth_shrp_f1}
	e^{-i \tau \varepsilon^{-2} A (t)} F(t) = \sum_{l=1}^{n} e^{-i \tau \varepsilon^{-2} \lambda_l(t)} (\cdot,\varphi_l(t)) \varphi_l(t).
	\end{equation}
	Затем, из сходимости рядов~(\ref{abstr_A(t)_eigenvectors_series}) следует, что
	\begin{equation}
	\label{abstr_smooth_shrp_f2}
	\| \varphi_l(t) - \omega_l \| \le c_1 |t|, \qquad |t| \le t_*, \quad l = 1,\ldots,n.
	\end{equation}
		
	Фиксируем $0 \ne \tau \in \mathbb{R}$. Будем рассуждать от противного. Предположим, что для некоторого $1 \le s < 2$ существует константа $C(\tau) > 0$ такая, что выполнено~(\ref{abstr_exp_smooth_shrp_est_2}) при всех достаточно малых $|t|$ и $\varepsilon$. В силу~(\ref{abstr_smooth_shrp_Fperp_est})--(\ref{abstr_smooth_shrp_f2}) это предположение равносильно существованию положительной константы $\widetilde{C} (\tau)$ такой, что  выполнено
	\begin{equation}
	\label{abstr_smooth_shrp_f3}
	\left\| \sum_{l=1}^{n} \left( e^{-i \tau \varepsilon^{-2} \lambda_l(t)} - e^{-i \tau \varepsilon^{-2} t^2 \gamma_l} \right) (\cdot,\omega_l) \omega_l  \right\| \varepsilon^{s} (t^2 + \varepsilon^2)^{-s/2} \le \widetilde{C} (\tau) \varepsilon
	\end{equation}
	при всех достаточно малых $|t|$ и $\varepsilon$.
	
	Условие $N_0 = 0$ и $\mathcal{N}^{(q)} \ne 0$ для некоторого $q \in \{1, \ldots, p\}$ означает, что в разложениях~(\ref{abstr_A(t)_eigenvalues_series}) $\mu_l = 0$ при всех $l=1, \ldots,n$, и что хотя бы для одного $j$ выполнено $\nu_j \ne 0$. Применим оператор под знаком нормы в~(\ref{abstr_smooth_shrp_f3}) к элементу $\omega_j$. Тогда
	\begin{equation}
	\label{abstr_smooth_shrp_f4}
	\left| e^{-i \tau \varepsilon^{-2} \lambda_j(t)} - e^{-i \tau \varepsilon^{-2} t^2 \gamma_j} \right| \varepsilon^{s} (t^2 + \varepsilon^2)^{-s/2} \le \widetilde{C} (\tau) \varepsilon
	\end{equation}
	при всех достаточно малых $|t|$ и $\varepsilon$. Левую часть~(\ref{abstr_smooth_shrp_f4}) c учётом~(\ref{abstr_A(t)_eigenvalues_series}) можно записать в виде $
	2 \left| \sin \left( \frac{1}{2} \tau \varepsilon^{-2} (\nu_j t^4 + O(t^5)) \right) \right| \varepsilon^{s} (t^2 + \varepsilon^2)^{-s/2}$.	Будем считать, что $t_*$ достаточно мало, так что $
	\frac{1}{2} |\nu_j| |t|^4 \le |\nu_j t^4 + O(t^5)| \le \frac{3}{2} |\nu_j| |t|^4$, $|t| \le t_*$. Далее, для фиксированного $\tau \ne 0$, предполагая, что $\varepsilon$ достаточно мало (а именно, $\varepsilon \le \pi^{-1/2} |\nu_j\tau|^{1/2} t_*^2$), положим $t = t(\varepsilon) = \pi^{1/4} |\nu_j \tau|^{-1/4} \varepsilon^{1/2} = c \varepsilon^{1/2}$. Для таких $t$ справедливо $2 \left| \sin \left( \frac{1}{2} \tau \varepsilon^{-2} (\nu_j t^4 + O(t^5)) \right) \right|  \ge \sqrt{2}$, поэтому из~(\ref{abstr_smooth_shrp_f4}) следует, что $\sqrt{2} \varepsilon^{s} (c^2 \varepsilon + \varepsilon^2)^{-s/2} \le \widetilde{C} (\tau) \varepsilon$.
	Это означает, что функция $\varepsilon^{s/2-1} (c^2 + \varepsilon)^{-s/2}$ равномерно ограничена при малых $\varepsilon$. Но это неверно, если $s < 2$. Полученное противоречие завершает доказательство.
\end{proof}

\subsection{Точность результатов относительно времени}
Теперь мы докажем следующее утверждение, подтверждающее точность теоремы~\ref{abstr_exp_general_thrm} относительно времени.

\begin{thrm}
	\label{abstr_exp_time_shrp_thrm_1}
	Пусть $N_0 \ne 0$. Тогда не существует положительной функции $C(\tau)$ такой, что $\lim_{\tau \to \infty} C(\tau)/ |\tau| = 0$ и выполнена оценка
	\begin{equation}
	\label{abstr_exp_time_shrp_est_1}
	\bigl\| e^{-i \tau \varepsilon^{-2} A (t)} P - e^{-i \tau \varepsilon^{-2} t^2 S P} P \bigr\| \varepsilon^3 (t^2 + \varepsilon^2)^{-3/2} \le  C(\tau) \varepsilon
	\end{equation}
	при всех $\tau \in \mathbb{R}$ и всех достаточно малых $|t|$ и $\varepsilon > 0$.
\end{thrm}
\begin{proof}
	Будем рассуждать от противного. Предположим, что существует положительная функция $C(\tau)$ такая, что $\lim_{\tau \to \infty} C(\tau)/ |\tau| = 0$ и выполнено~(\ref{abstr_exp_time_shrp_est_1}) при всех достаточно малых $|t|$ и $\varepsilon$. В силу~(\ref{abstr_smooth_shrp_Fperp_est})--(\ref{abstr_smooth_shrp_f2}) это предположение равносильно существованию положительной функции $\widetilde{C} (\tau)$ такой, что $\lim_{\tau \to \infty} \widetilde{C}(\tau)/ |\tau| = 0$ и выполнено
	\begin{equation}
	\label{abstr_time_shrp_f1}
	\left\| \sum_{l=1}^{n} \left( e^{-i \tau \varepsilon^{-2} \lambda_l(t)} - e^{-i \tau \varepsilon^{-2} t^2 \gamma_l} \right) (\cdot,\omega_l) \omega_l  \right\| \varepsilon^{3} (t^2 + \varepsilon^2)^{-3/2} \le \widetilde{C} (\tau) \varepsilon
	\end{equation}
	при всех достаточно малых $|t|$ и $\varepsilon$.
	
	Условие $N_0 \ne 0$ означает, что хотя бы для одного $j$ выполнено $\mu_j \ne 0$. Применим оператор под знаком нормы в~(\ref{abstr_time_shrp_f1}) к элементу $\omega_j$. Тогда
	\begin{equation}
	\label{abstr_time_shrp_f2}
	\left| e^{-i \tau \varepsilon^{-2} \lambda_j(t)} - e^{-i \tau \varepsilon^{-2} t^2 \gamma_j} \right| \varepsilon^{3} (t^2 + \varepsilon^2)^{-3/2} \le \widetilde{C} (\tau) \varepsilon
	\end{equation}
	при всех достаточно малых $|t|$ и $\varepsilon$. Модуль в левой части~(\ref{abstr_time_shrp_f2}) c учётом~(\ref{abstr_A(t)_eigenvalues_series}) может быть записан следующим образом: 
	\begin{equation*}
	\left| e^{-i \tau \varepsilon^{-2} \lambda_j(t)} - e^{-i \tau \varepsilon^{-2} t^2 \gamma_j} \right| = 2 \left| \sin \left( \frac{1}{2} \tau \varepsilon^{-2} (\lambda_j(t) - t^2 \gamma_j) \right) \right| = 2 \left| \sin \left( \frac{1}{2} \tau \varepsilon^{-2} (\mu_j t^3 + O(t^4)) \right) \right|.
	\end{equation*}
	Поэтому 
	\begin{equation}
	\label{abstr_time_shrp_f3}
	2 \left| \sin \left( \frac{1}{2} \tau \varepsilon^{-2} (\mu_j t^3 + O(t^4)) \right) \right| \varepsilon^{3} (t^2 + \varepsilon^2)^{-3/2} \le \widetilde{C} (\tau) \varepsilon
	\end{equation}
	при всех достаточно малых $|t|$ и $\varepsilon$.
	Мы хотим оценить снизу максимум левой части~(\ref{abstr_time_shrp_f3}) (как функции от $t$). Введём функцию 
	\begin{equation*}
	h(y) = \left\{
	\begin{aligned}
	&\frac{2}{\pi} y, & &\text{при} \, y \in [0, \pi/2],\\
	&\frac{2}{\pi} (\pi - y), & &\text{при} \, y \in (\pi/2, \pi],
	\end{aligned}
	\right. 
	\end{equation*}
	(и продолжим её $\pi$-периодическим образом). Очевидно, $|\sin y| \ge h (y)$. Будем считать, что $t_*$ достаточно мало, так что
	\begin{equation}
	\label{abstr_time_shrp_f4}
	\frac{1}{2} |\mu_j| |t|^3 \le |\mu_j t^3 + O(t^4)| \le \frac{3}{2} |\mu_j| |t|^3, \qquad |t| \le t_*.
	\end{equation}
	Далее, пусть $\varepsilon$ достаточно мало, а именно, $\varepsilon \le (2 \pi)^{-1/2} |\mu_j \tau|^{1/2} t_*^{3/2}$.
	Тогда множество  $T = \left\{t \colon \frac{1}{2} |\tau| \varepsilon^{-2} |t^3\mu_j + O(t^4)| \le \frac{\pi}{2} \right\}$ содержится в интервале $|t| \le t_*$. Левая часть~(\ref{abstr_time_shrp_f3}) при $t \in T$ допускает оценку снизу
	\begin{multline*}
	2 \left| \sin \left( \frac{1}{2} \tau \varepsilon^{-2} (\mu_j t^3 + O(t^4)) \right) \right| \varepsilon^{3} (t^2 + \varepsilon^2)^{-3/2} \ge \\ \ge 2 h \left( \frac{1}{2} \tau \varepsilon^{-2} (\mu_j t^3 + O(t^4)) \right) \varepsilon^{3} (t^2 + \varepsilon^2)^{-3/2} \ge \varepsilon^{3} (t^2 + \varepsilon^2)^{-3/2} \cdot \frac{1}{\pi} |\tau| \varepsilon^{-2} |\mu_j| |t|^3.
	\end{multline*}
	Положим $t_\diamond = (2 \pi / 3)^{1/3} |\mu_j \tau|^{-1/3} \varepsilon^{2/3}$. Тогда $\frac{3}{4} \tau \varepsilon^{-2} |\mu_j| t_\diamond^3 \le \frac{\pi}{2}$, а поэтому $t_\diamond \in T$ в силу~(\ref{abstr_time_shrp_f4}). Максимум левой части~(\ref{abstr_time_shrp_f3}) оценивается снизу через
	\begin{equation*}
	\varepsilon^{3} (t^2 + \varepsilon^2)^{-3/2} \cdot \frac{1}{\pi} |\tau| \varepsilon^{-2} |\mu_j| t_\diamond^3 = \frac{2}{3} \cdot \frac{\varepsilon |\tau|}{\bigl((2\pi/3)^{2/3}|\mu_j|^{-2/3} + \varepsilon^{2/3} |\tau|^{2/3}\bigr)^{3/2}} \ge \frac{\sqrt{2}}{3} \cdot \frac{\varepsilon |\tau|}{2\pi/(3 |\mu_j|) + \varepsilon |\tau|}.
	\end{equation*}
	Это означает, что 
	\begin{equation}
	\label{abstr_time_shrp_f5}
	\frac{\sqrt{2}}{3} \left( \frac{2\pi}{3|\mu_j|} + \varepsilon |\tau| \right)^{-1} \le \frac{\widetilde{C}(\tau)}{|\tau|}
	\end{equation}
	при всех достаточно малых $\varepsilon > 0$. Но, поскольку $\lim_{\tau \to \infty} \widetilde{C}(\tau)/ |\tau| = 0$,   оценка~(\ref{abstr_time_shrp_f5}) не может выполняться при больших $|\tau|$ и $\varepsilon = O(|\tau|^{-1})$. Полученное противоречие завершает доказательство.
\end{proof}
Аналогично доказывается следующее утверждение, подтверждающее точность теорем~\ref{abstr_exp_enchcd_thrm_1}, \ref{abstr_exp_enchcd_thrm_2}.
\begin{thrm}
	\label{abstr_exp_time_shrp_thrm_2}
	Пусть $N_0 = 0$ и $\mathcal{N}^{(q)} \ne 0$ для некоторого $q \in \{1, \ldots, p\}$. Тогда не существует положительной функции $C(\tau)$ такой, что $\lim_{\tau \to \infty} C(\tau)/ |\tau|^{1/2} = 0$ и выполнена оценка
	\begin{equation*}
	\bigl\| e^{-i \tau \varepsilon^{-2} A (t)} P - e^{-i \tau \varepsilon^{-2} t^2 S P} P \bigr\| \varepsilon^2 (t^2 + \varepsilon^2)^{-1} \le  C(\tau) \varepsilon
	\end{equation*}
	при всех $\tau \in \mathbb{R}$ и всех достаточно малых $|t|$ и $\varepsilon > 0$.
\end{thrm}

\section{Аппроксимация окаймлённой операторной экспоненты}
\label{abstr_sndw_section}
\subsection{Операторное семейство вида $A(t) = M^* \widehat{A} (t) M$}
\label{abstr_A_and_Ahat_section}
Наряду с пространством $\mathfrak{H}$ рассмотрим ещё одно сепарабельное гильбертово пространство $\widehat{\mathfrak{H}}$. Пусть $\widehat{X} (t) = \widehat{X}_0 + t \widehat{X}_1 \colon \widehat{\mathfrak{H}} \to \mathfrak{H}_* $~--- семейство операторов того же вида, что и $X(t)$, причём для $\widehat{X} (t)$ выполнены предположения п.~\ref{abstr_X_A_section}. Пусть $M \colon \mathfrak{H} \to \widehat{\mathfrak{H}}$~--- изоморфизм. Предположим, что $M \Dom X_0 = \Dom \widehat{X}_0$, $X(t) = \widehat{X} (t) M$, а тогда и $X_0 = \widehat{X}_0 M$, $X_1 = \widehat{X}_1 M$. В $\widehat{\mathfrak{H}}$ введём семейство самосопряжённых операторов $\widehat{A} (t) = \widehat{X} (t)^* \widehat{X} (t)$. Тогда, очевидно,
\begin{equation}
\label{abstr_A_Ahat}
A(t) = M^* \widehat{A} (t) M.
\end{equation} 
Все объекты, отвечающие семейству $\widehat{A}(t)$, далее помечаются значком \textquotedblleft$\widehat{\phantom{m}} $\textquotedblright. Отметим, что $\widehat{\mathfrak{N}} = M \mathfrak{N}$ и $\widehat{\mathfrak{N}}_* =  \mathfrak{N}_*$. В пространстве $\widehat{\mathfrak{H}}$ рассмотрим положительно определённый оператор $Q \coloneqq (M M^*)^{-1} \colon \widehat{\mathfrak{H}} \to \widehat{\mathfrak{H}}$. Пусть $Q_{\widehat{\mathfrak{N}}}$~--- блок оператора $Q$ в подпространстве $\widehat{\mathfrak{N}}$, т.~е.
$Q_{\widehat{\mathfrak{N}}} = \widehat{P} Q|_{\widehat{\mathfrak{N}}} \colon \widehat{\mathfrak{N}} \to \widehat{\mathfrak{N}}$. Очевидно, $Q_{\widehat{\mathfrak{N}}}$~--- изоморфизм в $\widehat{\mathfrak{N}}$.

Как показано в \cite[предложение~1.2]{Su2007}, ортопроектор $P$ в $\mathfrak{H}$ на $\mathfrak{N}$ и ортопроектор $\widehat{P}$ в $\widehat{\mathfrak{H}}$ на $\widehat{\mathfrak{N}}$ связаны соотношением 
\begin{equation}
\label{abstr_P_Phat}
P = M^{-1} (Q_{\widehat{\mathfrak{N}}})^{-1} \widehat{P} (M^*)^{-1}.
\end{equation}
Пусть $\widehat{S} \colon \widehat{\mathfrak{N}} \to \widehat{\mathfrak{N}}$~--- спектральный росток семейства $\widehat{A} (t)$ при $t = 0$,
а $S$~--- росток семейства $A (t)$. В \cite[гл.~1, п.~1.5]{BSu2003} установлено следующее тождество:
\begin{equation}
\label{abstr_S_Shat}
S = P M^* \widehat{S} M |_\mathfrak{N}.
\end{equation}

\subsection{Операторы $\widehat{Z}_Q$ и $\widehat{N}_Q$}
\label{abstr_hatZ_Q_and_hatN_Q_section}
Для операторного семейства $\widehat{A} (t)$ введём оператор $\widehat{Z}_Q$, действующий в $\widehat{\mathfrak{H}}$ и сопоставляющий элементу $\widehat{u} \in \widehat{\mathfrak{H}}$ решение  $\widehat{\psi}_Q$ задачи $\widehat{X}^*_0 (\widehat{X}_0 \widehat{\psi}_Q + \widehat{X}_1 \widehat{\omega}) = 0$, $Q \widehat{\psi}_Q \perp \widehat{\mathfrak{N}}$, где $\widehat{\omega} = \widehat{P} \widehat{u}$. Как показано в~\cite[\S6]{BSu2005}, оператор $Z$ для семейства $A(t)$ и введённый оператор $\widehat{Z}_Q$ связаны соотношением 
\begin{equation}
\label{abstr_Z_hatZ_Q}
\widehat{Z}_Q =M Z M^{-1} \widehat{P}.
\end{equation}
Введём оператор $\widehat{N}_Q \coloneqq \widehat{Z}_Q^* \widehat{X}_1^* \widehat{R} \widehat{P} + (\widehat{R} \widehat{P})^* \widehat{X}_1  \widehat{Z}_Q$. Согласно~\cite[\S6]{BSu2005}, оператор $N$ для семейства $A(t)$ и введённый оператор $\widehat{N}_Q$ связаны соотношением
\begin{equation}
\label{abstr_N_hatN_Q}
\widehat{N}_Q = \widehat{P} (M^*)^{-1} N M^{-1} \widehat{P}.
\end{equation}
Поскольку $N = N_0 + N_*$, то $\widehat{N}_Q = \widehat{N}_{0,Q} + \widehat{N}_{*,Q}$, где
\begin{equation}
\label{abstr_N0*_hatN0*_Q}
\widehat{N}_{0,Q} = \widehat{P} (M^*)^{-1} N_0 M^{-1} \widehat{P}, \qquad \widehat{N}_{*,Q} = \widehat{P} (M^*)^{-1} N_* M^{-1} \widehat{P}.
\end{equation}

Справедлива следующая лемма, доказанная в~\cite[лемма~5.1]{Su2017}.
\begin{lemma}[\cite{Su2017}]
	\label{abstr_N_hatNQ_lemma}
	Условие $N = 0$ равносильно равенству $\widehat{N}_Q = 0$. Условие $N_0 = 0$ равносильно равенству $\widehat{N}_{0,Q} = 0$.
\end{lemma}

\subsection{Операторы $\widehat{Z}_{2,Q}$, $\widehat{R}_{2,Q}$ и $\widehat{N}_{1,Q}^0$}
\label{abstr_hatZ2_Q_hatR2_Q_N1^0_Q_section}
Пусть $\widehat{u} \in \widehat{\mathfrak{H}}$ и пусть $\widehat{\phi}_Q = \widehat{\phi}_Q(\widehat{u}) \in \Dom \widehat{X}_0$~--- (слабое) решение уравнения 
\begin{equation*}
\widehat{X}^*_0 (\widehat{X}_0 \widehat{\phi}_Q + \widehat{X}_1 \widehat{Z}_Q \widehat{\omega}) = -\widehat{X}_1^* \widehat{R} \widehat{\omega} + Q (Q_{\widehat{\mathfrak{N}}})^{-1} \widehat{P} \widehat{X}_1^* \widehat{R} \widehat{\omega}, \qquad Q \widehat{\phi}_Q \perp \widehat{\mathfrak{N}},
\end{equation*}
где $\widehat{\omega} = \widehat{P} \widehat{u}$. Ясно, что правая часть этого уравнения принадлежит $\widehat{\mathfrak{N}}^{\perp} = \Ran \widehat{X}_0^*$, поэтому условие разрешимости выполнено.
Определим оператор $\widehat{Z}_{2,Q} \colon \widehat{\mathfrak{H}} \to \widehat{\mathfrak{H}}$ по формуле $\widehat{Z}_{2,Q} \widehat{u} = \widehat{\phi}_Q(\widehat{u})$. 

Теперь, введём оператор $\widehat{R}_{2,Q} \colon \widehat{\mathfrak{N}} \to \mathfrak{H}_*$ по формуле
$\widehat{R}_{2,Q} = \widehat{X}_0 \widehat{Z}_{2,Q}  + \widehat{X}_1 \widehat{Z}_Q $.
Наконец, определим оператор $\widehat{N}_{1,Q}^0$:
\begin{equation*}
\widehat{N}_{1,Q}^0 = \widehat{Z}_{2,Q}^* \widehat{X}_1^* \widehat{R} \widehat{P} + (\widehat{R} \widehat{P})^* \widehat{X}_1 \widehat{Z}_{2,Q} + \widehat{R}_{2,Q}^* \widehat{R}_{2,Q} \widehat{P}.
\end{equation*}
В~\cite[\S6, п.~6.3]{VSu2011} были установлены следующие соотношения:
\begin{gather}
\notag
\widehat{Z}_{2,Q} =M Z_2 M^{-1} \widehat{P}, \\
\notag
R_2 = \widehat{R}_{2,Q} M |_{\mathfrak{N}}, \quad \widehat{R}_{2,Q} = R_2 M^{-1} |_{\widehat{\mathfrak{N}}},\\
\label{abstr_N1^0_N1Q^0_hat}
\widehat{N}_{1,Q}^0 = \widehat{P} (M^*)^{-1} N_1^0 M^{-1} \widehat{P}.
\end{gather}

\subsection{Связь операторов и коэффициентов степенных разложений}
Укажем связь коэффициентов степенных разложений~(\ref{abstr_A(t)_eigenvalues_series}), (\ref{abstr_A(t)_eigenvectors_series}) и операторов $\widehat{S}$ и $Q_{\widehat{\mathfrak{N}}}$. (См.~\cite[п.~1.6,~1.7]{BSu2005}.) Положим $\zeta_l \coloneqq M \omega_l \in \widehat{\mathfrak{N}}$, $l = 1, \ldots, n$. Тогда из~(\ref{abstr_S_eigenvectors}) и~(\ref{abstr_P_Phat}), (\ref{abstr_S_Shat}) видно, что
\begin{equation}
\label{abstr_hatS_gener_spec_problem}
\widehat{S} \zeta_l  = \gamma_l Q_{\widehat{\mathfrak{N}}} \zeta_l, \qquad l = 1, \ldots, n. 
\end{equation}
Набор $\zeta_1, \ldots, \zeta_n$ образует базис в $\widehat{\mathfrak{N}}$, ортонормированный с весом $Q_{\widehat{\mathfrak{N}}}$: $(Q_{\widehat{\mathfrak{N}}} \zeta_l, \zeta_j) = \delta_{lj}$, $l,j = 1,\ldots,n$.

Операторы $\widehat{N}_{0,Q}$ и $\widehat{N}_{*,Q}$ можно описать в терминах коэффициентов степенных разложений~(\ref{abstr_A(t)_eigenvalues_series}) и~(\ref{abstr_A(t)_eigenvectors_series}); ср.~(\ref{abstr_N0_N*}). Положим $\widetilde{\zeta}_l \coloneqq M \widetilde{\omega}_l \in \widehat{\mathfrak{N}}$, $l = 1, \ldots,n$. Тогда
\begin{equation}
\label{abstr_hatN_0Q_N_*Q}
\widehat{N}_{0,Q} = \sum_{k=1}^{n} \mu_k (\cdot, Q_{\widehat{\mathfrak{N}}} \zeta_k) Q_{\widehat{\mathfrak{N}}} \zeta_k, \qquad \widehat{N}_{*,Q} = \sum_{k=1}^{n} \gamma_k \left( (\cdot, Q_{\widehat{\mathfrak{N}}} \widetilde{\zeta}_k) Q_{\widehat{\mathfrak{N}}} \zeta_k + (\cdot, Q_{\widehat{\mathfrak{N}}} \zeta_k) Q_{\widehat{\mathfrak{N}}} \widetilde{\zeta}_k \right). 
\end{equation}
Перейдём теперь к обозначениям, принятым в~п.~\ref{abstr_cluster_section}. Напомним, что различные собственные значения ростка $S$ обозначаются через $\gamma^{\circ}_q$, $q = 1,\ldots,p$, а соответствующие собственные подпространства через $\mathfrak{N}_q$. Набор векторов $\omega_l$, $l=i, \ldots, i+k_q-1$, где $i = i(q) = k_1+\ldots+k_{q-1}+1$, образует ортонормированный базис в $\mathfrak{N}_q$. Тогда те же числа $\gamma^{\circ}_q$, $q = 1,\ldots,p$~--- это различные собственные значения задачи~(\ref{abstr_hatS_gener_spec_problem}), а $M \mathfrak{N}_q$~--- соответствующие собственные подпространства. Векторы $\zeta_l = M \omega_l$, $l=i, \ldots, i+k_q-1$, образуют базис в $M \mathfrak{N}_q$, ортонормированный с весом $Q_{\widehat{\mathfrak{N}}}$. Через $\mathcal{P}_q$ обозначим \textquotedblleft косой\textquotedblright \ проектор на $M \mathfrak{N}_q$, ортогональный относительно скалярного произведения $(Q_{\widehat{\mathfrak{N}}} \cdot, \cdot)$, т.~е. $\mathcal{P}_q = \sum_{l=i}^{i+k_q-1} (\cdot, Q_{\widehat{\mathfrak{N}}} \zeta_l) \zeta_l$. Легко видеть, что $\mathcal{P}_q =M P_q M^{-1} \widehat{P}$.
Используя~(\ref{abstr_N0_N*_invar_repr}), (\ref{abstr_N_hatN_Q}) и~(\ref{abstr_N0*_hatN0*_Q}), нетрудно проверить равенства
\begin{equation}
\label{abstr_hatN_0Q_N_*Q_invar_repr}
\widehat{N}_{0,Q} = \sum_{j=1}^{p} \mathcal{P}_j^* \widehat{N}_Q \mathcal{P}_j, \qquad \widehat{N}_{*,Q} = \sum_{\substack{1 \le l,j \le p \\ j \ne l}} \mathcal{P}_l^* \widehat{N}_Q \mathcal{P}_j.
\end{equation}

Далее, можно указать связь между собственными числами и собственными векторами задачи~(\ref{abstr_N_eigenvalues}) и оператором $\widehat{N}_Q$. Пусть $\gamma^{\circ}_q$~--- $k_q$-кратное собственное значение задачи~(\ref{abstr_hatS_gener_spec_problem}). Тогда из~(\ref{abstr_N_hatN_Q}) и очевидного равенства $M P_q = \widehat{P}_{M \mathfrak{N}_q} M P_q$, где $\widehat{P}_{M \mathfrak{N}_{q}}$~--- ортопроектор на подпространство $M \mathfrak{N}_{q}$, видно, что
\begin{equation}
\label{abstr_hatN_Q_gener_spec_problem}
\widehat{P}_{M \mathfrak{N}_{q}} \widehat{N}_Q \zeta_l = \mu_l Q_{M \mathfrak{N}_q} \zeta_l, \qquad l = i(q), \ldots, i(q)+k_q-1,
\end{equation}
где $Q_{M \mathfrak{N}_q} = \widehat{P}_{M \mathfrak{N}_q} Q_{\widehat{\mathfrak{N}}}|_{M\mathfrak{N}_q}$. Напомним, что различные собственные значения задачи~(\ref{abstr_N_eigenvalues}) обозначаются через $\mu^{\circ}_{q',q}$, $q' = 1,\ldots,p'(q)$, а соответствующие собственные подпространства через $\mathfrak{N}_{q',q}$. Тогда те же числа $\mu^{\circ}_{q',q}$, $q' = 1,\ldots,p'(q)$~--- это различные собственные значения задачи~(\ref{abstr_hatN_Q_gener_spec_problem}), а $M \mathfrak{N}_{q',q}$~--- соответствующие собственные подпространства.

Наконец, свяжем собственные числа и собственные векторы задачи~(\ref{abstr_srcN^q_eigenvalues}) и оператор
\begin{equation*}
\widehat{\mathcal{N}}_Q^{(q',q)} =  \widehat{P}_{M \mathfrak{N}_{q',q}} \left. \left( \widehat{N}_{1,Q}^0 - \frac{1}{2} \widehat{Z}_Q^* Q \widehat{Z}_Q (MM^*) \widehat{S} \widehat{P} - \frac{1}{2} \widehat{S} (MM^*) \widehat{Z}_Q^* Q \widehat{Z}_Q \right)\right|_{M\mathfrak{N}_{q',q}} + \widehat{\mathcal{N}}^{(q',q)}_{0,Q},
\end{equation*}
где $\widehat{\mathcal{N}}^{(q',q)}_{0,Q}$~--- оператор в $M \mathfrak{N}_{q',q}$, порождённый формой 
\begin{equation*}
\widehat{\mathfrak{n}}_{0,Q}^{(q',q)}[\cdot,\cdot] =  \sum_{\substack{j\in\{1,\ldots,p\} \\ j \ne q}} \frac{( \widehat{P}_{M \mathfrak{N}_j} (M M^*) \widehat{P}_{M \mathfrak{N}_j} \widehat{N}_Q  \cdot,\widehat{N}_Q \cdot)}{\gamma^{\circ}_q - \gamma^{\circ}_j},
\end{equation*}
а $\widehat{P}_{M \mathfrak{N}_{q',q}}$~--- ортопроектор на подпространство $M \mathfrak{N}_{q',q}$. Из~(\ref{abstr_S_Shat}), (\ref{abstr_Z_hatZ_Q}), (\ref{abstr_N1^0_N1Q^0_hat}) и равенств $M P_j = \widehat{P}_{M\mathfrak{N}_j} M P_j$, $\widehat{P}_{M\mathfrak{N}_j} M (I - P_j) = 0$, $j=1,\ldots,p$, $M P_{q',q} = \widehat{P}_{M \mathfrak{N}_{q',q}} M P_{q',q}$, где $\widehat{P}_{M \mathfrak{N}_{q',q}}$~--- ортопроектор на подпространство $M \mathfrak{N}_{q',q}$,  видно, что
\begin{equation}
\label{abstr_scrNhat_M^(q)_gener_spec_problem}
\widehat{\mathcal{N}}_Q^{(q',q)} \zeta_l = \nu_l Q_{M \mathfrak{N}_{q',q}} \zeta_l, \qquad l = i', \ldots, i'+k_{q',q}-1.
\end{equation}
Здесь $i' = i'(q',q) = i(q)+k_{1,q}+\ldots+k_{q'-1,q}$ и $Q_{M \mathfrak{N}_{q',q}} = \widehat{P}_{M \mathfrak{N}_{q',q}} Q_{\widehat{\mathfrak{N}}}|_{M\mathfrak{N}_{q',q}}$.

\subsection{Аппроксимация окаймлённой операторной экспоненты}
\label{abstr_sndw_exp_section}
В этом пункте мы находим аппроксимацию для операторной экспоненты $e^{-i \tau \varepsilon^{-2} A(t)}$ семейства вида~(\ref{abstr_A_Ahat}) в терминах ростка $\widehat{S}$ оператора $\widehat{A}(t)$ и изоморфизма $M$. При этом оказывается удобным окаймить операторную экспоненту подходящими множителями.

Положим $M_0 \coloneqq (Q_{\widehat{\mathfrak{N}}})^{-1/2}$.  В~\cite[лемма~5.3]{Su2017} были доказаны следующие оценки
\begin{align}
\label{abstr_sndw_exp_est_1}
\| M e^{-i\tau A(t)} M^{-1} \widehat{P} -  M_0 e^{-i \tau t^2 M_0 \widehat{S} M_0} M_0^{-1} \widehat{P}\| \le  \| M \|^2 \| M^{-1} \|^2 \| e^{-i \tau A(t)} P -  e^{-i \tau t^2 S P} P \|,\\
\label{abstr_sndw_exp_est_2}
\| e^{-i \tau A(t)} P -  e^{-i \tau t^2 S P} P \| \le  \| M \|^2 \| M^{-1} \|^2 \| M e^{-i\tau A(t)} M^{-1} \widehat{P} -  M_0 e^{-i \tau t^2 M_0 \widehat{S} M_0} M_0^{-1} \widehat{P}\|. 
\end{align}
Из теорем~\ref{abstr_exp_general_thrm}, \ref{abstr_exp_enchcd_thrm_1}, \ref{abstr_exp_enchcd_thrm_2}, леммы~\ref{abstr_N_hatNQ_lemma} и неравенства~(\ref{abstr_sndw_exp_est_1}) непосредственно вытекают следующие результаты.
\begin{thrm}[\cite{BSu2008}]
	\label{abstr_sndw_exp_general_thrm}
	В предположениях п.~\emph{\ref{abstr_A_and_Ahat_section}} при $\tau \in \mathbb{R}$, $\varepsilon > 0$ и $|t| \le t^0$ выполнена оценка
	\begin{equation*}
	\|M e^{-i \tau \varepsilon^{-2} A(t)} M^{-1} \widehat{P} -  M_0 e^{-i \tau \varepsilon^{-2} t^2 M_0 \widehat{S} M_0} M_0^{-1} \widehat{P}\| \varepsilon^{3} (t^2 + \varepsilon^2)^{-3/2} \le \| M \|^2 \| M^{-1} \|^2 (C_1  + C_2 |\tau| ) \varepsilon.
	\end{equation*}
\end{thrm}

\begin{thrm}
	\label{abstr_sndw_exp_enchcd_thrm_1}
	Пусть выполнены предположения п.~\emph{\ref{abstr_A_and_Ahat_section}} и пусть $\widehat{N}_Q = 0$. Тогда при $\tau \in \mathbb{R}$, $\varepsilon > 0$ и $|t| \le t^0$ выполнена оценка
	\begin{equation*}
	\| M e^{-i \tau \varepsilon^{-2} A(t)} M^{-1} \widehat{P} -  M_0 e^{-i \tau \varepsilon^{-2} t^2 M_0 \widehat{S} M_0} M_0^{-1} \widehat{P}\| \varepsilon^{2} (t^2 + \varepsilon^2)^{-1} \le \| M \|^2 \| M^{-1} \|^2 ( C_1  + C'_4 | \tau |^{1/2} ) \varepsilon.
	\end{equation*}
\end{thrm}

\begin{thrm}
	\label{abstr_sndw_exp_enchcd_thrm_2}
	Пусть выполнены предположения пп.~\emph{\ref{abstr_A_and_Ahat_section}} и условие~\emph{\ref{abstr_nondegeneracy_cond}}. Пусть $\widehat{N}_{0,Q} = 0$. Тогда при $\tau \in \mathbb{R}$ и $|t| \le t^{00}$ выполнена оценка
	\begin{equation*}
	\| M e^{-i \tau \varepsilon^{-2} A(t)} M^{-1} \widehat{P} -  M_0 e^{-i \tau \varepsilon^{-2} t^2 M_0 \widehat{S} M_0} M_0^{-1} \widehat{P} \| \varepsilon^{2} (t^2 + \varepsilon^2)^{-1} \le \| M \|^2 \| M^{-1} \|^2 (C_5 + C'_6 | \tau |^{1/2} ) \varepsilon.
	\end{equation*}
\end{thrm}

Теорема~\ref{abstr_sndw_exp_general_thrm} была доказана в~\cite[теорема~3.2]{BSu2008}.
\begin{remark}
	Теоремы~\emph{\ref{abstr_sndw_exp_enchcd_thrm_1}} и~\emph{\ref{abstr_sndw_exp_enchcd_thrm_2}} усиливают результаты теорем~\emph{5.8} и~\emph{5.9} из~\emph{\cite{Su2017}} в отношении зависимости оценок от $\tau$. 
\end{remark}

\subsection{Подтверждение точности}
Из теорем~\ref{abstr_exp_smooth_shrp_thrm_1}, \ref{abstr_exp_smooth_shrp_thrm_2}, \ref{abstr_exp_time_shrp_thrm_1}, \ref{abstr_exp_time_shrp_thrm_2} и неравенства~(\ref{abstr_sndw_exp_est_2}) непосредственно вытекают следующие утверждения.

\begin{thrm}[\cite{Su2017}]
	\label{abstr_sndw_exp_smooth_shrp_thrm_1}
	Пусть $\widehat{N}_{0,Q} \ne 0$. Пусть $\tau \ne 0$ и $0 \le s < 3$. Тогда не существует такой константы $C(\tau) >0$, чтобы оценка
	\begin{equation*}
	\bigl\| M e^{-i \tau \varepsilon^{-2} A(t)} M^{-1} \widehat{P} -  M_0 e^{-i \tau \varepsilon^{-2} t^2 M_0 \widehat{S} M_0} M_0^{-1} \widehat{P} \bigr\| \varepsilon^s (t^2 + \varepsilon^2)^{-s/2} \le  C(\tau) \varepsilon
	\end{equation*}
выполнялась для всех достаточно малых $|t|$ и $\varepsilon$.
\end{thrm}

\begin{thrm}
	\label{abstr_sndw_exp_smooth_shrp_thrm_2}
	Пусть $\widehat{N}_{0,Q} = 0$ и $\widehat{\mathcal{N}}_Q^{(q)} \ne 0$ для некоторого $q \in \{1, \ldots, p\}$.  Пусть $\tau \ne 0$ и $0 \le s < 2$. Тогда не существует такой константы $C(\tau) >0$, чтобы оценка
	\begin{equation*}
	\bigl\| M e^{-i \tau \varepsilon^{-2} A(t)} M^{-1} \widehat{P} -  M_0 e^{-i \tau \varepsilon^{-2} t^2 M_0 \widehat{S} M_0} M_0^{-1} \widehat{P} \bigr\| \varepsilon^s (t^2 + \varepsilon^2)^{-s/2} \le  C(\tau) \varepsilon
	\end{equation*}
выполнялась для всех достаточно малых $|t|$ и $\varepsilon$.
\end{thrm}

\begin{thrm}
	\label{abstr_sndw_exp_time_shrp_thrm_1}
	Пусть $\widehat{N}_{0,Q} \ne 0$. Тогда не существует положительной функции $C(\tau)$ такой, что $\lim_{\tau \to \infty} C(\tau)/ |\tau| = 0$ и выполнена оценка
	\begin{equation}
	\label{abstr_sndw_exp_time_shrp_est_1}
	\bigl\| M e^{-i \tau \varepsilon^{-2} A(t)} M^{-1} \widehat{P} -  M_0 e^{-i \tau \varepsilon^{-2} t^2 M_0 \widehat{S} M_0} M_0^{-1} \widehat{P} \bigr\| \varepsilon^3 (t^2 + \varepsilon^2)^{-3/2} \le  C(\tau) \varepsilon
	\end{equation}
	при всех $\tau \in \mathbb{R}$ и всех достаточно малых $|t|$ и $\varepsilon > 0$.
\end{thrm}

\begin{thrm}
	\label{abstr_sndw_exp_time_shrp_thrm_2}
	Пусть $\widehat{N}_{0,Q} = 0$ и $\widehat{\mathcal{N}}_Q^{(q)} \ne 0$ для некоторого $q \in \{1, \ldots, p\}$. Тогда не существует положительной функции $C(\tau)$ такой, что $\lim_{\tau \to \infty} C(\tau)/ |\tau|^{1/2} = 0$ и выполнена оценка
	\begin{equation*}
	\bigl\| M e^{-i \tau \varepsilon^{-2} A(t)} M^{-1} \widehat{P} -  M_0 e^{-i \tau \varepsilon^{-2} t^2 M_0 \widehat{S} M_0} M_0^{-1} \widehat{P} \bigr\| \varepsilon^2 (t^2 + \varepsilon^2)^{-1} \le  C(\tau) \varepsilon
	\end{equation*}
	при всех $\tau \in \mathbb{R}$ и всех достаточно малых $|t|$ и $\varepsilon > 0$.
\end{thrm}

Теорема~\ref{abstr_sndw_exp_smooth_shrp_thrm_1} была доказана в~\cite[теорема~5.10]{Su2017}.

\part{Усреднение периодических\\дифференциальных операторов в 	$L_2(\mathbb{R}^d; \mathbb{C}^n)$}
\label{L2_operators_part}

\section{Периодические дифференциальные операторы в $L_2(\mathbb{R}^d; \mathbb{C}^n)$}

\subsection{Предварительные сведения: решётки  и преобразование Гельфанда}
Пусть $\Gamma$~--- решётка в $\mathbb{R}^d$, порождённая базисом $\mathbf{a}_1, \ldots , \mathbf{a}_d$, т.~е. $\Gamma = \left\{ \mathbf{a} \in \mathbb{R}^d \colon \mathbf{a} = \sum_{j=1}^{d} n_j \mathbf{a}_j, \; n_j \in \mathbb{Z} \right\}$,
и пусть $\Omega$~--- элементарная ячейка решётки $\Gamma$: $
\Omega \coloneqq \left\{ \mathbf{x} \in \mathbb{R}^d \colon \mathbf{x} = \sum_{j=1}^{d} \xi_j \mathbf{a}_j, \; 0 < \xi_j < 1 \right\}$. 
Базис $\mathbf{b}_1, \ldots, \mathbf{b}_d$, двойственный по отношению к $\mathbf{a}_1, \ldots , \mathbf{a}_d$, определяется из соотношений $\left< \mathbf{b}_l, \mathbf{a}_j \right> = 2 \pi \delta_{lj}$. Этот базис порождает решётку $\widetilde \Gamma$, \emph{двойственную} к решётке $\Gamma$. Обозначим через $\widetilde \Omega$ центральную \emph{зону Бриллюэна} решётки $\widetilde \Gamma$:
\begin{equation}
\label{Brillouin_zone}
\widetilde \Omega = \left\{ \mathbf{k} \in \mathbb{R}^d \colon | \mathbf{k} | < | \mathbf{k} - \mathbf{b} |, \; 0 \ne \mathbf{b} \in \widetilde \Gamma \right\}.
\end{equation}
Будем пользоваться обозначениями $| \Omega | = \meas \Omega$, $| \widetilde \Omega | = \meas \widetilde \Omega$ и отметим, что $| \Omega |  | \widetilde \Omega | = (2 \pi)^d$. Пусть $r_0$~--- радиус шара, \emph{вписанного} в $\clos \widetilde \Omega$. Отметим, что
\begin{equation*}
2 r_0 = \min|\mathbf{b}|,  \quad 0 \ne \mathbf{b} \in \widetilde \Gamma.
\end{equation*}
С решёткой $ \Gamma $ связано разложение в ряд Фурье
$\{ \hat{\mathbf{u}}_{\mathbf{b}} \} \mapsto \mathbf{u}$:
$\mathbf{u}(\mathbf{x}) = | \Omega |^{-1/2} \sum_{\mathbf{b} \in \widetilde \Gamma} \hat{\mathbf{u}}_{\mathbf{b}} e^{i \left<\mathbf{b}, \mathbf{x} \right>}$, которое унитарно отображает $l_2 (\widetilde \Gamma; \mathbb{C}^n) $ на $L_2 (\Omega; \mathbb{C}^n)$. \textit{Через $\widetilde H^\sigma(\Omega; \mathbb{C}^n)$ обозначается подпространство тех функций  из $ H^\sigma(\Omega; \mathbb{C}^n)$,  $\Gamma$-перио\-ди\-ческое продолжение которых на $\mathbb{R}^d$ принадлежит $H^\sigma_{\mathrm{loc}}(\mathbb{R}^d; \mathbb{C}^n)$}. Имеет место равенство
\begin{equation}
\label{D_and_fourier}
\int_{\Omega} |(\mathbf{D} + \mathbf{k}) \mathbf{u}|^2\, d\mathbf{x} = \sum_{\mathbf{b} \in \widetilde{\Gamma}} |\mathbf{b} + \mathbf{k} |^2 |\hat{\mathbf{u}}_{\mathbf{b}}|^2, \qquad \mathbf{u} \in \widetilde{H}^1(\Omega; \mathbb{C}^n), \; \mathbf{k} \in \mathbb{R}^d,
\end{equation} 
причём сходимость ряда в правой части~(\ref{D_and_fourier}) равносильна включению $\mathbf{u} \in \widetilde{H}^1(\Omega; \mathbb{C}^n)$. Из~(\ref{Brillouin_zone})~и~(\ref{D_and_fourier}) следует оценка
\begin{equation}
\label{(D+k)u_est}
\int_{\Omega} |(\mathbf{D} + \mathbf{k}) \mathbf{u}|^2\, d\mathbf{x} \ge \sum_{\mathbf{b} \in \widetilde{\Gamma}} | \mathbf{k} |^2 |\hat{\mathbf{u}}_{\mathbf{b}}|^2 = | \mathbf{k} |^2 \int_{\Omega} |\mathbf{u}|^2\, d\mathbf{x}, \qquad  \mathbf{u} \in \widetilde{H}^1(\Omega; \mathbb{C}^n), \; \mathbf{k} \in \widetilde{\Omega}.
\end{equation}  

Преобразование Гельфанда $\mathscr{U}$ первоначально определяется на функциях из класса Шварца $\mathbf{v} \in \mathcal{S}(\mathbb{R}^d; \mathbb{C}^n)$ формулой:
\begin{equation*}
\tilde{\mathbf{v}} ( \mathbf{k}, \mathbf{x}) = (\mathscr{U} \- \mathbf{v}) (\mathbf{k}, \mathbf{x}) = | \widetilde \Omega |^{-1/2} \sum_{\mathbf{a} \in \Gamma} e^{- i \left< \mathbf{k}, \mathbf{x} + \mathbf{a} \right>} \mathbf{v} ( \mathbf{x} + \mathbf{a}), \qquad
\mathbf{x} \in \Omega, \; \mathbf{k} \in \widetilde \Omega,
\end{equation*}
и продолжается по непрерывности до унитарного отображения:
\begin{equation*}
\mathscr{U} \colon L_2 (\mathbb{R}^d; \mathbb{C}^n) \to \int_{\widetilde \Omega} \oplus  L_2 (\Omega; \mathbb{C}^n) \, d \mathbf{k} \eqqcolon \mathcal{K}.
\end{equation*}

\subsection{Факторизованные операторы $\mathcal{A}$ второго порядка}
\label{A_section}
Пусть $b (\mathbf{D})= \sum_{l=1}^d b_l D_l$, где $b_l$ --- постоянные ($ m \times n $)-матрицы (вообще говоря, с комплексными элементами).
\emph{Предполагается, что $m \ge n$}. Рассмотрим символ $b(\boldsymbol{\xi}) = \sum_{l=1}^d b_l \xi_l$, $\boldsymbol{\xi} \in \mathbb{R}^d$.
\emph{Предположим, что} $\rank b( \boldsymbol{\xi} ) = n$, $0 \ne  \boldsymbol{\xi} \in \mathbb{R}^d$. Это равносильно тому, что для некоторых $\alpha_0, \alpha_1$ выполнены неравенства
\begin{equation}
\label{alpha0_alpha1}
\alpha_0 \mathbf{1}_n \le b( \boldsymbol{\theta} )^* b( \boldsymbol{\theta} ) \le \alpha_1 \mathbf{1}_n, \quad  \boldsymbol{\theta} \in \mathbb{S}^{d-1}, \quad 0 < \alpha_0 \le \alpha_1 < \infty.
\end{equation}
Пусть $f(\mathbf{x})$~--- $\Gamma$-периодическая  ($n \times n$)-матричнозначная функция и $h(\mathbf{x})$~--- $\Gamma$-пе\-риоди\-чес\-кая  ($m \times m$)-матричнозначная функция, такие что
\begin{equation}
\label{h_f_Linfty}
f, f^{-1} \in L_{\infty} (\mathbb{R}^d); \quad h, h^{-1} \in L_{\infty} (\mathbb{R}^d).
\end{equation}
Рассмотрим замкнутый оператор $\mathcal{X}  \colon  L_2 (\mathbb{R}^d ; \mathbb{C}^n) \to  L_2 (\mathbb{R}^d ; \mathbb{C}^m)$, заданный выражением
$\mathcal{X} = h b( \mathbf{D} ) f$ на области определения
$\Dom \mathcal{X} = \left\lbrace \mathbf{u} \in L_2 (\mathbb{R}^d ; \mathbb{C}^n) \colon f \mathbf{u} \in H^1  (\mathbb{R}^d ; \mathbb{C}^n) \right\rbrace$. Самосопряжённый оператор $\mathcal{A} = \mathcal{X}^* \mathcal{X}$ в $L_2 (\mathbb{R}^d ; \mathbb{C}^n)$ порождается замкнутой квадратичной формой $\mathfrak{a}[\mathbf{u}, \mathbf{u}] = \| \mathcal{X} \mathbf{u} \|^2_{L_2(\mathbb{R}^d)}$, $ \mathbf{u} \in \Dom \mathcal{X}$. Формально,
\begin{equation}
\label{A}
\mathcal{A} = f (\mathbf{x})^* b( \mathbf{D} )^* g( \mathbf{x} )  b( \mathbf{D} ) f(\mathbf{x}),
\end{equation}
где $g( \mathbf{x} ) = h( \mathbf{x} )^*  h( \mathbf{x} )$. Используя преобразование Фурье и~(\ref{alpha0_alpha1}),~(\ref{h_f_Linfty}), легко проверить оценки
\begin{equation*}
\alpha_0 \| g^{-1} \|_{L_{\infty}}^{-1} \| \mathbf{D} (f \mathbf{u}) \|_{L_2(\mathbb{R}^d)}^2 \le \mathfrak{a}[\mathbf{u}, \mathbf{u}] \le \alpha_1 \| g \|_{L_{\infty}} \| \mathbf{D} (f \mathbf{u}) \|_{L_2(\mathbb{R}^d)}^2, \qquad \mathbf{u} \in \Dom \mathcal{X}.
\end{equation*}

\subsection{Операторы $\mathcal{A}(\mathbf{k})$}
Положим
\begin{equation}
\label{Spaces_H}
\mathfrak{H} = L_2 (\Omega; \mathbb{C}^n), \qquad \mathfrak{H}_* = L_2 (\Omega; \mathbb{C}^m)
\end{equation}
и рассмотрим замкнутый оператор $\mathcal{X} (\mathbf{k}) \colon \mathfrak{H} \to \mathfrak{H}_*$, зависящий от параметра $\mathbf{k} \in \mathbb{R}^d$ и заданный выражением
$\mathcal{X} (\mathbf{k}) = hb(\mathbf{D} + \mathbf{k})f$ на области 
\begin{equation*}
\Dom \mathcal{X} (\mathbf{k}) = \bigl\lbrace \mathbf{u} \in \mathfrak{H} \colon   f \mathbf{u} \in \widetilde{H}^1 (\Omega; \mathbb{C}^n)\bigr\rbrace \eqqcolon \mathfrak{d}.
\end{equation*}
Самосопряжённый оператор $\mathcal{A} (\mathbf{k}) =\mathcal{X} (\mathbf{k})^* \mathcal{X} (\mathbf{k}) \colon \mathfrak{H} \to \mathfrak{H}$ порождается квадратичной формой $\mathfrak{a}(\mathbf{k})[\mathbf{u}, \mathbf{u}] = \| \mathcal{X}(\mathbf{k}) \mathbf{u} \|_{\mathfrak{H}_*}^2$, $\mathbf{u} \in \mathfrak{d}$.
Используя разложение функции $\mathbf{u}$ в ряд Фурье и условия~(\ref{alpha0_alpha1}),~(\ref{h_f_Linfty}), легко проверить, что
\begin{equation}
\label{a(k)_est}
\alpha_0 \|g^{-1} \|_{L_\infty}^{-1} \|(\mathbf{D} + \mathbf{k}) f \mathbf{u} \|_{\widetilde{\mathfrak{H}}}^2 \le \mathfrak{a}(\mathbf{k})[\mathbf{u}, \mathbf{u}] \le \alpha_1 \|g \|_{L_\infty} \|(\mathbf{D} + \mathbf{k}) f \mathbf{u} \|_{\widetilde{\mathfrak{H}}}^2, \quad \mathbf{u} \in \mathfrak{d}.
\end{equation}

Из~(\ref{(D+k)u_est}) и нижней оценки~(\ref{a(k)_est}) вытекает, что
\begin{equation}
\label{c_*}	
\mathcal{A} (\mathbf{k}) \ge c_* |\mathbf{k}|^2 I, \qquad \mathbf{k} \in \widetilde{\Omega}, \; c_* = \alpha_0\|f^{-1} \|_{L_\infty}^{-2} \|g^{-1} \|_{L_\infty}^{-1} .
\end{equation}

Положим $\mathfrak{N} \coloneqq \Ker \mathcal{A} (0) = \Ker \mathcal{X} (0)$. Соотношения~(\ref{a(k)_est}) при $\mathbf{k} = 0$ показывают, что
\begin{equation}
\label{frakN}
\mathfrak{N} = \left\lbrace \mathbf{u} \in L_2 (\Omega; \mathbb{C}^n) \colon f \mathbf{u} = \mathbf{c} \in \mathbb{C}^n \right\rbrace, \qquad \dim \mathfrak{N} = n. 
\end{equation}

\subsection{Зонные функции}
Обозначим через $E_j(\mathbf{k})$, $j \in \mathbb{N}$, последовательные (с учётом кратностей) собственные значения оператора $\mathcal{A}(\mathbf{k})$ (зонные функции):
\begin{equation*}
E_1(\mathbf{k}) \le E_2(\mathbf{k}) \le \ldots \le E_j(\mathbf{k}) \le \ldots, \qquad \mathbf{k} \in \mathbb{R}^d.
\end{equation*}
Зонные функции $E_j(\mathbf{k})$ непрерывны и $\widetilde{\Gamma}$-периодичны.
Как показано в~\cite[гл.~2, п.~2.2]{BSu2003} (на основании простых вариационных соображений), зонные функции удовлетворяют следующим оценкам:
\begin{alignat*}{2}
E_j(\mathbf{k})  &\ge c_* | \mathbf{k} |^2, \qquad  &&\mathbf{k} \in \clos \widetilde{\Omega}, \qquad j = 1, \ldots, n, \\
E_{n+1}(\mathbf{k}) &\ge c_* r_0^2, \qquad &&\mathbf{k} \in \clos \widetilde{\Omega}, \\
E_{n+1}(0)  &\ge 4 c_* r_0^2.
\end{alignat*}

\subsection{Прямой интеграл для оператора $\mathcal{A}$}
Под действием преобразования Гельфанда $\mathscr{U}$ оператор $\mathcal{A}$ раскладывается в прямой интеграл по операторам  $\mathcal{A} (\mathbf{k})$:
\begin{equation}
\label{Gelfand_A_decompose}
\mathscr{U} \mathcal{A}  \mathscr{U}^{-1} = \int_{\widetilde \Omega} \oplus \mathcal{A} (\mathbf{k}) \, d \mathbf{k}.
\end{equation}
Подразумевается следующее. Пусть $\mathbf{v} \in \Dom \mathcal{X}$, тогда $\tilde{\mathbf{v}}(\mathbf{k}, \cdot) \in \mathfrak{d}$ при п.в. $\mathbf{k} \in \widetilde \Omega$ и
\begin{equation}
\label{Gelfand_a_form}
\mathfrak{a}[\mathbf{v}, \mathbf{v}] = \int_{\widetilde{\Omega}} \mathfrak{a}(\mathbf{k}) [\tilde{\mathbf{v}}(\mathbf{k}, \cdot), \tilde{\mathbf{v}}(\mathbf{k}, \cdot)] \, d \mathbf{k} .
\end{equation}
Обратно, если для $\tilde{\mathbf{v}} \in \mathcal{K}$ справедливо $\tilde{\mathbf{v}}(\mathbf{k}, \cdot) \in \mathfrak{d}$ при п.в. $\mathbf{k} \in \widetilde \Omega$ и интеграл в~(\ref{Gelfand_a_form}) конечен, то $\mathbf{v} \in \Dom \mathcal{X}$ и выполнено~(\ref{Gelfand_a_form}).

\subsection{Включение операторов $\mathcal{A} (\mathbf{k})$ в абстрактную схему}
Если $d > 1$, то операторы  $\mathcal{A} (\mathbf{k})$ зависят от многомерного параметра $\mathbf{k}$. Следуя \cite[гл.~2]{BSu2003}, введём одномерный параметр $t = | \mathbf{k}|$. Будем использовать схему главы~\ref{abstr_part}. При этом все построения будут зависеть от дополнительного параметра $\boldsymbol{\theta} = \mathbf{k} / | \mathbf{k}| \in \mathbb{S}^{d-1}$ и мы должны следить за равномерностью оценок по $\boldsymbol{\theta}$. Пространства $\mathfrak{H}$ и $\mathfrak{H}_*$ определены в~(\ref{Spaces_H}). Положим $X(t) = X(t; \boldsymbol{\theta}) \coloneqq \mathcal{X}(t \boldsymbol{\theta})$. При этом выполнено $X(t; \boldsymbol{\theta}) = X_0 + t  X_1 (\boldsymbol{\theta})$, где $X_0 = h(\mathbf{x}) b (\mathbf{D}) f(\mathbf{x})$, $\Dom X_0 = \mathfrak{d}$, а $X_1 (\boldsymbol{\theta})$  --- ограниченный оператор умножения на матрицу $h(\mathbf{x}) b(\boldsymbol{\theta}) f(\mathbf{x})$. Далее, положим $A(t) = A(t; \boldsymbol{\theta}) \coloneqq \mathcal{A}(t \boldsymbol{\theta})$. Ядро $\mathfrak{N} = \Ker X_0$ описано в~(\ref{frakN}). Как было показано в \cite[гл.~2,~\S3]{BSu2003}, расстояние $d^0$ от точки $\lambda_0 = 0$ до остального спектра оператора $\mathcal{A}(0)$ подчинено оценке $d^0 \ge 4 c_* r_0^2$. Условие $n \le n_* = \dim \Ker X^*_0$ также выполнено. Более того, либо $n_* = n$ (если $m = n$), либо $n_* = \infty$ (если $m > n$).

Следуя пункту~\ref{abstr_X_A_section}, мы должны фиксировать $\delta \in (0, d^0/8)$. Так как $d^0 \ge 4 c_* r_0^2$, положим
\begin{equation}
\label{delta_fixation}
\delta = \frac{1}{4} c_* r^2_0 = \frac{1}{4} \alpha_0\|f^{-1} \|_{L_\infty}^{-2} \|g^{-1} \|_{L_\infty}^{-1} r^2_0.
\end{equation}  
Отметим, что в силу~(\ref{alpha0_alpha1}) и~(\ref{h_f_Linfty}) справедлива оценка
\begin{equation}
\label{X_1_estimate}
\| X_1 (\boldsymbol{\theta}) \| \le  \alpha^{1/2}_1 \| h \|_{L_{\infty}} \| f \|_{L_{\infty}}, \qquad \boldsymbol{\theta} \in \mathbb{S}^{d-1}.
\end{equation}

Для $t^0$ (см.~(\ref{abstr_t0})) примем следующее значение:
\begin{equation}
\label{t0_fixation}
t^0 = \delta^{1/2} \alpha_1^{-1/2} \| h \|_{L_{\infty}}^{-1} \|f\|_{L_{\infty}}^{-1} = \frac{r_0}{2} \alpha_0^{1/2} \alpha_1^{-1/2} \left( \| h \|_{L_{\infty}} \| h^{-1} \|_{L_{\infty}} \|f \|_{L_\infty} \|f^{-1} \|_{L_\infty} \right)^{-1}.
\end{equation}
Отметим, что $t^0 \le r_0/2$. Следовательно, шар $|\mathbf{k}| \le t^0$ целиком лежит внутри $\widetilde{\Omega}$. Важно, что величины $c_*$, $\delta$, $t^0$ (см.~(\ref{c_*}), (\ref{delta_fixation}), (\ref{t0_fixation})) не зависят от $\boldsymbol{\theta}$. Условие~\ref{abstr_nondegeneracy_cond} выполнено в силу~(\ref{c_*}). Росток $S(\boldsymbol{\theta})$ оператора $A(t, \boldsymbol{\theta})$ невырожден равномерно по $\boldsymbol{\theta}$: выполнено $S(\boldsymbol{\theta}) \ge c_* I_{\mathfrak{N}}$ (ср.~(\ref{abstr_S_nondegeneracy})).

\section{Эффективные характеристики оператора $\widehat{\mathcal{A}} = b(\mathbf{D})^* g(\mathbf{x}) b(\mathbf{D})$}
\subsection{Оператор $A(t, \boldsymbol{\theta})$ в случае $f = \mathbf{1}_n$}
Особую роль играет оператор $A(t, \boldsymbol{\theta})$ при $f = \mathbf{1}_n$. Условимся в этом случае отмечать все объекты шляпкой \textquotedblleft$\, \widehat{\phantom{\_}} \,$\textquotedblright. Тогда для оператора
\begin{equation}
\label{hatA}
\widehat{\mathcal{A}} = b(\mathbf{D})^* g(\mathbf{x}) b(\mathbf{D})
\end{equation}
семейство $\widehat{\mathcal{A}} (\mathbf{k})$ обозначается $\widehat{A} (t, \boldsymbol{\theta})$. Ядро~(\ref{frakN}) принимает вид
\begin{equation}
\label{Ker3}
\widehat{\mathfrak{N}} = \left\lbrace \mathbf{u} \in L_2 (\Omega; \mathbb{C}^n) \colon \mathbf{u} = \mathbf{c} \in \mathbb{C}^n \right\rbrace, 
\end{equation}
т.~е. $\widehat{\mathfrak{N}}$ состоит из постоянных вектор-функций. Ортопроектор $\widehat{P}$ пространства $L_2 (\Omega; \mathbb{C}^n)$ на подпространство (\ref{Ker3}) есть оператор усреднения по ячейке:
\begin{equation}
\label{Phat_projector}
\widehat{P} \mathbf{u} = |\Omega|^{-1} \int_{\Omega} \mathbf{u} (\mathbf{x}) \, d\mathbf{x}.
\end{equation}
Согласно~\cite[гл.~3,~\S1]{BSu2003}, спектральный росток $\widehat{S} (\boldsymbol{\theta}) \colon \widehat{\mathfrak{N}} \to \widehat{\mathfrak{N}} $ семейства $\widehat{A}(t, \boldsymbol{\theta})$ представим в виде $
\widehat{S} (\boldsymbol{\theta}) = b(\boldsymbol{\theta})^* g^0 b(\boldsymbol{\theta})$,  $\boldsymbol{\theta} \in \mathbb{S}^{d-1}$, где $g^0$~--- так называемая \emph{эффективная матрица}. Постоянная ($m \times m$)-матрица $g^0$ определяется следующим образом. Пусть $\Lambda \in \widetilde{H}^1 (\Omega)$~--- периодическая ($n \times m$)-матричнозначная функция, удовлетворяющая уравнению
\begin{equation}
\label{equation_for_Lambda}
b(\mathbf{D})^* g(\mathbf{x}) (b(\mathbf{D}) \Lambda (\mathbf{x}) + \mathbf{1}_m) = 0, \qquad \int_{\Omega} \Lambda (\mathbf{x}) \, d \mathbf{x} = 0.
\end{equation}
Эффективная матрица $g^0$ может быть определена в терминах матрицы $\Lambda (\mathbf{x})$:
\begin{align}
\label{g0}
&g^0 = | \Omega |^{-1} \int_{\Omega} \widetilde{g} (\mathbf{x}) \, d \mathbf{x},\\
\label{g_tilde}
&\widetilde{g} (\mathbf{x}) \coloneqq g(\mathbf{x})( b(\mathbf{D}) \Lambda (\mathbf{x}) + \mathbf{1}_m).
\end{align}
Выясняется, что матрица $g^0$ положительно определена. Рассмотрим символ
\begin{equation}
\label{effective_oper_symb}
\widehat{S} (\mathbf{k}) \coloneqq t^2 \widehat{S} (\boldsymbol{\theta}) = b(\mathbf{k})^* g^0 b(\mathbf{k}), \qquad \mathbf{k} \in \mathbb{R}^{d}.
\end{equation}
Выражение~(\ref{effective_oper_symb}) является символом ДО
\begin{equation}
\label{hatA0}
\widehat{\mathcal{A}}^0 = b(\mathbf{D})^* g^0 b(\mathbf{D}),
\end{equation}
действующего в $L_2(\mathbb{R}^d; \mathbb{C}^n)$ и называемого \emph{эффективным оператором} для оператора $\widehat{\mathcal{A}}$.

Пусть $\widehat{\mathcal{A}}^0 (\mathbf{k})$~--- операторное семейство в $L_2(\Omega; \mathbb{C}^n)$, отвечающее оператору~(\ref{hatA0}). Тогда $\widehat{\mathcal{A}}^0 (\mathbf{k}) = b(\mathbf{D} + \mathbf{k})^* g^0 b(\mathbf{D} + \mathbf{k})$ при периодических граничных условиях. Отсюда с учётом~(\ref{Phat_projector}) и~(\ref{effective_oper_symb}) вытекает тождество
\begin{equation}
\label{hatS_P=hatA^0_P}
\widehat{S} (\mathbf{k}) \widehat{P} = \widehat{\mathcal{A}}^0 (\mathbf{k}) \widehat{P}.
\end{equation}

\subsection{Свойства эффективной матрицы}

Следующие свойства $g^0$ были проверены в~\cite[гл.~3, теорема~1.5]{BSu2003}.
\begin{proposition}[\cite{BSu2003}]
	Для эффективной матрицы справедливы оценки
	\begin{equation}
	\label{Voigt_Reuss}
	\underline{g} \le g^0 \le \overline{g},
	\end{equation}
	где $\overline{g} \coloneqq | \Omega |^{-1} \int_{\Omega} g (\mathbf{x}) \, d \mathbf{x}$ и $\underline{g} \coloneqq \left( | \Omega |^{-1} \int_{\Omega} g (\mathbf{x})^{-1} \, d \mathbf{x}\right)^{-1}$.	В случае $m = n$ всегда выполнено $g^0 = \underline{g}$.
\end{proposition}
Оценки~(\ref{Voigt_Reuss}) известны в теории усреднения для конкретных ДО как вилка Фойгта--Рейсса. Выделим теперь условия, при которых реализуется верхняя или нижняя грань в~(\ref{Voigt_Reuss}). Следующие утверждения были проверены в~\cite[гл.~3, предложения~1.6, 1.7]{BSu2003}.
\begin{proposition}[\cite{BSu2003}]
	Равенство $g^0 = \overline{g}$ равносильно соотношениям
	\begin{equation}
	\label{g0=overline_g_relat}
	b(\mathbf{D})^* \mathbf{g}_k (\mathbf{x}) = 0, \quad k = 1, \ldots, m,
	\end{equation}
	где $\mathbf{g}_k (\mathbf{x}), \; k = 1, \ldots,m$~--- столбцы матрицы $g (\mathbf{x})$.
\end{proposition}
\begin{proposition}[\cite{BSu2003}]
	Равенство $g^0 = \underline{g}$ равносильно представлениям
	\begin{equation}
	\label{g0=underline_g_relat}
	\mathbf{l}_k (\mathbf{x}) = \mathbf{l}^0_k + b(\mathbf{D}) \mathbf{w}_k(\mathbf{x}), \quad \mathbf{l}^0_k \in \mathbb{C}^m, \quad \mathbf{w}_k \in \widetilde{H}^1 (\Omega; \mathbb{C}^n), \quad k = 1, \ldots,m,
	\end{equation}
	где $\mathbf{l}_k (\mathbf{x}), \; k = 1, \ldots,m$~--- столбцы матрицы $g (\mathbf{x})^{-1}$.
\end{proposition}

\subsection{Аналитические ветви собственных значений и собственных элементов}

Аналитические (по $t$) ветви собственных значений $\widehat{\lambda}_l (t, \boldsymbol{\theta})$ и ветви собственных элементов $\widehat{\varphi}_l (t, \boldsymbol{\theta})$ оператора $\widehat{\mathcal{A}} (t, \boldsymbol{\theta})$ допускают степенные разложения вида~(\ref{abstr_A(t)_eigenvalues_series}),~(\ref{abstr_A(t)_eigenvectors_series}) с коэффициентами, зависящими от $\boldsymbol{\theta}$ (интервал сходимости $t = |\mathbf{k}| \le t_* (\boldsymbol{\theta})$ мы не контролируем):
\begin{align}
\label{hatA_eigenvalues_series}
 \widehat{\lambda}_l (t, \boldsymbol{\theta}) &= \widehat{\gamma}_l (\boldsymbol{\theta}) t^2 + \widehat{\mu}_l (\boldsymbol{\theta}) t^3 + \widehat{\nu}_l (\boldsymbol{\theta}) t^4 + \ldots, & l &= 1, \ldots, n,   \\
\label{hatA_eigenvectors_series}
\widehat{\varphi}_l (t, \boldsymbol{\theta}) &= \widehat{\omega}_l (\boldsymbol{\theta}) + t \widehat{\psi}^{(1)}_l (\boldsymbol{\theta}) + \ldots, & l &= 1, \ldots, n.
\end{align}
Согласно~(\ref{abstr_S_eigenvectors}) числа $\widehat{\gamma}_l (\boldsymbol{\theta})$ и элементы $\widehat{\omega}_l (\boldsymbol{\theta})$ являются собственными значениями и собственными элементами ростка: $
b(\boldsymbol{\theta})^* g^0 b(\boldsymbol{\theta}) \widehat{\omega}_l (\boldsymbol{\theta}) = \widehat{\gamma}_l (\boldsymbol{\theta}) \widehat{\omega}_l (\boldsymbol{\theta})$, $l = 1, \ldots, n$.

\subsection{Оператор $\widehat{N} (\boldsymbol{\theta})$}

Нам понадобится описать оператор $N$ (в абстрактных терминах определённый в теореме~\ref{abstr_threshold_approx_thrm_2}). Как проверено в~\cite[\S4]{BSu2005-2}, для семейства $\widehat{A} (t, \boldsymbol{\theta})$ этот оператор принимает вид
\begin{align}
\label{hatN(theta)}
\widehat{N} (\boldsymbol{\theta}) &= b(\boldsymbol{\theta})^* L(\boldsymbol{\theta}) b(\boldsymbol{\theta}) \widehat{P}, \\ 
\notag
L (\boldsymbol{\theta}) &\coloneqq | \Omega |^{-1} \int_{\Omega} (\Lambda (\mathbf{x})^* b(\boldsymbol{\theta})^* \widetilde{g}(\mathbf{x}) + \widetilde{g}(\mathbf{x})^* b(\boldsymbol{\theta}) \Lambda (\mathbf{x}) ) \, d \mathbf{x}.
\end{align}
Здесь $\Lambda (\mathbf{x})$~--- $\Gamma$-периодическое решение задачи~(\ref{equation_for_Lambda}), а $\widetilde{g}(\mathbf{x})$~--- матрица-функция~(\ref{g_tilde}).

В~\cite[\S4]{BSu2005} указаны некоторые достаточные условия, при которых $\widehat{N} (\boldsymbol{\theta}) = 0$.
\begin{proposition}[\cite{BSu2005}]
	\label{N=0_proposit}
	Пусть выполнено хотя бы одно из следующих предположений:
	\begin{enumerate}[label=\emph{\arabic*$^{\circ}.$}, ref=\arabic*$^{\circ}$, leftmargin=2.5\parindent]
		\setlength\itemsep{-0.1em}
		\item \label{N=0_proposit_p1}  
		$\widehat{\mathcal{A}} = \mathbf{D}^* g(\mathbf{x}) \mathbf{D}$, где $g(\mathbf{x})$~--- симметричная матрица с вещественными элементами.
		\item Выполнены соотношения~\emph{(\ref{g0=overline_g_relat})}, т.~е. $g^0 = \overline{g}$.
		\item Выполнены соотношения~\emph{(\ref{g0=underline_g_relat})}, т.~е. $g^0 = \underline{g}$. \emph{(}В частности, это автоматически выполнено, если $m = n$.\emph{)} 
	\end{enumerate}
	Тогда $\widehat{N} (\boldsymbol{\theta}) = 0$ при всех $\boldsymbol{\theta} \in \mathbb{S}^{d-1}$.
\end{proposition}

Напомним (см.~замечание~\ref{abstr_N_remark}), что $\widehat{N} (\boldsymbol{\theta}) = \widehat{N}_0 (\boldsymbol{\theta}) + \widehat{N}_* (\boldsymbol{\theta})$, где оператор $\widehat{N}_0 (\boldsymbol{\theta})$ диагонален в базисе $\{ \widehat{\omega}_l (\boldsymbol{\theta})\}_{l=1}^n$, а оператор $\widehat{N}_* (\boldsymbol{\theta})$ имеет нулевые диагональные элементы. При этом
\begin{equation*}
(\widehat{N} (\boldsymbol{\theta}) \widehat{\omega}_l (\boldsymbol{\theta}), \widehat{\omega}_l (\boldsymbol{\theta}))_{L_2 (\Omega)} = (\widehat{N}_0 (\boldsymbol{\theta}) \widehat{\omega}_l (\boldsymbol{\theta}), \widehat{\omega}_l (\boldsymbol{\theta}))_{L_2 (\Omega)} = \widehat{\mu}_l (\boldsymbol{\theta}), \qquad l=1, \ldots, n.
\end{equation*}

В~\cite[п.~4.3]{BSu2005-2} было доказано следующее предложение.
\begin{proposition}[\cite{BSu2005-2}]
	Пусть $b(\boldsymbol{\theta})$ и $g (\mathbf{x})$~--- матрицы с вещественными элементами. Пусть в разложениях~\emph{(\ref{hatA_eigenvectors_series})}  \textquotedblleft зародыши\textquotedblright \ $\widehat{\omega}_l (\boldsymbol{\theta})$, $l = 1, \ldots, n$, можно выбрать вещественными. Тогда в~\emph{(\ref{hatA_eigenvalues_series})} выполнено $\widehat{\mu}_l (\boldsymbol{\theta}) = 0$, $l=1, \ldots, n$, то есть, $\widehat{N}_0 (\boldsymbol{\theta}) = 0$.
\end{proposition}

В рассматриваемом \textquotedblleft вещественном\textquotedblright \ случае росток $\widehat{S} (\boldsymbol{\theta})$ представляет собой симметричную вещественную матрицу. Ясно, что в случае простого собственного значения $\widehat{\gamma}_j (\boldsymbol{\theta})$ ростка зародыш $\widehat{\omega}_j (\boldsymbol{\theta})$ определяется однозначно с точностью до фазового множителя, и его всегда можно выбрать вещественным. Мы получаем следующее следствие.

\begin{corollary}
	\label{real_S_spec_simple_coroll}
	Пусть $b(\boldsymbol{\theta})$ и $g (\mathbf{x})$~--- матрицы с вещественными элементами и пусть спектр ростка $\widehat{S} (\boldsymbol{\theta})$ простой. Тогда $\widehat{N}_0 (\boldsymbol{\theta}) = 0$.
\end{corollary}

Однако, как показывают примеры~\cite[пример~8.7]{Su2017}, \cite[п.~14.3]{DSu2018}, в \textquotedblleft вещественном\textquotedblright \ случае не всегда возможно выбрать векторы $\widehat{\omega}_l (\boldsymbol{\theta})$ вещественными. Может случиться, что $\widehat{N}_0 (\boldsymbol{\theta}) \ne 0$ в некоторых точках $\boldsymbol{\theta}$.

\subsection{Операторы $\widehat{Z}_2(\boldsymbol{\theta})$, $\widehat{R}_2(\boldsymbol{\theta})$, $\widehat{N}_1^0(\boldsymbol{\theta})$}
Опишем операторы $Z_2$, $R_2$, $N_1^0$ (в абстрактных терминах определённые в пп.~\ref{abstr_Z2_R2_section} и \ref{abstr_nu_section}) для семейства $\widehat{A} (t, \boldsymbol{\theta})$. Пусть $\Lambda^{(2)}_l (\mathbf{x})$~--- $\Gamma$-периодическое решение задачи
\begin{equation*}
b(\mathbf{D})^* g(\mathbf{x}) (b(\mathbf{D}) \Lambda^{(2)}_l (\mathbf{x}) + b_l \Lambda(\mathbf{x})) = b_l^* (g^0 - \widetilde{g} (\mathbf{x})), \qquad \int_{\Omega} \Lambda^{(2)}_l (\mathbf{x}) \, d \mathbf{x} = 0.
\end{equation*}
Положим $
\Lambda^{(2)} (\mathbf{x}; \boldsymbol{\theta}) \coloneqq \sum_{l=1}^{d} \Lambda^{(2)}_l (\mathbf{x}) \theta_l$.
Как проверено в~\cite[п.~6.3]{VSu2012}
\begin{gather*}
\widehat{Z}_2(\boldsymbol{\theta}) = \Lambda^{(2)} (\mathbf{x}; \boldsymbol{\theta}) b(\boldsymbol{\theta}) \widehat{P}, \qquad
\widehat{R}_2(\boldsymbol{\theta}) = h(\mathbf{x}) (b(\mathbf{D}) \Lambda^{(2)} (\mathbf{x}; \boldsymbol{\theta}) + b(\boldsymbol{\theta}) \Lambda(\mathbf{x})) b(\boldsymbol{\theta}).
\end{gather*}
Наконец, в~\cite[п.~6.4]{VSu2012} было получено представление
\begin{align*}
&\widehat{N}_1^0(\boldsymbol{\theta}) = b(\boldsymbol{\theta})^* L_2(\boldsymbol{\theta}) b(\boldsymbol{\theta}) \widehat{P}, 
\\
&\begin{multlined}[c][15cm]
L_2(\boldsymbol{\theta}) \coloneqq |\Omega|^{-1} \int_{\Omega} \left(\Lambda^{(2)} (\mathbf{x}; \boldsymbol{\theta})^* b(\boldsymbol{\theta})^* \widetilde{g}(\mathbf{x}) + \widetilde{g}(\mathbf{x})^* b(\boldsymbol{\theta}) \Lambda^{(2)} (\mathbf{x}; \boldsymbol{\theta}) \right) \, d \mathbf{x} + \\ +
|\Omega|^{-1} \int_{\Omega} \left(b(\mathbf{D}) \Lambda^{(2)} (\mathbf{x}; \boldsymbol{\theta}) + b(\boldsymbol{\theta}) \Lambda(\mathbf{x}) \right)^* g(\mathbf{x}) \left(b(\mathbf{D}) \Lambda^{(2)} (\mathbf{x}; \boldsymbol{\theta}) + b(\boldsymbol{\theta}) \Lambda(\mathbf{x})\right) \, d \mathbf{x}.
\end{multlined}
\end{align*}

\subsection{Кратности собственных значений ростка}
\label{eigenval_multipl_section}
В данном пункте считаем, что $n \ge 2$. Перейдём к обозначениям, принятым в~п.~\ref{abstr_cluster_section}, следя за кратностями собственных значений спектрального ростка $\widehat{S} (\boldsymbol{\theta})$. Вообще говоря, количество $p(\boldsymbol{\theta})$ различных собственных значений $\widehat{\gamma}^{\circ}_1 (\boldsymbol{\theta}), \ldots, \widehat{\gamma}^{\circ}_{p(\boldsymbol{\theta})} (\boldsymbol{\theta})$ спектрального ростка $\widehat{S}(\boldsymbol{\theta})$ и их кратности $k_1 (\boldsymbol{\theta}), \ldots, k_{p(\boldsymbol{\theta})} (\boldsymbol{\theta})$ зависят от параметра $\boldsymbol{\theta} \in \mathbb{S}^{d-1}$. При каждом фиксированном $\boldsymbol{\theta}$ через $\widehat{P}_j (\boldsymbol{\theta})$ обозначим ортопроектор в $L_2 (\Omega; \mathbb{C}^n)$ на собственное подпространство ростка $\widehat{S}(\boldsymbol{\theta})$, отвечающее собственному значению $\widehat{\gamma}_j^{\circ} (\boldsymbol{\theta})$. В силу~(\ref{abstr_N0_N*_invar_repr}) справедливы инвариантные (не зависящие от выбора базиса) представления для операторов $\widehat{N}_0 (\boldsymbol{\theta})$  и $\widehat{N}_* (\boldsymbol{\theta})$:
\begin{equation}
\label{N0_N*_invar_repr}
\widehat{N}_0 (\boldsymbol{\theta}) = \sum_{j=1}^{p(\boldsymbol{\theta})} \widehat{P}_j (\boldsymbol{\theta}) \widehat{N} (\boldsymbol{\theta}) \widehat{P}_j (\boldsymbol{\theta}), \qquad \widehat{N}_* (\boldsymbol{\theta}) = \sum_{\substack{1 \le l,j \le p(\boldsymbol{\theta})\\ j \ne l}} \widehat{P}_j (\boldsymbol{\theta}) \widehat{N} (\boldsymbol{\theta}) \widehat{P}_l (\boldsymbol{\theta}).
\end{equation}

\subsection{Коэффициенты $\widehat{\nu}_l (\boldsymbol{\theta})$, $l=1, \ldots,n$}
Количество $p'(q, \boldsymbol{\theta})$ различных собственных значений $\widehat{\mu}^{\circ}_{1,q} (\boldsymbol{\theta}), \ldots, \widehat{\mu}^{\circ}_{p'(q, \boldsymbol{\theta}),q} (\boldsymbol{\theta})$ оператора $\widehat{P}_q(\boldsymbol{\theta}) \widehat{N}(\boldsymbol{\theta}) |_{\widehat{\mathfrak{N}}_q(\boldsymbol{\theta})}$ и их кратности $k_{1,q} (\boldsymbol{\theta}), \ldots, k_{p'(\boldsymbol{\theta}),q} (\boldsymbol{\theta})$ также зависят от параметра $\boldsymbol{\theta} \in \mathbb{S}^{d-1}$. При каждом фиксированном $\boldsymbol{\theta}$ через $\widehat{P}_{q',q} (\boldsymbol{\theta})$ обозначим ортопроектор в $L_2 (\Omega; \mathbb{C}^n)$ на собственное подпространство $\widehat{\mathfrak{N}}_{q',q}(\boldsymbol{\theta})$, отвечающее собственному значению $\widehat{\mu}^{\circ}_{q',q} (\boldsymbol{\theta})$. 

Коэффициенты $\widehat{\nu}_l (\boldsymbol{\theta})$, $l = i'(q',q,\boldsymbol{\theta}), \ldots, i'(q',q,\boldsymbol{\theta})+k_{q',q}(\boldsymbol{\theta})-1$, где $i'(q',q,\boldsymbol{\theta}) = i(q,\boldsymbol{\theta})+k_{1,q}(\boldsymbol{\theta})+\ldots+k_{q'-1,q}(\boldsymbol{\theta})$, $i(q,\boldsymbol{\theta}) = k_1(\boldsymbol{\theta})+\ldots+k_{q-1}(\boldsymbol{\theta})+1$, являются собственными числами следующей задачи 
\begin{equation*}
\widehat{\mathcal{N}}^{(q',q)}(\boldsymbol{\theta}) \widehat{\omega}_l(\boldsymbol{\theta}) = \widehat{\nu}_l(\boldsymbol{\theta}) \widehat{\omega}_l(\boldsymbol{\theta}), \qquad l = i'(q',q,\boldsymbol{\theta}), \ldots, i'(q',q,\boldsymbol{\theta})+k_{q',q}(\boldsymbol{\theta})-1,
\end{equation*}
где 
\begin{multline*}
\widehat{\mathcal{N}}^{(q',q)}(\boldsymbol{\theta}) \coloneqq \widehat{P}_{q',q}(\boldsymbol{\theta}) \left. \left( \widehat{N}_1^0(\boldsymbol{\theta}) - \frac{1}{2} \widehat{Z}(\boldsymbol{\theta})^* \widehat{Z}(\boldsymbol{\theta}) \widehat{S}(\boldsymbol{\theta}) \widehat{P} - \frac{1}{2} \widehat{S}(\boldsymbol{\theta}) \widehat{P} \widehat{Z}(\boldsymbol{\theta})^* \widehat{Z}(\boldsymbol{\theta}) \right)\right|_{\widehat{\mathfrak{N}}_{q',q}} + \\ +
\sum_{\substack{j\in\{1,\ldots,p(\boldsymbol{\theta})\} \\ j \ne q}} \bigl(\gamma^{\circ}_q(\boldsymbol{\theta}) - \gamma^{\circ}_{j}(\boldsymbol{\theta})\bigr)^{-1} \widehat{P}_{q',q}(\boldsymbol{\theta}) \widehat{N}(\boldsymbol{\theta}) \widehat{P}_{j}(\boldsymbol{\theta}) \widehat{N}(\boldsymbol{\theta})|_{\widehat{\mathfrak{N}}_{q',q}(\boldsymbol{\theta})}.
\end{multline*}

Отметим, что в случае, когда $\widehat{N}_0(\boldsymbol{\theta}) = 0$, имеет место $\widehat{\mathfrak{N}}_{1,q}(\boldsymbol{\theta}) = \widehat{\mathfrak{N}}_{q}(\boldsymbol{\theta})$, $q=1, \ldots, p(\boldsymbol{\theta})$. Тогда вместо $\widehat{\mathcal{N}}^{(1,q)}(\boldsymbol{\theta})$ мы будем писать $\widehat{\mathcal{N}}^{(q)}(\boldsymbol{\theta})$. 

\subsection{Пример}
Рассмотрим скалярный эллиптический оператор 
\begin{equation*}
\widehat{\mathcal{A}} = - \operatorname{div} g (\mathbf{x}) \nabla = \mathbf{D}^* g (\mathbf{x}) \mathbf{D},
\end{equation*}
действующий в $L_2 (\mathbb{R}^d)$, $d \ge 1$, который является частным случаем оператора~(\ref{hatA}). Сейчас $n=1$, $m=d$, $b(\mathbf{D}) = \mathbf{D}$.

Эффективная матрица $g^0$ определяется стандартным образом. Пусть $\psi_j \in \widetilde{H}^1(\Omega)$~--- (слабое) $\Gamma$-периодическое решение задачи
\begin{equation}
\label{psi_equation}
\operatorname{div} g(\mathbf{x}) (\nabla \psi_j (\mathbf{x}) + \mathbf{e}_j) = 0, \qquad \int_{\Omega} \psi_j (\mathbf{x}) \, d\mathbf{x} = 0.
\end{equation}
Здесь $\mathbf{e}_1, \ldots, \mathbf{e}_d$~--- стандартные орты в $\mathbb{R}^d$. Матрица $\widetilde{g}(\mathbf{x})$~--- это ($d \times d$)-матрица со столбцами $\widetilde{\mathbf{g}}_j (\mathbf{x}) \coloneqq g(\mathbf{x}) (\nabla \psi_j (\mathbf{x}) + \mathbf{e}_j)$, $j=1,\dots,d$.  Тогда $g^0 = | \Omega |^{-1} \int_{\Omega} \widetilde{g}(\mathbf{x}) \, d \mathbf{x}$.

Если $g(\mathbf{x})$~--- эрмитова матрица с вещественными элементами, то согласно предложению~\ref{N=0_proposit}(\ref{N=0_proposit_p1}) выполнено $\widehat{N} (\boldsymbol{\theta}) = 0$ при всех $\boldsymbol{\theta} \in \mathbb{S}^{d-1}$. Если  же $g(\mathbf{x})$~--- эрмитова матрица с комплексными элементами, то в общей ситуации оператор $\widehat{N} (\boldsymbol{\theta})$ отличен от нуля. Сейчас $n=1$, а потому оператор $\widehat{N} (\boldsymbol{\theta}) = \widehat{N}_0 (\boldsymbol{\theta})$ есть оператор умножения на $\widehat{\mu}(\boldsymbol{\theta})$, где $\widehat{\mu}(\boldsymbol{\theta})$~--- коэффициент в разложении для первого собственного значения
\begin{equation*}
\widehat{\lambda}(t, \boldsymbol{\theta}) = \widehat{\gamma} (\boldsymbol{\theta}) t^2 + \widehat{\mu} (\boldsymbol{\theta}) t^3 + \widehat{\nu} (\boldsymbol{\theta}) t^4 + \ldots
\end{equation*}
оператора $\widehat{\mathcal{A}} (\mathbf{k})$. Вычисление~(см.~\cite[п.~10.3]{BSu2005-2}) показывает, что
\begin{align*}
&\widehat{N} (\boldsymbol{\theta}) = \widehat{\mu} (\boldsymbol{\theta}) = -i \sum_{j,l,r=1}^{d} (a_{jlr} - a_{ljr}^*) \theta_j \theta_l \theta_r, \\
&a_{jlr} = |\Omega|^{-1} \int_{\Omega} \psi_j (\mathbf{x})^* \left\langle g(\mathbf{x}) (\nabla \psi_l (\mathbf{x}) + \mathbf{e}_l), \mathbf{e}_r \right\rangle \, d \mathbf{x}, \qquad j, l, r = 1, \ldots, d. 
\end{align*}
Следующий пример заимствован из~\cite[п.~10.4]{BSu2005-2}.
\begin{example}[\cite{BSu2005-2}]
	Пусть $d=2$, $\Gamma = (2 \pi \mathbb{Z})^2$ и матрица $g(\mathbf{x})$ задана соотношением
	\begin{equation*}
	g(\mathbf{x}) = \begin{pmatrix}
	1 & i \beta'(x_1) \\
	- i \beta' (x_1) & 1
	\end{pmatrix},
	\end{equation*}
	где $\beta(x_1)$~--- гладкая вещественная $(2 \pi)$-периодическая функция такая, что $1 - (\beta'(x_1))^2 > 0$ и $\int_{0}^{2 \pi} \beta(x_1)\, d x_1 = 0$. В этом случае $\widehat{N} (\boldsymbol{\theta}) = - \alpha \pi^{-1} \theta_2^3$, где $\alpha = \int_{0}^{2 \pi} \beta (x_1) (\beta' (x_1))^2 dx_1$. Легко указать конкретный пример, когда $\alpha \ne 0$: достаточно положить $\beta (x_1) = c (\sin x_1 + \cos 2x_1)$ при $0 < c < 1/3$; тогда $\alpha = - (3 \pi/2) c^3 \ne 0$. В данном примере $\widehat{N} (\boldsymbol{\theta}) = \widehat{\mu} (\boldsymbol{\theta}) \ne 0$ при всех $\boldsymbol{\theta} \in \mathbb{S}^1$ за исключением точек $(\pm 1 ,0)$.
\end{example}

Далее, пусть $\phi_{jl} (\mathbf{x})$~--- $\Gamma$-периодическое решение задачи
\begin{equation}
\label{phi_equation}
-\operatorname{div} g(\mathbf{x}) (\nabla \phi_{jl} (\mathbf{x}) - \psi_j(\mathbf{x}) \mathbf{e}_l) = g^0_{lj} - \widetilde{g}_{lj} (\mathbf{x}), \qquad \int_{\Omega} \phi_{jl} (\mathbf{x}) \, d\mathbf{x} = 0.
\end{equation}
Оператор $\widehat{\mathcal{N}}^{(1,1)} (\boldsymbol{\theta})$ есть оператор умножения на $\widehat{\nu}(\boldsymbol{\theta})$. Вычисление~(см.~\cite[п.~14.5]{VSu2012}) показывает, что
\begin{align}
\label{hat_nu(k)}
&\widehat{\mathcal{N}}^{(1,1)} (\boldsymbol{\theta}) = \widehat{\nu}(\boldsymbol{\theta}) = \sum_{p,q,l,r=1}^{d} (\alpha_{pqlr} - (\overline{\psi_p^* \psi_q}) g^0_{lr}) \theta_p \theta_q \theta_l \theta_r,\\
\notag
&\begin{multlined}[c][15cm]
\alpha_{pqlr} = |\Omega|^{-1} \int_{\Omega} (\widetilde{g}_{lp} (\mathbf{x}) \phi_{qr}(\mathbf{x}) + \widetilde{g}_{rq} (\mathbf{x}) \phi_{pl}(\mathbf{x})) \, d \mathbf{x}
+\\ +
|\Omega|^{-1} \int_{\Omega} \left\langle g(\mathbf{x}) (\nabla \phi_{qr} (\mathbf{x}) - \psi_q(\mathbf{x}) \mathbf{e}_r), \nabla \phi_{pl} (\mathbf{x}) - \psi_p(\mathbf{x}) \mathbf{e}_l \right\rangle \, d \mathbf{x},\\
p,q,l,r = 1, \ldots, d.
\end{multlined}
\end{align}
\begin{lemma}
	Пусть $d=1$. Если $g \ne \const$, то $\widehat{\nu}(-1) = \widehat{\nu}(1) \ne 0$.
\end{lemma}
\begin{proof}
	Задача~(\ref{psi_equation}) сейчас имеет вид $\frac{d}{dx} g(x) (\frac{d}{dx} \psi_1 (x) + 1) = 0$, $\overline{\psi_1}= 0$. Тогда $\frac{d}{dx} \psi_1 (x) = \underline{g} (g(x))^{-1} - 1$. Поскольку $g(x) \ne \const$, то $\underline{g} (g(x))^{-1} - 1 \not\equiv 0$ и поэтому $\psi_1 \not\equiv 0$. Далее, $\widetilde{g}(x) = \underline{g} = g^0$ и  уравнение~(\ref{phi_equation}) имеет вид $\frac{d}{dx} g(x) (\frac{d}{dx} \phi_{11} (x) - \psi_1(x)) = 0$, $\overline{\phi_{11}} = 0$. Тогда $\frac{d}{dx} \phi_{11} (x) - \psi_1(x) = 0$. Нетрудно убедиться в том, что $\alpha_{1111}$ в~(\ref{hat_nu(k)}) сейчас равно нулю: $\alpha_{1111} = 0$. Так как $\overline{\psi_1^2} g^0 \ne 0$, то $\widehat{\nu}(-1) = \widehat{\nu}(1) \ne 0$.
\end{proof}

\section{Аппроксимация сглаженного оператора $e^{-i \tau \varepsilon^{-2} \widehat{\mathcal{A}}(\mathbf{k})}$}
\subsection{Общий случай}

Рассмотрим оператор $\mathcal{H}_0 = -\Delta$ в $L_2 (\mathbb{R}^d; \mathbb{C}^n)$. При разложении в прямой интеграл оператору $\mathcal{H}_0$ отвечает семейство операторов $\mathcal{H}_0 (\mathbf{k})$, действующих в $L_2 (\Omega; \mathbb{C}^n)$. Оператор $\mathcal{H}_0 (\mathbf{k})$ задаётся дифференциальным выражением $| \mathbf{D} + \mathbf{k} |^2$ при периодических граничных условиях. Введём обозначение
\begin{equation}
\label{R(k,eps)}
\mathcal{R}(\mathbf{k}, \varepsilon) \coloneqq \varepsilon^2 (\mathcal{H}_0 (\mathbf{k}) + \varepsilon^2 I)^{-1}.
\end{equation}
Очевидно, 
\begin{equation}
\label{R_P}
\mathcal{R}(\mathbf{k}, \varepsilon)^{s/2}\widehat{P} = \varepsilon^s (t^2 + \varepsilon^2)^{-s/2} \widehat{P}, \qquad s > 0.
\end{equation}
Отметим, что при $ |\mathbf{k}| > \widehat{t}^{\,0}$ выполнено неравенство
\begin{equation}
\label{R_hatP_est}
\| \mathcal{R}(\mathbf{k}, \varepsilon)^{s/2}\widehat{P} \|_{L_2 (\Omega) \to L_2 (\Omega)} \le (\widehat{t}^{\,0})^{-s} \varepsilon^s, \qquad \varepsilon > 0, \; \mathbf{k} \in \widetilde{\Omega}, \; |\mathbf{k}| > \widehat{t}^{\,0}.
\end{equation}
Далее, используя разложение в ряд Фурье, получаем
\begin{equation}
\label{R(k,eps)(I-P)_est}
\| \mathcal{R}(\mathbf{k}, \varepsilon)^{s/2} (I - \widehat{P}) \|_{L_2(\Omega) \to L_2 (\Omega) }  = \sup_{0 \ne \mathbf{b} \in \widetilde{\Gamma}} \varepsilon^s (|\mathbf{b} + \mathbf{k}|^2 + \varepsilon^2)^{-s/2} \le r_0^{-s} \varepsilon^s, \qquad
\varepsilon > 0, \; \mathbf{k} \in \widetilde{\Omega}.
\end{equation}

Обозначим 
\begin{equation}
\label{Jhat(k,eps)}
\widehat{J} (\mathbf{k}, \varepsilon; \tau) \coloneqq e^{-i \tau \varepsilon^{-2} \widehat{\mathcal{A}} (\mathbf{k})} - e^{-i \tau \varepsilon^{-2} \widehat{\mathcal{A}}^0 (\mathbf{k})}.
\end{equation}

Мы применим к оператору $\widehat{A}(t; \boldsymbol{\theta}) = \widehat{\mathcal{A}}(\mathbf{k})$ теоремы из~\S\ref{abstr_exp_section}. При этом мы можем отследить зависимость постоянных в оценках от исходных данных. Отметим, что $\widehat{c}_*$, $\widehat{\delta}$ и $\widehat{t}^ 0$ не зависят от $\boldsymbol{\theta}$ (см.~(\ref{c_*}), (\ref{delta_fixation}), (\ref{t0_fixation}) при $f = \mathbf{1}_n$). Согласно~(\ref{X_1_estimate}) (при $f = \mathbf{1}_n$) норму $\| \widehat{X}_1 (\boldsymbol{\theta}) \|$ можно заменить на $\alpha_1^{1/2} \| g \|_{L_{\infty}}^{1/2}$. Поэтому постоянные в теоремах~\ref{abstr_exp_general_thrm} и~\ref{abstr_exp_enchcd_thrm_1} (применённых к оператору $\widehat{\mathcal{A}}(\mathbf{k})$) не будут зависеть от $\boldsymbol{\theta}$. Они будут зависеть только от следующих величин: $\alpha_0$, $\alpha_1$, $\|g\|_{L_\infty}$, $\|g^{-1}\|_{L_\infty}$ и $r_0$.

Применяя теорему~\ref{abstr_exp_general_thrm} с учётом~(\ref{hatS_P=hatA^0_P}), (\ref{R_P})--(\ref{R(k,eps)(I-P)_est}), приходим к следующему утверждению, ранее доказанному в~\cite[теорема~7.1]{BSu2008}.
\begin{thrm}[\cite{BSu2008}]
	\label{hatA(k)_exp_general_thrm}
	При $\tau \in \mathbb{R}$, $\varepsilon > 0$ и $\mathbf{k} \in \widetilde{\Omega}$ выполнена оценка
	\begin{equation*}
	\| \widehat{J} (\mathbf{k}, \varepsilon; \tau) \mathcal{R}(\mathbf{k}, \varepsilon)^{3/2}\|_{L_2(\Omega) \to L_2 (\Omega) }  \le \widehat{\mathcal{C}}_1 (1 + |\tau|) \varepsilon,
	\end{equation*}
	где константа $\widehat{\mathcal{C}}_1$ зависит только от $\alpha_0$, $\alpha_1$, $\|g\|_{L_\infty}$, $\|g^{-1}\|_{L_\infty}$ и $r_0$.
\end{thrm}

\subsection{Случай, когда $\widehat{N}(\boldsymbol{\theta}) = 0$}
Применим теперь теорему~\ref{abstr_exp_enchcd_thrm_1}, предполагая, что $\widehat{N}(\boldsymbol{\theta}) = 0$ при всех $\boldsymbol{\theta} \in \mathbb{S}^{d-1}$. С  учётом~(\ref{hatS_P=hatA^0_P}), (\ref{R_P})--(\ref{R(k,eps)(I-P)_est}) это влечёт следующий результат.
\begin{thrm}
	\label{hatA(k)_exp_enchcd_thrm_1}
	Пусть оператор $\widehat{N}(\boldsymbol{\theta})$ определён в~\emph{(\ref{hatN(theta)})}. Пусть $\widehat{N}(\boldsymbol{\theta})~=~0$ при всех $\boldsymbol{\theta} \in \mathbb{S}^{d-1}$. Тогда при $\tau \in \mathbb{R}$, $\varepsilon > 0$ и $\mathbf{k} \in \widetilde{\Omega}$ выполнена оценка
	\begin{equation*}
	\| \widehat{J}(\mathbf{k}, \varepsilon; \tau ) \mathcal{R}(\mathbf{k}, \varepsilon)\|_{L_2(\Omega) \to L_2 (\Omega) }  \le \widehat{\mathcal{C}}_2 (1+ |\tau|^{1/2}) \varepsilon,
	\end{equation*}
	где константа $\widehat{\mathcal{C}}_2$ зависит только от $\alpha_0$, $\alpha_1$, $\|g\|_{L_\infty}$, $\|g^{-1}\|_{L_\infty}$ и $r_0$.
\end{thrm}

\subsection{Случай, когда $\widehat{N}_0(\boldsymbol{\theta}) = 0$}
\label{ench_approx2_section}
Теперь мы отказываемся от предположения теоремы~\ref{hatA(k)_exp_enchcd_thrm_1}, но взамен предположим, что $\widehat{N}_0(\boldsymbol{\theta}) = 0$ при всех $\boldsymbol{\theta}$. Нам хотелось бы применить теорему~\ref{abstr_exp_enchcd_thrm_2}. Однако, возникает дополнительное осложнение: кратность спектра ростка $\widehat{S} (\boldsymbol{\theta})$ может меняться в некоторых точках $\boldsymbol{\theta}$. При приближении к таким точкам расстояние между какой-то парой различных собственных значений стремится к нулю и мы не можем выбрать величины $\widehat{c}^{\circ}_{jl}$, $\widehat{t}^{\,00}_{jl}$ не зависящими от $\boldsymbol{\theta}$. Поэтому мы вынуждены накладывать дополнительные условия. Заботиться надо только о тех собственных значениях, для которых соответствующее слагаемое во второй формуле~(\ref{N0_N*_invar_repr}) отлично от нуля. При формулировке дополнительного условия удобнее пользоваться исходной нумерацией собственных значений $\widehat{\gamma}_1 (\boldsymbol{\theta}), \ldots , \widehat{\gamma}_n (\boldsymbol{\theta})$ ростка $\widehat{S} (\boldsymbol{\theta})$ (каждое собственное значение повторяется столько раз, какова его кратность), условившись нумеровать их в порядке неубывания: $\widehat{\gamma}_1 (\boldsymbol{\theta}) \le \widehat{\gamma}_2 (\boldsymbol{\theta}) \le \ldots \le \widehat{\gamma}_n (\boldsymbol{\theta})$. Через $\widehat{P}^{(k)} (\boldsymbol{\theta})$ обозначим ортопроектор пространства $L_2 (\Omega; \mathbb{C}^n)$ на собственное подпространство оператора $\widehat{S} (\boldsymbol{\theta})$, отвечающее собственному значению $\widehat{\gamma}_k (\boldsymbol{\theta})$. Ясно, что при каждом $\boldsymbol{\theta}$ оператор $\widehat{P}^{(k)} (\boldsymbol{\theta})$ совпадает с одним из проекторов $\widehat{P}_j (\boldsymbol{\theta})$, введённых в п.~\ref{eigenval_multipl_section} (но номер $j$ может зависеть от $\boldsymbol{\theta}$). 
\begin{condition}
	\label{cond1}
	\begin{enumerate*}[label=\emph{\arabic*$^{\circ}.$}, ref=\arabic*$^{\circ}$]
		\item $\widehat{N}_0(\boldsymbol{\theta})=0$ при всех $\boldsymbol{\theta} \in \mathbb{S}^{d-1}$.
		\item \label{cond1_it2} Для каждой пары индексов $(k,r)$, $1 \le k,r \le n$, $k \ne r$, такой, что $\widehat{\gamma}_k (\boldsymbol{\theta}_0) = \widehat{\gamma}_r (\boldsymbol{\theta}_0) $ при некотором $\boldsymbol{\theta}_0 \in \mathbb{S}^{d-1}$, выполнено $\widehat{P}^{(k)} (\boldsymbol{\theta}) \widehat{N} (\boldsymbol{\theta}) \widehat{P}^{(r)} (\boldsymbol{\theta}) = 0$ при всех $\boldsymbol{\theta} \in \mathbb{S}^{d-1}$.    
	\end{enumerate*}
\end{condition}
 
Условие~\ref{cond1_it2} может быть переформулировано: мы требуем, чтобы для ненулевых (тождественно) \textquotedblleft блоков\textquotedblright \ $\widehat{P}^{(k)} (\boldsymbol{\theta}) \widehat{N} (\boldsymbol{\theta}) \widehat{P}^{(r)} (\boldsymbol{\theta})$ оператора $\widehat{N} (\boldsymbol{\theta})$ соответствующие ветви собственных значений $\widehat{\gamma}_k (\boldsymbol{\theta})$ и  $\widehat{\gamma}_r (\boldsymbol{\theta})$ не пересекались.

Разумеется, выполнение условия~\ref{cond1} гарантируется следующим более сильным условием.
\begin{condition}
	\label{cond2}
	\begin{enumerate*}[label=\emph{\arabic*$^{\circ}.$}, ref=\arabic*$^{\circ}$]
		\item $\widehat{N}_0(\boldsymbol{\theta})=0$ при всех $\boldsymbol{\theta} \in \mathbb{S}^{d-1}$.
		\item \label{cond2_it2} Количество $p$ различных собственных значений спектрального ростка $\widehat{S}(\boldsymbol{\theta})$ не зависит от $\boldsymbol{\theta} \in \mathbb{S}^{d-1}$.        
	\end{enumerate*}
\end{condition}

При условии~\ref{cond2} обозначим различные собственные значения ростка, занумерованные в порядке возрастания, через $\widehat{\gamma}^{\circ}_1(\boldsymbol{\theta}), \ldots, \widehat{\gamma}^{\circ}_p(\boldsymbol{\theta})$. Тогда их кратности $k_1, \ldots, k_p$ не зависят от $\boldsymbol{\theta} \in \mathbb{S}^{d-1}$.  

\begin{remark}
	\begin{enumerate*}[label=\emph{\arabic*$^{\circ}.$}, ref=\arabic*$^{\circ}$]
		\item Предположение пункта~\emph{\ref{cond2_it2}} условия~\emph{\ref{cond2}} заведомо выполнено, если спектр ростка $\widehat{S}(\boldsymbol{\theta})$ простой при всех $\boldsymbol{\theta} \in \mathbb{S}^{d-1}$.
		\item Из следствия~\emph{\ref{real_S_spec_simple_coroll}} вытекает, что условие~\emph{\ref{cond2}} выполнено, если $b (\boldsymbol{\theta})$ и $g (\mathbf{x})$~--- матрицы с вещественными элементами и спектр ростка $\widehat{S}(\boldsymbol{\theta})$ простой при всех $\boldsymbol{\theta} \in \mathbb{S}^{d-1}$.
	\end{enumerate*}
\end{remark}

Итак, предполагаем выполненным условие~\ref{cond1}. Нас интересуют только пары индексов из множества 
\begin{equation*}
\widehat{\mathcal{K}} \coloneqq \{ (k,r) \colon 1 \le k,r \le n, \;  k \ne r, \; \widehat{P}^{(k)} (\boldsymbol{\theta}) \widehat{N} (\boldsymbol{\theta}) \widehat{P}^{(r)} (\boldsymbol{\theta}) \not\equiv 0 \}.
\end{equation*}
Введём обозначение $
\widehat{c}^{\circ}_{kr} (\boldsymbol{\theta}) \coloneqq \min \{\widehat{c}_*, n^{-1} |\widehat{\gamma}_k (\boldsymbol{\theta}) - \widehat{\gamma}_r (\boldsymbol{\theta})| \}$, $(k,r) \in \widehat{\mathcal{K}}$. Поскольку оператор $\widehat{S} (\boldsymbol{\theta})$ непрерывно зависит от $\boldsymbol{\theta} \in \mathbb{S}^{d-1}$, то из теории возмущений дискретного спектра следует, что $\widehat{\gamma}_j (\boldsymbol{\theta})$~--- непрерывные функции на сфере $\mathbb{S}^{d-1}$. В силу условия~\ref{cond1}(\ref{cond1_it2}) при $(k,r) \in \widehat{\mathcal{K}}$ выполнено~$|\widehat{\gamma}_k (\boldsymbol{\theta}) - \widehat{\gamma}_r (\boldsymbol{\theta})| > 0$ при всех $\boldsymbol{\theta} \in \mathbb{S}^{d-1}$, а тогда $\widehat{c}^{\circ}_{kr} \coloneqq \min_{\boldsymbol{\theta} \in \mathbb{S}^{d-1}} \widehat{c}^{\circ}_{kr} (\boldsymbol{\theta}) > 0$ при $(k,r) \in \widehat{\mathcal{K}}$. Положим
\begin{equation}
\label{hatc^circ}
\widehat{c}^{\circ} \coloneqq \min_{(k,r) \in \widehat{\mathcal{K}}} \widehat{c}^{\circ}_{kr}.
\end{equation}

Ясно, что число~(\ref{hatc^circ})~--- это реализация величины~(\ref{abstr_c^circ}), выбранная не зависящей от $\boldsymbol{\theta}$.
Число $\widehat{t}^{\,00}$, подчинённое~(\ref{abstr_t00}), при условии~\ref{cond1} также можно выбрать не зависящим от $\boldsymbol{\theta} \in \mathbb{S}^{d-1}$. С учётом~(\ref{delta_fixation}) и~(\ref{X_1_estimate}) (при $f = \mathbf{1}_n$) положим
\begin{equation*}
\widehat{t}^{\,00} = (8 \beta_2)^{-1} r_0 \alpha_1^{-3/2} \alpha_0^{1/2} \| g\|_{L_{\infty}}^{-3/2} \| g^{-1}\|_{L_{\infty}}^{-1/2} \widehat{c}^{\circ},
\end{equation*}
где $\widehat{c}^{\circ}$ определено в~(\ref{hatc^circ}). (Условие $\widehat{t}^{\,00} \le \widehat{t}^{\,0}$ выполнено автоматически, поскольку $\widehat{c}^{\circ} \le \| \widehat{S} (\boldsymbol{\theta}) \| \le \alpha_1 \|g\|_{L_{\infty}}$.)

\begin{remark}
	В отличие от числа $\widehat{t}^{\,0}$ \emph{(}см.~\emph{(\ref{t0_fixation})} при $f = \mathbf{1}_n$\emph{)}, которое контролируется только через $r_0$, $\alpha_0$, $\alpha_1$, $\|g\|_{L_{\infty}}$ и $\|g^{-1}\|_{L_{\infty}}$, величина $\widehat{t}^{\,00}$ зависит от спектральной характеристики ростка~--- минимального расстояния между его различными собственными значениями $\widehat{\gamma}_k (\boldsymbol{\theta})$ и $ \widehat{\gamma}_r (\boldsymbol{\theta})$ \emph{(}где $(k,r)$ пробегает $\widehat{\mathcal{K}}$\emph{)}.
\end{remark}

Применяя теорему~\ref{abstr_exp_enchcd_thrm_2}, получаем следующий результат.
\begin{thrm}
	\label{hatA(k)_exp_enchcd_thrm_2}
	Пусть выполнено условие~\emph{\ref{cond1}} \emph{(}или более сильное условие~\emph{\ref{cond2}}\emph{)}. Тогда при $\tau \in \mathbb{R}$, $\varepsilon > 0$ и $\mathbf{k} \in \widetilde{\Omega}$ выполнена оценка
	\begin{equation*}
	\| \widehat{J}(\mathbf{k}, \varepsilon; \tau ) \mathcal{R}(\mathbf{k}, \varepsilon)\|_{L_2(\Omega) \to L_2 (\Omega)}  \le \widehat{\mathcal{C}}_3 (1+ |\tau|^{1/2}) \varepsilon,
	\end{equation*}
	где константа $\widehat{\mathcal{C}}_3$ зависит от $\alpha_0$, $\alpha_1$, $\|g\|_{L_\infty}$, $\|g^{-1}\|_{L_\infty}$, $r_0$, а также от $n$ и $\widehat{c}^{\circ}$.
\end{thrm}

\subsection{Подтверждение точности относительно сглаживания}
Применение теорем~\ref{abstr_exp_smooth_shrp_thrm_1}, \ref{abstr_exp_smooth_shrp_thrm_2} позволяет подтвердить точность теорем~\ref{hatA(k)_exp_general_thrm}, \ref{hatA(k)_exp_enchcd_thrm_1}, \ref{hatA(k)_exp_enchcd_thrm_2} в отношении сглаживания.
\begin{thrm}[\cite{Su2017}]
	\label{hatA(k)_exp_smooth_shrp_thrm_1}
	Пусть $\widehat{N}_0 (\boldsymbol{\theta}_0) \ne 0$ при некотором $\boldsymbol{\theta}_0 \in \mathbb{S}^{d-1}$. Пусть $\tau \ne 0$ и $0 \le s < 3$. Тогда не существует такой константы $\mathcal{C}(\tau) >0$, чтобы оценка
	\begin{equation*}
	\bigl\| \bigl( e^{-i \tau \varepsilon^{-2} \widehat{\mathcal{A}} (\mathbf{k})}  - e^{-i \tau \varepsilon^{-2} \widehat{\mathcal{A}}^0 (\mathbf{k})} \bigr) \mathcal{R} (\mathbf{k}, \varepsilon)^{s/2} \bigr\|_{L_2(\Omega) \to L_2 (\Omega)} \le \mathcal{C}(\tau) \varepsilon
	\end{equation*}
	выполнялась при почти всех $\mathbf{k} = t \boldsymbol{\theta} \in \widetilde{\Omega}$ и достаточно малых $\varepsilon > 0$.
\end{thrm}

\begin{thrm}
	\label{hatA(k)_exp_smooth_shrp_thrm_2}
	Пусть $\widehat{N}_0 (\boldsymbol{\theta}) = 0$ при всех $\boldsymbol{\theta} \in \mathbb{S}^{d-1}$ и пусть $\widehat{\mathcal{N}}^{(q)} (\boldsymbol{\theta}_0) \ne 0$  при некоторых $q \in \{1,\ldots,p(\boldsymbol{\theta}_0)\}$ и $\boldsymbol{\theta}_0 \in \mathbb{S}^{d-1}$. 
	Пусть $\tau \ne 0$ и $0 \le s < 2$. Тогда не существует такой константы $\mathcal{C}(\tau) >0$, чтобы оценка
	\begin{equation*}
	\bigl\| \bigl( e^{-i \tau \varepsilon^{-2} \widehat{\mathcal{A}} (\mathbf{k})}  - e^{-i \tau \varepsilon^{-2} \widehat{\mathcal{A}}^0 (\mathbf{k})} \bigr) \mathcal{R} (\mathbf{k}, \varepsilon)^{s/2} \bigr\|_{L_2(\Omega) \to L_2 (\Omega)} \le \mathcal{C}(\tau) \varepsilon
	\end{equation*}
	выполнялась при почти всех $\mathbf{k} = t \boldsymbol{\theta} \in \widetilde{\Omega}$ и достаточно малых $\varepsilon > 0$.
\end{thrm}
Теорема~\ref{hatA(k)_exp_smooth_shrp_thrm_1} была доказана в~\cite[теорема~9.8]{Su2017}.

\subsection{Подтверждение точности относительно времени}
Применение теоремы~\ref{abstr_exp_time_shrp_thrm_1} позволяет подтвердить точность теоремы~\ref{hatA(k)_exp_general_thrm} в отношении зависимости оценки от времени.
\begin{thrm}
		\label{hatA(k)_exp_time_shrp_thrm_1}
		Пусть $\widehat{N}_0 (\boldsymbol{\theta}_0) \ne 0$ при некотором $\boldsymbol{\theta}_0 \in \mathbb{S}^{d-1}$. Тогда не существует положительной функции $\mathcal{C}(\tau)$ такой, что $\lim_{\tau \to \infty} \mathcal{C}(\tau)/ |\tau| = 0$ и выполнена оценка
		\begin{equation}
		\label{hatA(k)_exp_time_shrp_est_1}
		\bigl\| \bigl( e^{-i \tau \varepsilon^{-2} \widehat{\mathcal{A}} (\mathbf{k})}  - e^{-i \tau \varepsilon^{-2} \widehat{\mathcal{A}}^0 (\mathbf{k})} \bigr) \mathcal{R} (\mathbf{k}, \varepsilon)^{3/2} \bigr\|_{L_2(\Omega) \to L_2 (\Omega)} \le \mathcal{C}(\tau) \varepsilon
		\end{equation}
		при всех $\tau \in \mathbb{R}$, почти всех $\mathbf{k} = t \boldsymbol{\theta} \in \widetilde{\Omega}$ и достаточно малых $\varepsilon > 0$.
\end{thrm}

\begin{proof}[Доказательство\nopunct] проведём от противного.  Предположим, что найдётся функция $\mathcal{C}(\tau) > 0$ такая, что $\lim_{\tau \to \infty} \mathcal{C}(\tau)/ |\tau| = 0$ и выполнена оценка~(\ref{hatA(k)_exp_time_shrp_est_1}) при почти всех $\mathbf{k} \in \widetilde{\Omega}$ и достаточно малом $\varepsilon > 0$. Учитывая~(\ref{R_P}), (\ref{R(k,eps)(I-P)_est}), а также оценку
\begin{equation}
\label{hatA(k)_exp_shrp_f1}
\bigl\| \widehat{F} (\mathbf{k}) - \widehat{P} \bigr\|_{L_2 (\Omega) \to L_2 (\Omega)} \le \widehat{C}_1 |\mathbf{k}|, \qquad |\mathbf{k}| \le \widehat{t}^{\,0},
\end{equation}
(см.~(\ref{abstr_F(t)_threshold})), убеждаемся, что найдётся функция $\widetilde{\mathcal{C}}(\tau) > 0$ такая, что $\lim_{\tau \to \infty} \widetilde{\mathcal{C}}(\tau)/ |\tau| = 0$ и выполнена оценка
\begin{equation}
\label{hatA(k)_exp_shrp_f2}
\bigl\|  e^{-i \tau \varepsilon^{-2} \widehat{\mathcal{A}} ( \mathbf{k})} \widehat{F} (\mathbf{k})  - e^{-i \tau \varepsilon^{-2} \widehat{\mathcal{A}}^0 (\mathbf{k})} \widehat{P} \bigr\|_{L_2(\Omega) \to L_2 (\Omega) } \varepsilon^3 (|\mathbf{k}|^2 + \varepsilon^2)^{-3/2}  \le \widetilde{\mathcal{C}}(\tau) \varepsilon
\end{equation}
при почти всех $\mathbf{k} \in \widetilde{\Omega}$ в шаре $|\mathbf{k}| \le \widehat{t}^{\,0}$ и достаточно малых $\varepsilon > 0$.
Оператор, стоящий под знаком нормы в~(\ref{hatA(k)_exp_shrp_f2}), непрерывен по $\mathbf{k}$ в шаре $|\mathbf{k}| \le \widehat{t}^{\,0}$ при фиксированных $\tau$ и $\varepsilon$ (см.~\cite[лемма~9.9]{Su2017}). Следовательно, оценка~(\ref{hatA(k)_exp_shrp_f2}) справедлива при всех значениях $\mathbf{k}$ из данного шара. В частности, она верна в точке $\mathbf{k} = t\boldsymbol{\theta}_0$, если $t \le \widehat{t}^{\,0}$. Применяя снова неравенство~(\ref{hatA(k)_exp_shrp_f1}), получаем, что справедливо неравенство
\begin{equation}
\label{hatA(k)_exp_shrp_f3}
\bigl\| \bigl( e^{-i \tau \varepsilon^{-2} \widehat{\mathcal{A}} ( t \boldsymbol{\theta}_0)}  - e^{-i \tau \varepsilon^{-2} \widehat{\mathcal{A}}^0 (t \boldsymbol{\theta}_0)} \bigr) \widehat{P} \bigr\|_{L_2(\Omega) \to L_2(\Omega)} \varepsilon^3 (t^2 + \varepsilon^2)^{-3/2}  \le \check{\mathcal{C}}(\tau) \varepsilon
\end{equation}
c функцией $\check{\mathcal{C}}(\tau) > 0$ такой, что $\lim_{\tau \to \infty} \check{\mathcal{C}}(\tau)/ |\tau| = 0$,  при всех $t \le \widehat{t}^{\,0}$ и достаточно малых $\varepsilon > 0$. 

Оценка~(\ref{hatA(k)_exp_shrp_f3}) в абстрактных терминах соответствует оценке~(\ref{abstr_exp_time_shrp_est_1}). Поскольку по условию выполнено $\widehat{N}_{0}(\boldsymbol{\theta}_0) \ne 0$, то применение теоремы~\ref{abstr_exp_time_shrp_thrm_1} приводит нас к противоречию.
\end{proof}

Аналогично, применение теоремы~\ref{abstr_exp_time_shrp_thrm_2} позволяет подтвердить точность теорем~\ref{hatA(k)_exp_enchcd_thrm_1}, \ref{hatA(k)_exp_enchcd_thrm_2}.

\begin{thrm}
	\label{hatA(k)_exp_time_shrp_thrm_2}
	Пусть $\widehat{N}_0 (\boldsymbol{\theta}) = 0$ при всех $\boldsymbol{\theta} \in \mathbb{S}^{d-1}$ и пусть $\widehat{\mathcal{N}}^{(q)} (\boldsymbol{\theta}_0) \ne 0$  при некоторых $q \in \{1,\ldots,p(\boldsymbol{\theta}_0)\}$ и $\boldsymbol{\theta}_0 \in \mathbb{S}^{d-1}$. Тогда не существует положительной функции $\mathcal{C}(\tau)$ такой, что $\lim_{\tau \to \infty} \mathcal{C}(\tau)/ |\tau|^{1/2} = 0$ и выполнена оценка
	\begin{equation*}
	\bigl\| \bigl( e^{-i \tau \varepsilon^{-2} \widehat{\mathcal{A}} (\mathbf{k})}  - e^{-i \tau \varepsilon^{-2} \widehat{\mathcal{A}}^0 (\mathbf{k})} \bigr) \mathcal{R} (\mathbf{k}, \varepsilon) \bigr\|_{L_2(\Omega) \to L_2 (\Omega)} \le \mathcal{C}(\tau) \varepsilon
	\end{equation*}
	при всех $\tau \in \mathbb{R}$, почти всех $\mathbf{k} = t \boldsymbol{\theta} \in \widetilde{\Omega}$ и достаточно малых $\varepsilon > 0$.
\end{thrm}

\section{Оператор $\mathcal{A} (\mathbf{k})$. Применение схемы~\S\ref{abstr_sndw_section}}
\subsection{Оператор $\mathcal{A} (\mathbf{k})$}

Оператор $\mathcal{A} (\mathbf{k}) = f^* \widehat{\mathcal{A}} (\mathbf{k}) f$ изучается на основании схемы~\S\ref{abstr_sndw_section}. Сейчас $\mathfrak{H} = \widehat{\mathfrak{H}} = L_2 (\Omega; \mathbb{C}^n)$, $\mathfrak{H}_* = L_2 (\Omega; \mathbb{C}^m)$, роль оператора $A(t)$ играет $A(t, \boldsymbol{\theta}) = \mathcal{A}(\mathbf{k})$, роль оператора $\widehat{A}(t)$ играет $\widehat{A}(t, \boldsymbol{\theta}) = \widehat{\mathcal{A}}(\mathbf{k})$. В качестве изоморфизма $M$ выступает оператор умножения на матричнозначную функцию $f(\mathbf{x})$. Оператор $Q$ является оператором умножения на матрицу-функцию $Q(\mathbf{x}) = (f (\mathbf{x}) f (\mathbf{x})^*)^{-1}$. Блок оператора $Q$ в подпространстве $\widehat{\mathfrak{N}}$ (см.~(\ref{Ker3}))~--- это оператор умножения на постоянную матрицу $\overline{Q} = (\underline{f f^*})^{-1} = |\Omega|^{-1} \int_{\Omega} (f (\mathbf{x}) f (\mathbf{x})^*)^{-1} d \mathbf{x}$. Далее, $M_0$ есть оператор умножения на постоянную матрицу
\begin{equation}
\label{f0}
f_0 = (\overline{Q})^{-1/2} = (\underline{f f^*})^{1/2}.
\end{equation}
Отметим элементарные неравенства $| f_0 | \le \| f \|_{L_{\infty}}$, $| f_0^{-1} | \le \| f^{-1} \|_{L_{\infty}}$.

В $L_2 (\mathbb{R}^d; \mathbb{C}^n)$ определим оператор
\begin{equation}
\label{A0}
\mathcal{A}^0 \coloneqq f_0 \widehat{\mathcal{A}}^0 f_0 = f_0 b(\mathbf{D})^* g^0 b(\mathbf{D}) f_0.
\end{equation}
Пусть $\mathcal{A}^0 (\mathbf{k})$~--- соответствующее операторное семейство в $L_2 (\Omega; \mathbb{C}^n)$. Тогда $\mathcal{A}^0 (\mathbf{k}) = f_0 \widehat{\mathcal{A}}^0 (\mathbf{k}) f_0$. С учётом~(\ref{Ker3}) и~(\ref{hatS_P=hatA^0_P}) справедливо тождество 
\begin{equation}
\label{A^0(k)P}
f_0 \widehat{S} (\mathbf{k}) f_0 \widehat{P} = \mathcal{A}^0 (\mathbf{k}) \widehat{P}.
\end{equation}

\subsection{Аналитические ветви собственных значений и собственных элементов}
Согласно~(\ref{abstr_S_Shat}), спектральный росток $S(\boldsymbol{\theta})$ оператора $A (t, \boldsymbol{\theta})$, действующий в подпространстве $\mathfrak{N}$ (см.~(\ref{frakN})), представляется в виде $S(\boldsymbol{\theta}) = P f^* b(\boldsymbol{\theta})^* g^0 b(\boldsymbol{\theta}) f|_{\mathfrak{N}}$, где $P$~--- ортопроектор пространства $L_2 (\Omega; \mathbb{C}^n)$ на $\mathfrak{N}$.

Аналитические (по $t$) ветви собственных значений $\lambda_l (t, \boldsymbol{\theta})$ и собственных элементов $\varphi_l (t, \boldsymbol{\theta})$ оператора $A (t, \boldsymbol{\theta})$ допускают степенные разложения вида~(\ref{abstr_A(t)_eigenvalues_series}), (\ref{abstr_A(t)_eigenvectors_series}) с коэффициентами, зависящими от $\boldsymbol{\theta}$:
\begin{align}
\label{A_eigenvalues_series}
 \lambda_l (t, \boldsymbol{\theta}) &= \gamma_l (\boldsymbol{\theta}) t^2 + \mu_l (\boldsymbol{\theta}) t^3 + \nu_l (\boldsymbol{\theta}) t^4 + \ldots, & l &= 1, \ldots, n, 
\\
\label{A_eigenvectors_series}
\varphi_l (t, \boldsymbol{\theta}) &= \omega_l (\boldsymbol{\theta}) + t \psi^{(1)}_l (\boldsymbol{\theta}) + \ldots, & l &= 1, \ldots, n.
\end{align}

При этом $\omega_1 (\boldsymbol{\theta}), \ldots, \omega_n (\boldsymbol{\theta})$ образуют ортонормированный базис в подпространстве $\mathfrak{N}$, а векторы $\zeta_l (\boldsymbol{\theta}) = f \omega_l (\boldsymbol{\theta})$, $l = 1, \ldots, n$, образуют базис в $\widehat{\mathfrak{N}}$~(см.~(\ref{Ker3})), ортонормированный с весом $\overline{Q}$. Числа $\gamma_l (\boldsymbol{\theta})$ и элементы $\omega_l (\boldsymbol{\theta})$ являются собственными для спектрального ростка $S(\boldsymbol{\theta})$. Согласно~(\ref{abstr_hatS_gener_spec_problem}),
\begin{equation}
\label{hatS_gener_spec_problem}
b(\boldsymbol{\theta})^* g^0 b(\boldsymbol{\theta}) \zeta_l (\boldsymbol{\theta}) = \gamma_l (\boldsymbol{\theta}) \overline{Q} \zeta_l (\boldsymbol{\theta}), \qquad l = 1, \ldots, n.
\end{equation}

\subsection{Оператор $\widehat{N}_Q (\boldsymbol{\theta})$}
Нам понадобится описать оператор $\widehat{N}_Q$ (см. п.~\ref{abstr_hatZ_Q_and_hatN_Q_section}). Для этого введём $\Gamma$-периодическое решение $\Lambda_Q(\mathbf{x})$ задачи
\begin{equation*}
b(\mathbf{D})^* g(\mathbf{x}) (b(\mathbf{D}) \Lambda_Q(\mathbf{x}) + \mathbf{1}_m) = 0, \qquad \int_{\Omega} Q(\mathbf{x}) \Lambda_Q(\mathbf{x}) \, d \mathbf{x} = 0.
\end{equation*}
Ясно, что $\Lambda_Q(\mathbf{x}) = \Lambda(\mathbf{x}) - (\overline{Q})^{-1} (\overline{Q \Lambda})$. Как проверено в~\cite[\S5]{BSu2005-2}, оператор $\widehat{N}_Q (\boldsymbol{\theta})$ сейчас принимает вид
\begin{align}
\label{N_Q(theta)}
\widehat{N}_Q (\boldsymbol{\theta}) &= b(\boldsymbol{\theta})^* L_Q (\boldsymbol{\theta}) b(\boldsymbol{\theta}) \widehat{P},\\
\notag
L_Q (\boldsymbol{\theta}) &\coloneqq | \Omega |^{-1} \int_{\Omega} (\Lambda_Q(\mathbf{x})^*b(\boldsymbol{\theta})^* \widetilde{g} (\mathbf{x}) + \widetilde{g} (\mathbf{x})^* b(\boldsymbol{\theta}) \Lambda_Q(\mathbf{x}))\, d \mathbf{x}.
\end{align}
В~\cite[\S5]{BSu2005-2} указаны некоторые достаточные условия, при которых оператор~(\ref{N_Q(theta)}) обращается в ноль.
\begin{proposition}[\cite{BSu2005-2}]
	Пусть выполнено хотя бы одно из следующих предположений:
	\begin{enumerate}[label=\emph{\arabic*$^{\circ}.$}, ref=\arabic*$^{\circ}$, leftmargin=2.5\parindent]
		\setlength\itemsep{-0.1em}
		\item $\mathcal{A} = f(\mathbf{x})^*\mathbf{D}^* g(\mathbf{x}) \mathbf{D}f(\mathbf{x})$, где $g(\mathbf{x})$~--- симметричная матрица с вещественными элементами.
		\item Выполнены соотношения~\emph{(\ref{g0=overline_g_relat})}, т.~е. $g^0 = \overline{g}$.
	\end{enumerate}
	Тогда $\widehat{N}_Q (\boldsymbol{\theta}) = 0$ при всех $\boldsymbol{\theta} \in \mathbb{S}^{d-1}$.
\end{proposition}

Напомним (см.~п.~\ref{abstr_hatZ_Q_and_hatN_Q_section}), что  $\widehat{N}_Q (\boldsymbol{\theta}) = \widehat{N}_{0, Q} (\boldsymbol{\theta}) + \widehat{N}_{*,Q} (\boldsymbol{\theta})$. Согласно~(\ref{abstr_hatN_0Q_N_*Q}),
\begin{equation*}
\widehat{N}_{0, Q} (\boldsymbol{\theta}) = \sum_{l=1}^{n} \mu_l (\boldsymbol{\theta}) (\cdot, \overline{Q} \zeta_l(\boldsymbol{\theta}))_{L_2(\Omega)} \overline{Q} \zeta_l(\boldsymbol{\theta}).
\end{equation*}
При этом
\begin{equation*}
(\widehat{N}_Q (\boldsymbol{\theta}) \zeta_l (\boldsymbol{\theta}), \zeta_l (\boldsymbol{\theta}))_{L_2 (\Omega)} = (\widehat{N}_{0,Q} (\boldsymbol{\theta}) \zeta_l (\boldsymbol{\theta}), \zeta_l (\boldsymbol{\theta}))_{L_2 (\Omega)} = \mu_l (\boldsymbol{\theta}), \qquad l=1, \ldots, n.
\end{equation*}
В~\cite[предложение~5.2]{BSu2005-2} было доказано следующее утверждение.
\begin{proposition}[\cite{BSu2005-2}]
	Пусть $b(\boldsymbol{\theta})$, $g (\mathbf{x})$ и $Q(\mathbf{x})$~--- матрицы с вещественными элементами. Пусть в разложениях~\emph{(\ref{A_eigenvectors_series})}  \textquotedblleft зародыши\textquotedblright \ $\omega_l (\boldsymbol{\theta})$, $l = 1, \ldots, n$, можно выбрать так, чтобы векторы  $\zeta_l (\boldsymbol{\theta}) = f \omega_l (\boldsymbol{\theta})$ оказались вещественными. Тогда в~\emph{(\ref{A_eigenvalues_series})} выполнено $\mu_l (\boldsymbol{\theta}) = 0$, $l=1, \ldots, n$, то есть, $\widehat{N}_{0,Q} (\boldsymbol{\theta}) = 0$ при всех $\boldsymbol{\theta} \in \mathbb{S}^{d-1}$.
\end{proposition}

В рассматриваемом \textquotedblleft вещественном\textquotedblright \ случае $\widehat{S} (\boldsymbol{\theta})$ и $\overline{Q}$  являются симметричными вещественными матрицами. Ясно, что в случае простого собственного значения $\gamma_j (\boldsymbol{\theta})$ обобщённой задачи~(\ref{hatS_gener_spec_problem}) собственный вектор $\zeta_j (\boldsymbol{\theta}) = f \omega_j (\boldsymbol{\theta})$ определяется однозначно с точностью до фазового множителя, и его всегда можно выбрать вещественным. Мы получаем следующее следствие.

\begin{corollary}
	\label{sndw_real_spec_simple_coroll}
	Пусть $b(\boldsymbol{\theta})$, $g (\mathbf{x})$ и $Q(\mathbf{x})$~--- матрицы с вещественными элементами. Пусть обобщённая спектральная задача~\emph{(\ref{hatS_gener_spec_problem})} имеет простой спектр. Тогда $\widehat{N}_{0,Q} (\boldsymbol{\theta}) = 0$ при всех $\boldsymbol{\theta} \in \mathbb{S}^{d-1}$.
\end{corollary}

\subsection{Операторы $\widehat{Z}_{2,Q}(\boldsymbol{\theta})$, $\widehat{R}_{2,Q}(\boldsymbol{\theta})$, $\widehat{N}_{1,Q}^0(\boldsymbol{\theta})$}
Опишем операторы $\widehat{Z}_{2,Q}$, $\widehat{R}_{2,Q}$, $\widehat{N}_{1,Q}^0$ в абстрактных терминах определённые в п.~\ref{abstr_hatZ2_Q_hatR2_Q_N1^0_Q_section}. Пусть $\Lambda^{(2)}_{Q,l} (\mathbf{x})$~--- $\Gamma$-периодическое решение задачи
\begin{equation*}
b(\mathbf{D})^* g(\mathbf{x}) (b(\mathbf{D}) \Lambda^{(2)}_{Q,l} (\mathbf{x}) + b_l \Lambda_{Q}(\mathbf{x})) = -b_l^* \widetilde{g} (\mathbf{x}) + Q(\mathbf{x}) (\overline{Q})^{-1} b_l^* g^0, \qquad \int_{\Omega} Q(\mathbf{x}) \Lambda^{(2)}_{Q,l} (\mathbf{x}) \, d \mathbf{x} = 0.
\end{equation*}
Положим $\Lambda^{(2)}_Q (\mathbf{x}; \boldsymbol{\theta}) \coloneqq \sum_{l=1}^{d} \Lambda^{(2)}_{Q,l} (\mathbf{x}) \theta_l$. Как проверено в~\cite[п.~8.4]{VSu2012}
\begin{gather*}
\widehat{Z}_{2,Q}(\boldsymbol{\theta}) = \Lambda^{(2)} (\mathbf{x}; \boldsymbol{\theta}) b(\boldsymbol{\theta}) \widehat{P}, \qquad
\widehat{R}_{2,Q}(\boldsymbol{\theta}) = h(\mathbf{x}) (b(\mathbf{D}) \Lambda^{(2)}_Q (\mathbf{x}; \boldsymbol{\theta}) + b(\boldsymbol{\theta}) \Lambda_Q(\mathbf{x})) b(\boldsymbol{\theta}).
\end{gather*}
Наконец, в~\cite[п.~8.5]{VSu2012} было получено представление
\begin{align*}
&\widehat{N}_{1,Q}^0(\boldsymbol{\theta}) = b(\boldsymbol{\theta})^* L_{2,Q}(\boldsymbol{\theta}) b(\boldsymbol{\theta}) \widehat{P}, \\
&\begin{multlined}[c][15cm]
L_{2,Q}(\boldsymbol{\theta}) \coloneqq |\Omega|^{-1} \int_{\Omega} \bigl(\Lambda^{(2)}_Q (\mathbf{x}; \boldsymbol{\theta})^* b(\boldsymbol{\theta})^* \widetilde{g}(\mathbf{x}) + \widetilde{g}(\mathbf{x})^* b(\boldsymbol{\theta}) \Lambda^{(2)}_Q (\mathbf{x}; \boldsymbol{\theta}) \bigr) \, d \mathbf{x} + \\ +
|\Omega|^{-1} \int_{\Omega} \bigl(b(\mathbf{D}) \Lambda^{(2)}_Q (\mathbf{x}; \boldsymbol{\theta}) + b(\boldsymbol{\theta}) \Lambda_Q (\mathbf{x})\bigr)^* g(\mathbf{x}) \bigl(b(\mathbf{D}) \Lambda^{(2)}_Q (\mathbf{x}; \boldsymbol{\theta}) + b(\boldsymbol{\theta}) \Lambda_Q(\mathbf{x}) \bigr) \, d \mathbf{x}.
\end{multlined}
\end{align*}

\subsection{Кратности собственных значений ростка}
\label{sndw_eigenval_multipl_section}
В данном пункте считаем, что $n \ge 2$. Перейдём к обозначениям, принятым в~п.~\ref{abstr_cluster_section}. Вообще говоря, количество $p(\boldsymbol{\theta})$ различных собственных значений $\gamma^{\circ}_1 (\boldsymbol{\theta}), \ldots, \gamma^{\circ}_{p(\boldsymbol{\theta})} (\boldsymbol{\theta})$ спектрального ростка $S(\boldsymbol{\theta})$ (или задачи~(\ref{hatS_gener_spec_problem})) и их кратности $k_1 (\boldsymbol{\theta}), \ldots, k_{p(\boldsymbol{\theta})} (\boldsymbol{\theta})$ зависят от параметра $\boldsymbol{\theta} \in \mathbb{S}^{d-1}$. При каждом фиксированном $\boldsymbol{\theta}$ через $\mathfrak{N}_j (\boldsymbol{\theta})$ обозначим собственное подпространство ростка $S (\boldsymbol{\theta})$, отвечающее собственному значению $\gamma^{\circ}_j (\boldsymbol{\theta})$. Тогда $f \mathfrak{N}_j (\boldsymbol{\theta})$~--- собственное подпространство задачи~(\ref{hatS_gener_spec_problem}), отвечающее тому же значению $\gamma^{\circ}_j (\boldsymbol{\theta})$. Введём обозначение $\mathcal{P}_j (\boldsymbol{\theta})$ для \textquotedblleft косого\textquotedblright \ проектора пространства $L_2(\Omega; \mathbb{C}^n)$ на подпространство $f \mathfrak{N}_j (\boldsymbol{\theta})$; $\mathcal{P}_j (\boldsymbol{\theta})$ ортогонален относительно скалярного произведения с весом $\overline{Q}$. Согласно~(\ref{abstr_hatN_0Q_N_*Q_invar_repr}),
\begin{equation*}
\widehat{N}_{0,Q} (\boldsymbol{\theta}) = \sum_{j=1}^{p(\boldsymbol{\theta})} \mathcal{P}_j (\boldsymbol{\theta})^* \widehat{N}_Q (\boldsymbol{\theta}) \mathcal{P}_j (\boldsymbol{\theta}), \qquad \widehat{N}_{*,Q} (\boldsymbol{\theta}) = \sum_{\substack{1 \le l,j \le p(\boldsymbol{\theta})\\ j \ne l}} \mathcal{P}_j (\boldsymbol{\theta})^* \widehat{N}_Q (\boldsymbol{\theta}) \mathcal{P}_l (\boldsymbol{\theta}).
\end{equation*}

\subsection{Коэффициенты $\nu_l(\boldsymbol{\theta})$, $l=1, \ldots,n$}
Согласно~(\ref{abstr_N_eigenvalues}), числа $\mu_l (\boldsymbol{\theta})$ и элементы $\omega_l (\boldsymbol{\theta})$, $l = i(q,\boldsymbol{\theta}),\ldots, i(q,\boldsymbol{\theta})+k_q(\boldsymbol{\theta})-1$, где $i(q,\boldsymbol{\theta}) = k_1(\boldsymbol{\theta})+\ldots+k_{q-1}(\boldsymbol{\theta})+1$,  являются собственными для оператора $P_q (\boldsymbol{\theta}) N(\boldsymbol{\theta}) |_{\mathfrak{N}_q(\boldsymbol{\theta})}$. Тогда, в силу~(\ref{abstr_hatN_Q_gener_spec_problem}),
\begin{equation}
\label{hatN_Q_gener_spec_problem}
\widehat{P}_{f \mathfrak{N}_q(\boldsymbol{\theta})} (\boldsymbol{\theta}) \widehat{N}_Q (\boldsymbol{\theta}) \zeta_l (\boldsymbol{\theta}) = \mu_l (\boldsymbol{\theta}) \widehat{P}_{f \mathfrak{N}_q(\boldsymbol{\theta})} \overline{Q} \zeta_l (\boldsymbol{\theta}), \qquad l = i(q,\boldsymbol{\theta}),\ldots, i(q,\boldsymbol{\theta})+k_q(\boldsymbol{\theta})-1,
\end{equation} 
где $\widehat{P}_{f \mathfrak{N}_q(\boldsymbol{\theta})}(\boldsymbol{\theta})$~---  ортопроектор на $f \mathfrak{N}_q (\boldsymbol{\theta})$.

Количество $p'(q, \boldsymbol{\theta})$ различных собственных чисел $\mu^{\circ}_{1,q} (\boldsymbol{\theta}), \ldots, \mu^{\circ}_{p'(q, \boldsymbol{\theta}),q} (\boldsymbol{\theta})$ оператора $P_q (\boldsymbol{\theta}) N(\boldsymbol{\theta}) |_{\mathfrak{N}_q(\boldsymbol{\theta})}$ 
и их кратности $k_{1,q} (\boldsymbol{\theta}), \ldots, k_{p'(\boldsymbol{\theta}),q} (\boldsymbol{\theta})$ зависят от параметра $\boldsymbol{\theta} \in \mathbb{S}^{d-1}$.
При каждом фиксированном $\boldsymbol{\theta}$ через $\mathfrak{N}_{q',q} (\boldsymbol{\theta})$ обозначим собственное подпространство, отвечающее собственному значению $\mu^{\circ}_{q',q} (\boldsymbol{\theta})$. Тогда $f \mathfrak{N}_{q',q} (\boldsymbol{\theta})$~--- собственное подпространство задачи~(\ref{hatN_Q_gener_spec_problem}), отвечающее тому же значению $\mu^{\circ}_{q',q} (\boldsymbol{\theta})$. 

Наконец, согласно~(\ref{abstr_scrNhat_M^(q)_gener_spec_problem}), числа $\nu_l (\boldsymbol{\theta})$ и элементы $\zeta_l (\boldsymbol{\theta})$, $l = i'(q',q,\boldsymbol{\theta}),\ldots, i'(q',q,\boldsymbol{\theta})+k_{q',q}(\boldsymbol{\theta})-1$, где $i'(q',q,\boldsymbol{\theta}) = i(q, \boldsymbol{\theta}) +k_{1,q}(\boldsymbol{\theta})+\ldots+k_{q'-1,q}(\boldsymbol{\theta})$, являются собственными значениями и собственными элементами следующей обобщённой спектральной задачи:
\begin{equation*}
\widehat{\mathcal{N}}_Q^{(q',q)}(\boldsymbol{\theta}) \zeta_l(\boldsymbol{\theta}) = \nu_l(\boldsymbol{\theta}) \widehat{P}_{f \mathfrak{N}_{q',q} (\boldsymbol{\theta})} \overline{Q} \zeta_l(\boldsymbol{\theta}), \qquad l = i'(q',q,\boldsymbol{\theta}),\ldots, i'(q',q,\boldsymbol{\theta})+k_q(\boldsymbol{\theta})-1,
\end{equation*}
где 
\begin{multline*}
\widehat{\mathcal{N}}_Q^{(q',q)}(\boldsymbol{\theta}) \coloneqq \\ \widehat{P}_{f \mathfrak{N}_{q',q} (\boldsymbol{\theta})}(\boldsymbol{\theta}) \left. \left(
\widehat{N}_{1,Q}^0(\boldsymbol{\theta}) - \frac{1}{2} \widehat{Z}_Q^*(\boldsymbol{\theta}) Q \widehat{Z}_Q(\boldsymbol{\theta}) (f f^*) \widehat{S}(\boldsymbol{\theta}) \widehat{P} - \frac{1}{2} \widehat{S}(\boldsymbol{\theta}) (ff^*) \widehat{Z}_Q(\boldsymbol{\theta})^* Q \widehat{Z}_Q(\boldsymbol{\theta}) \right)\right|_{f \mathfrak{N}_{q',q} (\boldsymbol{\theta})} + \\ +
\sum_{\substack{j\in\{1,\ldots,p(\boldsymbol{\theta})\} \\ j \ne q}} \bigl(\gamma^{\circ}_q(\boldsymbol{\theta}) - \gamma^{\circ}_j(\boldsymbol{\theta})\bigr)^{-1} \widehat{P}_{f \mathfrak{N}_{q',q}(\boldsymbol{\theta})} \widehat{N}_Q(\boldsymbol{\theta}) \widehat{P}_{f \mathfrak{N}_j(\boldsymbol{\theta})}(\boldsymbol{\theta}) (f f^*) \widehat{P}_{f \mathfrak{N}_j(\boldsymbol{\theta})}(\boldsymbol{\theta}) \widehat{N}_Q(\boldsymbol{\theta})|_{f \mathfrak{N}_{q',q}(\boldsymbol{\theta})}.
\end{multline*}

Отметим, что в случае, когда $\widehat{N}_{0,Q}(\boldsymbol{\theta}) = 0$ имеет место $f \mathfrak{N}_{1,q} (\boldsymbol{\theta}) = f \mathfrak{N}_{q} (\boldsymbol{\theta})$, $q=1, \ldots,p(\boldsymbol{\theta})$. Тогда вместо $\widehat{\mathcal{N}}_Q^{(1,q)}(\boldsymbol{\theta})$ мы будем писать $\widehat{\mathcal{N}}_Q^{(q)}(\boldsymbol{\theta})$. 

\section{Аппроксимация окаймлённого оператора $e^{-i \tau \varepsilon^{-2}  \mathcal{A}(\mathbf{k})}$}

\subsection{Общий случай}
Обозначим 
\begin{equation}
\label{J(k,eps)}
J (\mathbf{k}, \varepsilon; \tau) \coloneqq f e^{-i \tau \varepsilon^{-2} \mathcal{A} (\mathbf{k})} f^{-1} - f_0 e^{-i \tau \varepsilon^{-2} \mathcal{A}^0 (\mathbf{k})} f_0^{-1}.
\end{equation}

Мы применим к оператору $A(t; \boldsymbol{\theta}) = \mathcal{A}(\mathbf{k})$ теоремы из п.~\ref{abstr_sndw_exp_section}. При этом мы можем отследить зависимость постоянных в оценках от исходных данных. Отметим что $c_*$, $\delta$ и $t^0$ не зависят от $\boldsymbol{\theta}$ (см.~(\ref{c_*}), (\ref{delta_fixation}), (\ref{t0_fixation})). Согласно~(\ref{X_1_estimate}) норму $\| X_1 (\boldsymbol{\theta}) \|$ можно заменить на $\alpha_1^{1/2} \| g \|_{L_{\infty}}^{1/2} \| f \|_{L_{\infty}}$. Поэтому постоянные в теоремах~\ref{abstr_sndw_exp_general_thrm} и~\ref{abstr_sndw_exp_enchcd_thrm_1} (применённых к оператору $\mathcal{A}(\mathbf{k})$) не будут зависеть от $\boldsymbol{\theta}$. Они будут зависеть только от следующих величин: $\alpha_0$, $\alpha_1$, $\|g\|_{L_\infty}$, $\|g^{-1}\|_{L_\infty}$, $\|f\|_{L_{\infty}}$, $\|f^{-1}\|_{L_{\infty}}$ и $r_0$.

Применяя теорему~\ref{abstr_sndw_exp_general_thrm} с учётом~(\ref{R_P})--(\ref{R(k,eps)(I-P)_est}), (\ref{A^0(k)P})  получаем следующий результат, ранее доказанный в~\cite[теорема~8.1]{BSu2008}
\begin{thrm}[\cite{BSu2008}]
	\label{sndw_A(k)_exp_general_thrm}
	При $\tau \in \mathbb{R}$, $\varepsilon > 0$ и $\mathbf{k} \in \widetilde{\Omega}$ выполнена оценка
	\begin{equation*}
	\| J (\mathbf{k}, \varepsilon; \tau) \mathcal{R}(\mathbf{k}, \varepsilon)^{3/2}\|_{L_2(\Omega) \to L_2 (\Omega) }  \le \mathcal{C}_1 (1 + |\tau|) \varepsilon,
	\end{equation*}
	где константа $\mathcal{C}_1$ зависит только от $\alpha_0$, $\alpha_1$, $\|g\|_{L_\infty}$, $\|g^{-1}\|_{L_\infty}$, $\|f\|_{L_{\infty}}$, $\|f^{-1}\|_{L_{\infty}}$ и $r_0$.
\end{thrm}

\subsection{Случай, когда $\widehat{N}_Q(\boldsymbol{\theta}) = 0$}
Применим теорему~\ref{abstr_sndw_exp_enchcd_thrm_1}, предполагая, что $\widehat{N}_Q(\boldsymbol{\theta}) = 0$ при всех $\boldsymbol{\theta} \in \mathbb{S}^{d-1}$. С учётом~(\ref{R_P})--(\ref{R(k,eps)(I-P)_est}), (\ref{A^0(k)P}) это влечёт следующий результат.
\begin{thrm}
	\label{sndw_A(k)_exp_enchcd_thrm_1}
	Пусть оператор $\widehat{N}_Q(\boldsymbol{\theta})$ определён в~\emph{(\ref{N_Q(theta)})}. Пусть $\widehat{N}_Q(\boldsymbol{\theta}) = 0$ при всех $\boldsymbol{\theta} \in \mathbb{S}^{d-1}$. Тогда при $\tau \in \mathbb{R}$, $\varepsilon > 0$ и $\mathbf{k} \in \widetilde{\Omega}$ выполнена оценка
	\begin{equation*}
	\| J (\mathbf{k}, \varepsilon; \tau) \mathcal{R}(\mathbf{k}, \varepsilon)\|_{L_2(\Omega) \to L_2 (\Omega)}  \le \mathcal{C}_2 (1+ |\tau|^{1/2}) \varepsilon,
	\end{equation*}
	где константа $\mathcal{C}_2$ зависит только от $\alpha_0$, $\alpha_1$, $\|g\|_{L_\infty}$, $\|g^{-1}\|_{L_\infty}$, $\|f\|_{L_{\infty}}$, $\|f^{-1}\|_{L_{\infty}}$ и $r_0$.
\end{thrm}

\subsection{Случай, когда $\widehat{N}_{0,Q}(\boldsymbol{\theta}) = 0$} Теперь мы отказываемся от предположения теоремы~\ref{sndw_A(k)_exp_enchcd_thrm_1}, но взамен предположим, что $\widehat{N}_{0,Q}(\boldsymbol{\theta}) = 0$ при всех $\boldsymbol{\theta}$. Как и в п.~\ref{ench_approx2_section}, для того, чтобы применить теорему~\ref{abstr_sndw_exp_enchcd_thrm_2}, приходится накладывать дополнительные условия. Используем исходную нумерацию собственных значений $\gamma_1 (\boldsymbol{\theta}) \le \ldots \le \gamma_n (\boldsymbol{\theta})$ ростка $S (\boldsymbol{\theta})$. Они также являются собственными значениями обобщённой спектральной задачи~(\ref{hatS_gener_spec_problem}). При каждом $\boldsymbol{\theta}$ через $\mathcal{P}^{(k)} (\boldsymbol{\theta})$ обозначим \textquotedblleft косой\textquotedblright \ (ортогональный с весом $\overline{Q}$) проектор пространства $L_2 (\Omega; \mathbb{C}^n)$ на собственное подпространство задачи~(\ref{hatS_gener_spec_problem}), отвечающее собственному значению $\gamma_k (\boldsymbol{\theta})$. Ясно, что при каждом $\boldsymbol{\theta}$ оператор $\mathcal{P}^{(k)} (\boldsymbol{\theta})$ совпадает с одним из проекторов $\mathcal{P}_j (\boldsymbol{\theta})$, введённых в п.~\ref{sndw_eigenval_multipl_section} (но номер $j$ может зависеть от $\boldsymbol{\theta}$).

\begin{condition}
	\label{sndw_cond1}
	\begin{enumerate*}[label=\emph{\arabic*$^{\circ}.$}, ref=\arabic*$^{\circ}$]
		\item  $\widehat{N}_{0,Q}(\boldsymbol{\theta})=0$ при всех $\boldsymbol{\theta} \in \mathbb{S}^{d-1}$.
		\item \label{sndw_cond1_it2} Для каждой пары индексов $(k,r)$, $1 \le k,r \le n$, $k \ne r$, такой, что $\gamma_k (\boldsymbol{\theta}_0) = \gamma_r (\boldsymbol{\theta}_0) $ при некотором $\boldsymbol{\theta}_0 \in \mathbb{S}^{d-1}$, выполнено $(\mathcal{P}^{(k)} (\boldsymbol{\theta}))^* \widehat{N}_Q (\boldsymbol{\theta}) \mathcal{P}^{(r)} (\boldsymbol{\theta}) = 0$ при всех $\boldsymbol{\theta} \in \mathbb{S}^{d-1}$.    
	\end{enumerate*}
\end{condition}

Условие~\ref{sndw_cond1_it2} может быть переформулировано: мы требуем, чтобы для ненулевых (тождественно) \textquotedblleft блоков\textquotedblright \ $(\mathcal{P}^{(k)} (\boldsymbol{\theta}))^* \widehat{N}_Q (\boldsymbol{\theta}) \mathcal{P}^{(r)} (\boldsymbol{\theta})$ оператора $\widehat{N}_Q (\boldsymbol{\theta})$ соответствующие ветви собственных значений $\gamma_k (\boldsymbol{\theta})$ и  $\gamma_r (\boldsymbol{\theta})$ не пересекались.

Разумеется, выполнение условия~\ref{sndw_cond1} гарантируется следующим более сильным условием.

\begin{condition}
	\label{sndw_cond2}
	\begin{enumerate*}[label=\emph{\arabic*$^{\circ}.$}, ref=\arabic*$^{\circ}$]
		\item $\widehat{N}_{0,Q}(\boldsymbol{\theta})=0$ при всех $\boldsymbol{\theta} \in \mathbb{S}^{d-1}$.
		\item \label{sndw_cond2_it2} Предположим, что количество $p$ различных собственных значений обобщённой спектральной задачи~\emph{(\ref{hatS_gener_spec_problem})} не зависит от $\boldsymbol{\theta} \in \mathbb{S}^{d-1}$.       
	\end{enumerate*}
\end{condition}

При условии~\ref{sndw_cond2} обозначим различные собственные значения ростка, занумерованные в порядке возрастания, через $\gamma^{\circ}_1(\boldsymbol{\theta}), \ldots, \gamma^{\circ}_p(\boldsymbol{\theta})$. Тогда из их кратности $k_1, \ldots, k_p$ не зависят от $\boldsymbol{\theta} \in \mathbb{S}^{d-1}$.  

\begin{remark}
	\begin{enumerate*}[label=\emph{\arabic*$^{\circ}.$}, ref=\arabic*$^{\circ}$]
		\item Предположение пункта~\emph{\ref{sndw_cond2_it2}} условия~\emph{\ref{sndw_cond2}} заведомо выполнено, если спектр задачи~\emph{(\ref{hatS_gener_spec_problem})} простой при всех $\boldsymbol{\theta} \in \mathbb{S}^{d-1}$.
		\item Из следствия~\emph{\ref{sndw_real_spec_simple_coroll}} вытекает, что условие~\emph{\ref{sndw_cond2}} выполнено, если $b (\boldsymbol{\theta})$, $g (\mathbf{x})$ и $Q (\mathbf{x})$~--- матрицы с вещественными элементами и спектр задачи~\emph{(\ref{hatS_gener_spec_problem})} простой при всех $\boldsymbol{\theta} \in \mathbb{S}^{d-1}$.
	\end{enumerate*}
\end{remark}

Итак, предположим выполненным условие~\ref{sndw_cond1} и введём обозначение 
\begin{equation*}
\mathcal{K} \coloneqq \{ (k,r) \colon 1 \le k,r \le n, \; k \ne r, \;  (\mathcal{P}^{(k)} (\boldsymbol{\theta}))^* \widehat{N}_Q (\boldsymbol{\theta}) \mathcal{P}^{(r)} (\boldsymbol{\theta}) \not\equiv 0 \}.
\end{equation*}
Обозначим $c^{\circ}_{kr} (\boldsymbol{\theta}) \coloneqq \min \{c_*, n^{-1} |\gamma_k (\boldsymbol{\theta}) - \gamma_r (\boldsymbol{\theta})| \}$, $(k,r) \in \mathcal{K}$.

Поскольку оператор $S (\boldsymbol{\theta})$ непрерывно зависит от $\boldsymbol{\theta} \in \mathbb{S}^{d-1}$, то из теории возмущений дискретного спектра следует, что $\gamma_j (\boldsymbol{\theta})$~--- непрерывные функции на $\mathbb{S}^{d-1}$. В силу условия~\ref{sndw_cond1}(\ref{sndw_cond1_it2}) при $(k,r) \in \mathcal{K}$ выполнено~$|\gamma_k (\boldsymbol{\theta}) - \gamma_r (\boldsymbol{\theta})| > 0$ при всех $\boldsymbol{\theta} \in \mathbb{S}^{d-1}$, а тогда $c^{\circ}_{kr} \coloneqq \min_{\boldsymbol{\theta} \in \mathbb{S}^{d-1}} c^{\circ}_{kr} (\boldsymbol{\theta}) > 0$, $(k,r) \in \mathcal{K}$. Положим
\begin{equation}
\label{c^circ}
c^{\circ} \coloneqq \min_{(k,r) \in \mathcal{K}} c^{\circ}_{kr}.
\end{equation}
Ясно, что число~(\ref{c^circ})~--- это реализация величины~(\ref{abstr_c^circ}), выбранная не зависящей от $\boldsymbol{\theta}$.
Число, подчинённое~(\ref{abstr_t00}), при условии~\ref{sndw_cond1} также можно выбрать не зависящим от $\boldsymbol{\theta} \in \mathbb{S}^{d-1}$. С учётом~(\ref{delta_fixation}) и~(\ref{X_1_estimate}) положим
\begin{equation*}
t^{00} = (8 \beta_2)^{-1} r_0 \alpha_1^{-3/2} \alpha_0^{1/2} \| g\|_{L_{\infty}}^{-3/2} \| g^{-1}\|_{L_{\infty}}^{-1/2} \|f\|_{L_\infty}^{-3} \|f^{-1}\|_{L_\infty}^{-1} c^{\circ}.
\end{equation*}
(Условие $t^{00} \le t^{0}$ выполнено автоматически, поскольку $c^{\circ} \le \| S (\boldsymbol{\theta}) \| \le \alpha_1 \|g\|_{L_{\infty}}\|f\|_{L_\infty}^2$.)

Предполагая выполненным условие~\ref{sndw_cond1}, применим теорему~\ref{abstr_sndw_exp_enchcd_thrm_2}.
\begin{thrm}
	\label{sndw_A(k)_exp_enchcd_thrm_2}
	Пусть выполнено условие~\emph{\ref{sndw_cond1}} \emph{(}или более сильное условие~\emph{\ref{sndw_cond2}}\emph{)}. Тогда при $\tau \in \mathbb{R}$, $\varepsilon > 0$ и $\mathbf{k} \in \widetilde{\Omega}$ справедлива оценка
	\begin{equation*}
	\| J (\mathbf{k}, \varepsilon; \tau) \mathcal{R}(\mathbf{k}, \varepsilon)\|_{L_2(\Omega) \to L_2 (\Omega)}  \le \mathcal{C}_3 (1+ |\tau|^{1/2}) \varepsilon,
	\end{equation*}
	где константа $\mathcal{C}_3$ зависит от $\alpha_0$, $\alpha_1$, $\|g\|_{L_\infty}$, $\|g^{-1}\|_{L_\infty}$, $\|f\|_{L_{\infty}}$, $\|f^{-1}\|_{L_{\infty}}$, $r_0$, а также от $n$ и $c^{\circ}$.
\end{thrm}

\subsection{Подтверждение точности относительно сглаживания}
Применение теорем~\ref{abstr_sndw_exp_smooth_shrp_thrm_1}, \ref{abstr_sndw_exp_smooth_shrp_thrm_2} позволяет подтвердить точность теорем~\ref{sndw_A(k)_exp_general_thrm}, \ref{sndw_A(k)_exp_enchcd_thrm_1}, \ref{sndw_A(k)_exp_enchcd_thrm_2} в отношении сглаживания.
\begin{thrm}[\cite{Su2017}]
	\label{sndw_A(k)_exp_smooth_shrp_thrm_1}
	Пусть $\widehat{N}_{0,Q} (\boldsymbol{\theta}_0) \ne 0$ при некотором $\boldsymbol{\theta}_0 \in \mathbb{S}^{d-1}$. Пусть $\tau \ne 0$ и $0 \le s < 3$. Тогда не существует такой константы $\mathcal{C}(\tau) >0$, чтобы оценка
	\begin{equation*}
	\bigl\| \bigl( f e^{-i \tau \varepsilon^{-2} \mathcal{A} (\mathbf{k})} f^{-1} - f_0 e^{-i \tau \varepsilon^{-2} \mathcal{A}^0 (\mathbf{k})} f_0^{-1} \bigr) \mathcal{R} (\mathbf{k}, \varepsilon)^{s/2} \bigr\|_{L_2(\Omega) \to L_2 (\Omega)} \le \mathcal{C}(\tau) \varepsilon
	\end{equation*}
	выполнялась при почти всех $\mathbf{k} = t \boldsymbol{\theta} \in \widetilde{\Omega}$ и достаточно малых $\varepsilon > 0$.
\end{thrm}

\begin{thrm}
	\label{sndw_A(k)_exp_smooth_shrp_thrm_2}
	Пусть $\widehat{N}_{0,Q} (\boldsymbol{\theta}) = 0$ при всех $\boldsymbol{\theta} \in \mathbb{S}^{d-1}$ и пусть $\widehat{\mathcal{N}}_Q^{(q)} (\boldsymbol{\theta}_0) \ne 0$  при некоторых $q \in \{1,\ldots,p(\boldsymbol{\theta}_0)\}$ и $\boldsymbol{\theta}_0 \in \mathbb{S}^{d-1}$.
	Пусть $\tau \ne 0$ и $0 \le s < 2$. Тогда не существует такой константы $\mathcal{C}(\tau) >0$, чтобы оценка
	\begin{equation*}
	\bigl\| \bigl( f e^{-i \tau \varepsilon^{-2} \mathcal{A} (\mathbf{k})} f^{-1} - f_0 e^{-i \tau \varepsilon^{-2} \mathcal{A}^0 (\mathbf{k})} f_0^{-1} \bigr) \mathcal{R} (\mathbf{k}, \varepsilon)^{s/2} \bigr\|_{L_2(\Omega) \to L_2 (\Omega)} \le \mathcal{C}(\tau) \varepsilon
	\end{equation*}
	выполнялась при почти всех $\mathbf{k} = t \boldsymbol{\theta} \in \widetilde{\Omega}$ и достаточно малых $\varepsilon > 0$.
\end{thrm}
Теорема~\ref{sndw_A(k)_exp_smooth_shrp_thrm_1} была доказана в~\cite[теорема~11.7]{Su2017}.

\subsection{Подтверждение точности относительно времени}
Применение теоремы~\ref{abstr_sndw_exp_time_shrp_thrm_1} позволяет подтвердить точность теоремы~\ref{sndw_A(k)_exp_general_thrm} в отношении зависимости оценки от времени.
\begin{thrm}
	\label{sndw_A(k)_exp_time_shrp_thrm_1}
	Пусть $\widehat{N}_{0,Q} (\boldsymbol{\theta}_0) \ne 0$ при некотором $\boldsymbol{\theta}_0 \in \mathbb{S}^{d-1}$. Тогда не существует положительной функции $\mathcal{C}(\tau)$ такой, что $\lim_{\tau \to \infty} \mathcal{C}(\tau)/ |\tau| = 0$ и выполнена оценка
	\begin{equation}
	\label{sndw_A(k)_exp_time_shrp_est_1}
	\bigl\| \bigl( f e^{-i \tau \varepsilon^{-2} \mathcal{A} (\mathbf{k})} f^{-1} - f_0 e^{-i \tau \varepsilon^{-2} \mathcal{A}^0 (\mathbf{k})} f_0^{-1}\bigr) \mathcal{R} (\mathbf{k}, \varepsilon)^{3/2} \bigr\|_{L_2(\Omega) \to L_2 (\Omega)} \le \mathcal{C}(\tau) \varepsilon
	\end{equation}
	при всех $\tau \in \mathbb{R}$, почти всех $\mathbf{k} = t \boldsymbol{\theta} \in \widetilde{\Omega}$ и достаточно малых $\varepsilon > 0$.
\end{thrm}

\begin{proof}[Доказательство\nopunct] проведём от противного.  Предположим, что найдётся функция $\mathcal{C}(\tau) > 0$ такая, что $\lim_{\tau \to \infty} \mathcal{C}(\tau)/ |\tau| = 0$ и выполнена оценка~(\ref{sndw_A(k)_exp_time_shrp_est_1}) при почти всех $\mathbf{k} \in \widetilde{\Omega}$ и достаточно малом $\varepsilon > 0$. Тогда, с учётом~(\ref{R_P}), (\ref{R(k,eps)(I-P)_est}), отсюда следует, что найдётся функция $\widetilde{\mathcal{C}}(\tau) > 0$ такая, что $\lim_{\tau \to \infty} \widetilde{\mathcal{C}}(\tau)/ |\tau| = 0$ и выполнена оценка
\begin{equation}
\label{sndw_A(k)_exp_shrp_f1}
\bigl\| \bigl( f e^{-i \tau \varepsilon^{-2} \mathcal{A} (\mathbf{k})} f^{-1} - f_0 e^{-i \tau \varepsilon^{-2} \mathcal{A}^0 (\mathbf{k})} f_0^{-1}\bigr) \widehat{P} \bigr\|_{L_2(\Omega) \to L_2(\Omega)} \varepsilon^3 (|\mathbf{k}|^2 + \varepsilon^2)^{-3/2}  \le \widetilde{\mathcal{C}}(\tau) \varepsilon
\end{equation}

В силу~(\ref{abstr_P_Phat}) справедливо тождество $f^{-1} \widehat{P} = P f^* \overline{Q}$, где $P$~--- ортогональный проектор пространства $L_2 (\Omega; \mathbb{C}^n)$ на подпространство~$\mathfrak{N}$ (см.~(\ref{frakN})). Тогда оператор под знаком нормы в~(\ref{sndw_A(k)_exp_shrp_f1}) можно записать в виде $f e^{-i \tau \varepsilon^{-2} \mathcal{A} (\mathbf{k})} P f^* \overline{Q} - f_0 e^{-i \tau \varepsilon^{-2} \mathcal{A}^0 (\mathbf{k})} f_0^{-1} \widehat{P}$.
	
Затем, воспользуемся оценкой
\begin{equation}
\label{sndw_A(k)_exp_shrp_f2}
\bigl\| F(\mathbf{k}) - P \bigr\|_{L_2 (\Omega) \to L_2 (\Omega)} \le C_1 |\mathbf{k}|, \qquad |\mathbf{k}| \le t^0,
\end{equation}
(см.~(\ref{abstr_F(t)_threshold})). Отсюда следует, что найдётся функция $\check{\mathcal{C}}(\tau) > 0$ такая, что $\lim_{\tau \to \infty} \check{\mathcal{C}}(\tau)/ |\tau| = 0$ и выполнена оценка
\begin{equation}
\label{sndw_A(k)_exp_shrp_f3}
\bigl\|  f e^{-i \tau \varepsilon^{-2} \mathcal{A} (\mathbf{k})} F(\mathbf{k}) f^* \overline{Q} - f_0 e^{-i \tau \varepsilon^{-2} \mathcal{A}^0 (\mathbf{k})} f_0^{-1} \widehat{P} \bigr\|_{L_2(\Omega) \to L_2 (\Omega) } \varepsilon^3 (|\mathbf{k}|^2 + \varepsilon^2)^{-3/2}  \le \check{\mathcal{C}}(\tau) \varepsilon
\end{equation}
при почти всех $\mathbf{k} \in \widetilde{\Omega}$ в шаре $|\mathbf{k}| \le t^0$ и достаточно малых $\varepsilon > 0$.
Оператор, стоящий под знаком нормы в~(\ref{sndw_A(k)_exp_shrp_f3}) непрерывен по $\mathbf{k}$ в шаре $|\mathbf{k}| \le t^0$ при фиксированных $\tau$ и $\varepsilon$ (см.~\cite[лемма~11.8]{Su2017}). Следовательно, оценка~(\ref{sndw_A(k)_exp_shrp_f3}) справедлива при всех значениях $\mathbf{k}$ из данного шара. В частности, она верна в точке $\mathbf{k} = t\boldsymbol{\theta}_0$, если $t \le t^0$. Применяя снова неравенство~(\ref{sndw_A(k)_exp_shrp_f2}) и равенство $P f^* \overline{Q} = f^{-1} \widehat{P}$, получаем, что справедливо неравенство
\begin{equation}
\label{sndw_A(k)_exp_shrp_f4}
\bigl\| \bigl( f e^{-i \tau \varepsilon^{-2} \mathcal{A} (t\boldsymbol{\theta}_0)} f^{-1} - f_0 e^{-i \tau \varepsilon^{-2} \mathcal{A}^0 (t\boldsymbol{\theta}_0)} f_0^{-1}\bigr) \widehat{P} \bigr\|_{L_2(\Omega) \to L_2(\Omega)} \varepsilon^3 (t^2 + \varepsilon^2)^{-3/2}  \le \check{\mathcal{C}}'(\tau) \varepsilon
\end{equation}
c функцией $\check{\mathcal{C}}'(\tau) > 0$ такой, что $\lim_{\tau \to \infty} \check{\mathcal{C}}'(\tau)/ |\tau| = 0$,  при всех $t \le t^{0}$ и достаточно малых $\varepsilon > 0$.
	
Оценка~(\ref{sndw_A(k)_exp_shrp_f4}) в абстрактных терминах соответствует оценке~(\ref{abstr_sndw_exp_time_shrp_est_1}). Поскольку по условию выполнено $\widehat{N}_{0,Q}(\boldsymbol{\theta}_0) \ne 0$, то применение теоремы~\ref{abstr_sndw_exp_time_shrp_thrm_1} приводит нас к противоречию.
\end{proof}
Аналогично, применение теоремы~\ref{abstr_sndw_exp_time_shrp_thrm_2} позволяет подтвердить точность теорем~\ref{sndw_A(k)_exp_enchcd_thrm_1}, \ref{sndw_A(k)_exp_enchcd_thrm_2}.
\begin{thrm}
	\label{sndw_A(k)_exp_time_shrp_thrm_2}
	Пусть $\widehat{N}_{0,Q} (\boldsymbol{\theta}) = 0$ при всех $\boldsymbol{\theta} \in \mathbb{S}^{d-1}$ и пусть $\widehat{\mathcal{N}}_Q^{\,(q)} (\boldsymbol{\theta}_0) \ne 0$  при некотором $q \in \{1,\ldots,p(\boldsymbol{\theta}_0)\}$ и $\boldsymbol{\theta}_0 \in \mathbb{S}^{d-1}$. Тогда не существует положительной функции $\mathcal{C}(\tau)$ такой, что $\lim_{\tau \to \infty} \mathcal{C}(\tau)/ |\tau|^{1/2} = 0$ и выполнена оценка
	\begin{equation*}
	\bigl\| \bigl( f e^{-i \tau \varepsilon^{-2} \mathcal{A} (\mathbf{k})} f^{-1} - f_0 e^{-i \tau \varepsilon^{-2} \mathcal{A}^0 (\mathbf{k})} f_0^{-1} \bigr) \mathcal{R} (\mathbf{k}, \varepsilon) \bigr\|_{L_2(\Omega) \to L_2 (\Omega)} \le \mathcal{C}(\tau) \varepsilon
	\end{equation*}
	при всех $\tau \in \mathbb{R}$, почти всех $\mathbf{k} = t \boldsymbol{\theta} \in \widetilde{\Omega}$ и достаточно малых $\varepsilon > 0$.
\end{thrm}

\section{Аппроксимация  операторной экспоненты $e^{-i \tau \varepsilon^{-2} {\mathcal{A}}}$}

\subsection{Аппроксимация  оператора $e^{-i\tau \varepsilon^{-2} \widehat{\mathcal{A}}}$}
В $L_2 (\mathbb{R}^d; \mathbb{C}^n)$ рассмотрим оператор~(\ref{hatA}). Пусть $\widehat{\mathcal{A}}^0$~--- эффективный оператор~(\ref{hatA0}). Обозначим $\widehat{J}(\varepsilon; \tau) \coloneqq e^{-i\tau \varepsilon^{-2} \widehat{\mathcal{A}}} - e^{-i\tau \varepsilon^{-2} \widehat{\mathcal{A}}^0}$. Напомним обозначение $\mathcal{H}_0 = - \Delta$ и положим
\begin{equation}
\label{R(epsilon)}
\mathcal{R} (\varepsilon) \coloneqq \varepsilon^2 (\mathcal{H}_0 + \varepsilon^2 I)^{-1}.
\end{equation}
Оператор $\mathcal{R} (\varepsilon)$ раскладывается в прямой интеграл по операторам~(\ref{R(k,eps)}):
\begin{equation*}
\mathcal{R} (\varepsilon) = \mathscr{U}^{-1} \left( \int_{\widetilde{\Omega}} \oplus  \mathcal{R} (\mathbf{k}, \varepsilon) \, d \mathbf{k}  \right) \mathscr{U}.
\end{equation*}

Напомним также обозначение (\ref{Jhat(k,eps)}). Из разложений вида~(\ref{Gelfand_A_decompose}) для $\widehat{\mathcal{A}}$ и $\widehat{\mathcal{A}}^0$ следует равенство
\begin{equation}
\label{hatA_exps_Gelfand}
\| \widehat{J}(\varepsilon; \tau) \mathcal{R}(\varepsilon)^{s/2} \|_{L_2(\mathbb{R}^d) \to L_2(\mathbb{R}^d)} = \underset{\mathbf{k} \in \widetilde{\Omega}}{\esssup} \| \widehat{J}(\mathbf{k}, \varepsilon; \tau) \mathcal{R}(\mathbf{k}, \varepsilon)^{s/2} \|_{L_2(\Omega) \to L_2(\Omega)}.
\end{equation}
Поэтому из теорем~\ref{hatA(k)_exp_general_thrm}, \ref{hatA(k)_exp_enchcd_thrm_1}, \ref{hatA(k)_exp_enchcd_thrm_2} прямо вытекают следующие утверждения.
\begin{thrm}[\cite{BSu2008}]
	\label{hatA_exp_general_thrm}
	Для $\tau \in \mathbb{R}$ и $\varepsilon > 0$ справедлива оценка
	\begin{equation*}
	\bigl\| \widehat{J}(\varepsilon; \tau) \mathcal{R}(\varepsilon)^{3/2} \bigr\|_{L_2(\mathbb{R}^d) \to L_2(\mathbb{R}^d)} \le \widehat{\mathcal{C}}_1(1 + |\tau|) \varepsilon. 
	\end{equation*}
	Константа $\widehat{\mathcal{C}}_1$ зависит только от $\alpha_0$, $\alpha_1$, $\|g\|_{L_\infty}$, $\|g^{-1}\|_{L_\infty}$ и $r_0$.
\end{thrm}
\begin{thrm}
	\label{hatA_exp_enchcd_thrm_1}
	Пусть оператор $\widehat{N}(\boldsymbol{\theta})$ определён в~\emph{(\ref{hatN(theta)})}. Пусть $\widehat{N}(\boldsymbol{\theta})=0$ при всех $\boldsymbol{\theta} \in \mathbb{S}^{d-1}$. Тогда для $\tau \in \mathbb{R}$ и $\varepsilon > 0$ справедлива оценка
	\begin{equation*}
	\bigl\| \widehat{J}(\varepsilon; \tau) \mathcal{R}(\varepsilon) \bigr\|_{L_2(\mathbb{R}^d) \to L_2(\mathbb{R}^d)} \le \widehat{\mathcal{C}}_2(1 + |\tau|^{1/2}) \varepsilon.
	\end{equation*}
	Константа $\widehat{\mathcal{C}}_2$ зависит только от $\alpha_0$, $\alpha_1$, $\|g\|_{L_\infty}$, $\|g^{-1}\|_{L_\infty}$ и $r_0$.
\end{thrm}
\begin{thrm}
	\label{hatA_exp_enchcd_thrm_2}
	Пусть выполнено условие~\emph{\ref{cond1}} \emph{(}или более сильное условие~\emph{\ref{cond2}}\emph{)}. Тогда для $\tau \in \mathbb{R}$ и $\varepsilon > 0$ справедлива оценка
	\begin{equation*}
	\bigl\| \widehat{J}(\varepsilon; \tau) \mathcal{R}(\varepsilon) \bigr\|_{L_2(\mathbb{R}^d) \to L_2(\mathbb{R}^d)} \le \widehat{\mathcal{C}}_3 (1 + |\tau|^{1/2}) \varepsilon.
	\end{equation*}
	Константа $\widehat{\mathcal{C}}_3$ зависит от $\alpha_0$, $\alpha_1$, $\|g\|_{L_\infty}$, $\|g^{-1}\|_{L_\infty}$, $r_0$, а также от $n$ и $\widehat{c}^{\circ}$.
\end{thrm}
Теорема~\ref{hatA_exp_general_thrm} была доказана в~\cite[теорема~9.1]{BSu2008}. Теоремы~\ref{hatA_exp_enchcd_thrm_1} и~\ref{hatA_exp_enchcd_thrm_2} усиливают результаты теорем~12.2 и~12.3 из~\cite{Su2017} в отношении зависимости оценок от $\tau$.

Применение теорем~\ref{hatA(k)_exp_smooth_shrp_thrm_1}, \ref{hatA(k)_exp_smooth_shrp_thrm_2} позволяет подтвердить точность теорем~\ref{hatA_exp_general_thrm}, \ref{hatA_exp_enchcd_thrm_1}, \ref{hatA_exp_enchcd_thrm_2} в отношении сглаживания.
\begin{thrm}[\cite{Su2017}]
	\label{hatA_exp_smooth_shrp_thrm_1}
	Пусть $\widehat{N}_0 (\boldsymbol{\theta}_0) \ne 0$ при некотором $\boldsymbol{\theta}_0 \in \mathbb{S}^{d-1}$. Пусть $\tau \ne 0$ и $0 \le s < 3$. Тогда не существует такой константы $\mathcal{C}(\tau) >0$, чтобы оценка
	\begin{equation*}
	\bigl\| \widehat{J}(\varepsilon; \tau) \mathcal{R}(\varepsilon)^{s/2} \bigr\|_{L_2(\mathbb{R}^d) \to L_2(\mathbb{R}^d)} \le \mathcal{C}(\tau) \varepsilon
	\end{equation*}
	выполнялась при всех достаточно малых $\varepsilon > 0$.
\end{thrm}

\begin{thrm}
	\label{hatA_exp_smooth_shrp_thrm_2}
	Пусть $\widehat{N}_0 (\boldsymbol{\theta}) = 0$ при всех $\boldsymbol{\theta} \in \mathbb{S}^{d-1}$ и пусть $\widehat{\mathcal{N}}^{(q)} (\boldsymbol{\theta}_0) \ne 0$  при некоторых $q \in \{1,\ldots,p(\boldsymbol{\theta}_0)\}$ и $\boldsymbol{\theta}_0 \in \mathbb{S}^{d-1}$. 
	Пусть $\tau \ne 0$ и $0 \le s < 2$. Тогда не существует такой константы $\mathcal{C}(\tau) >0$, чтобы оценка
	\begin{equation*}
	\bigl\| \widehat{J}(\varepsilon; \tau) \mathcal{R}(\varepsilon)^{s/2} \bigr\|_{L_2(\mathbb{R}^d) \to L_2(\mathbb{R}^d)} \le \mathcal{C}(\tau) \varepsilon
	\end{equation*}
	выполнялась при всех достаточно малых $\varepsilon > 0$.
\end{thrm}
Теорема~\ref{hatA_exp_smooth_shrp_thrm_1} была доказана в~\cite[теорема~12.4]{Su2017}.

Далее, применение теорем~\ref{hatA(k)_exp_time_shrp_thrm_1}, \ref{hatA(k)_exp_time_shrp_thrm_2} позволяет подтвердить точность теорем~\ref{hatA_exp_general_thrm}, \ref{hatA_exp_enchcd_thrm_1}, \ref{hatA_exp_enchcd_thrm_2} в отношении зависимости оценки от времени.
\begin{thrm}
	\label{hatA_exp_shrp_thrm_1}
	Пусть $\widehat{N}_0 (\boldsymbol{\theta}_0) \ne 0$ при некотором $\boldsymbol{\theta}_0 \in \mathbb{S}^{d-1}$. Тогда не существует положительной функции $\mathcal{C}(\tau)$ такой, что $\lim_{\tau \to \infty} \mathcal{C}(\tau)/ |\tau| = 0$ и выполнена оценка
	\begin{equation*}
	\bigl\| \widehat{J}(\varepsilon;\tau) \mathcal{R} (\varepsilon)^{3/2} \bigr\|_{L_2(\mathbb{R}^d) \to L_2(\mathbb{R}^d)} \le \mathcal{C}(\tau) \varepsilon
	\end{equation*}
	при всех $\tau \in \mathbb{R}$ и всех достаточно малых $\varepsilon > 0$.
\end{thrm}

\begin{thrm}
	\label{hatA_exp_shrp_thrm_2}
	Пусть $\widehat{N}_0 (\boldsymbol{\theta}) = 0$ при всех $\boldsymbol{\theta} \in \mathbb{S}^{d-1}$ и пусть $\widehat{\mathcal{N}}^{(q)} (\boldsymbol{\theta}_0) \ne 0$  при некотором $q \in \{1,\ldots,p(\boldsymbol{\theta}_0)\}$ и $\boldsymbol{\theta}_0 \in \mathbb{S}^{d-1}$. Тогда не существует положительной функции $\mathcal{C}(\tau)$ такой, что $\lim_{\tau \to \infty} \mathcal{C}(\tau)/ |\tau|^{1/2} = 0$ и выполнена оценка
	\begin{equation*}
	\bigl\| \widehat{J}(\varepsilon;\tau) \mathcal{R} (\varepsilon) \bigr\|_{L_2(\mathbb{R}^d) \to L_2(\mathbb{R}^d)} \le \mathcal{C}(\tau) \varepsilon
	\end{equation*}
	при всех $\tau \in \mathbb{R}$ и всех достаточно малых $\varepsilon > 0$.
\end{thrm}

\subsection{Аппроксимация окаймлённого оператора $e^{-i \tau \varepsilon^{-2} \mathcal{A}}$}

В $L_2(\mathbb{R}^d; \mathbb{C}^n)$ рассмотрим оператор~(\ref{A}). Пусть $f_0$~--- матрица~(\ref{f0}), а $\mathcal{A}^0$~--- оператор~(\ref{A0}). Обозначим
\begin{equation*}
J(\varepsilon; \tau) \coloneqq f e^{-i \tau \varepsilon^{-2} \mathcal{A}} f^{-1} - f_0 e^{-i \tau \varepsilon^{-2} \mathcal{A}^0 } f_0^{-1}.
\end{equation*}
Аналогично~(\ref{hatA_exps_Gelfand}) имеем
\begin{equation*}
\| J(\varepsilon; \tau) \mathcal{R}(\varepsilon)^{s/2} \|_{L_2(\mathbb{R}^d) \to L_2(\mathbb{R}^d)} = \underset{\mathbf{k} \in \widetilde{\Omega}}{\esssup} \| J(\mathbf{k}, \varepsilon; \tau) \mathcal{R}(\mathbf{k}, \varepsilon)^{s/2} \|_{L_2(\Omega) \to L_2(\Omega)}.
\end{equation*}
Здесь $J(\mathbf{k}, \varepsilon; \tau)$ определено в~(\ref{J(k,eps)}). Таким образом, из теорем~\ref{sndw_A(k)_exp_general_thrm}, \ref{sndw_A(k)_exp_enchcd_thrm_1}, \ref{sndw_A(k)_exp_enchcd_thrm_2} получаем следующие результаты.
\begin{thrm}[\cite{BSu2008}]
	\label{sndw_A_exp_general_thrm}
	Для $\tau \in \mathbb{R}$ и $\varepsilon > 0$ справедлива оценка
	\begin{equation*}
	\bigl\| J(\varepsilon; \tau) \mathcal{R}(\varepsilon)^{3/2} \bigr\|_{L_2(\mathbb{R}^d) \to L_2(\mathbb{R}^d)} \le \mathcal{C}_1(1 + |\tau|) \varepsilon. 
	\end{equation*}
	Константа $\mathcal{C}_1$ зависит только от $\alpha_0$, $\alpha_1$, $\|g\|_{L_\infty}$, $\|g^{-1}\|_{L_\infty}$, $\|f\|_{L_{\infty}}$, $\|f^{-1}\|_{L_{\infty}}$ и $r_0$.
\end{thrm}
\begin{thrm}
	\label{sndw_A_exp_enchcd_thrm_1}
	Пусть оператор $\widehat{N}_{Q}(\boldsymbol{\theta})$, определённый в~\emph{(\ref{N_Q(theta)})}, равен нулю\emph{:} $\widehat{N}_{Q}(\boldsymbol{\theta})=0$ при всех $\boldsymbol{\theta} \in \mathbb{S}^{d-1}$. Тогда при $\tau \in \mathbb{R}$ и $\varepsilon > 0$ справедлива оценка
	\begin{equation*}
	\bigl\| J(\varepsilon; \tau) \mathcal{R}(\varepsilon) \bigr\|_{L_2(\mathbb{R}^d) \to L_2(\mathbb{R}^d)} \le \mathcal{C}_2(1 + |\tau|^{1/2}) \varepsilon.
	\end{equation*}
	Константа $\mathcal{C}_2$ зависит только от $\alpha_0$, $\alpha_1$, $\|g\|_{L_\infty}$, $\|g^{-1}\|_{L_\infty}$, $\|f\|_{L_{\infty}}$, $\|f^{-1}\|_{L_{\infty}}$ и $r_0$.
\end{thrm}
\begin{thrm}
	\label{sndw_A_exp_enchcd_thrm_2}
	Пусть выполнено условие~\emph{\ref{sndw_cond1}} \emph{(}или более сильное условие~\emph{\ref{sndw_cond2}}\emph{)}. Тогда при $\tau \in \mathbb{R}$ и $\varepsilon > 0$ справедлива оценка
	\begin{equation*}
	\bigl\| J(\varepsilon; \tau) \mathcal{R}(\varepsilon) \bigr\|_{L_2(\mathbb{R}^d) \to L_2(\mathbb{R}^d)} \le \mathcal{C}_3(1 + |\tau|^{1/2}) \varepsilon.
	\end{equation*}
	Константа $\mathcal{C}_3$ зависит от $\alpha_0$, $\alpha_1$, $\|g\|_{L_\infty}$, $\|g^{-1}\|_{L_\infty}$, $\|f\|_{L_{\infty}}$, $\|f^{-1}\|_{L_{\infty}}$, $r_0$, а также от $n$ и $c^{\circ}$.
\end{thrm}
Теорема~\ref{sndw_A_exp_general_thrm} была доказана в~\cite[теорема~10.1]{BSu2008}. Теоремы~\ref{sndw_A_exp_enchcd_thrm_1} и~\ref{sndw_A_exp_enchcd_thrm_2} усиливают результаты теорем~12.6 и~12.7 из~\cite{Su2017} в отношении зависимости оценок от $\tau$.

Применение теорем~\ref{sndw_A(k)_exp_smooth_shrp_thrm_1}, \ref{sndw_A(k)_exp_smooth_shrp_thrm_2} позволяет подтвердить точность теорем~\ref{sndw_A_exp_general_thrm}, \ref{sndw_A_exp_enchcd_thrm_1}, \ref{sndw_A_exp_enchcd_thrm_2} в отношении сглаживания.
\begin{thrm}[\cite{Su2017}]
	\label{sndw_A_exp_smooth_shrp_thrm_1}
	Пусть $\widehat{N}_{0,Q} (\boldsymbol{\theta}_0) \ne 0$ при некотором $\boldsymbol{\theta}_0 \in \mathbb{S}^{d-1}$. Пусть $\tau \ne 0$ и $0 \le s < 3$. Тогда не существует такой константы $\mathcal{C}(\tau) >0$, чтобы оценка
	\begin{equation*}
	\bigl\| J(\varepsilon; \tau) \mathcal{R}(\varepsilon)^{s/2} \bigr\|_{L_2(\mathbb{R}^d) \to L_2(\mathbb{R}^d)} \le \mathcal{C}(\tau) \varepsilon
	\end{equation*}
	выполнялась при всех достаточно малых $\varepsilon > 0$.
\end{thrm}

\begin{thrm}
	\label{sndw_A_exp_smooth_shrp_thrm_2}
	Пусть $\widehat{N}_{0,Q} (\boldsymbol{\theta}) = 0$ при всех $\boldsymbol{\theta} \in \mathbb{S}^{d-1}$ и пусть $\widehat{\mathcal{N}}_Q^{(q)} (\boldsymbol{\theta}_0) \ne 0$  при некоторых $q \in \{1,\ldots,p(\boldsymbol{\theta}_0)\}$ и $\boldsymbol{\theta}_0 \in \mathbb{S}^{d-1}$. 
	Пусть $\tau \ne 0$ и $0 \le s < 2$. Тогда не существует такой константы $\mathcal{C}(\tau) >0$, чтобы оценка
	\begin{equation*}
	\bigl\| J(\varepsilon; \tau) \mathcal{R}(\varepsilon)^{s/2} \bigr\|_{L_2(\mathbb{R}^d) \to L_2(\mathbb{R}^d)} \le \mathcal{C}(\tau) \varepsilon
	\end{equation*}
	выполнялась при всех достаточно малых $\varepsilon > 0$.
\end{thrm}
Теорема~\ref{sndw_A_exp_smooth_shrp_thrm_1} была доказана в~\cite[теорема~12.8]{Su2017}.

Далее, применение теорем~\ref{sndw_A(k)_exp_time_shrp_thrm_1}, \ref{sndw_A(k)_exp_time_shrp_thrm_2} позволяет подтвердить точность теорем~\ref{sndw_A_exp_general_thrm}, \ref{sndw_A_exp_enchcd_thrm_1}, \ref{sndw_A_exp_enchcd_thrm_2} в отношении зависимости оценки от времени.

\begin{thrm}
	\label{sndw_A_exp_shrp_thrm_1}
	Пусть $\widehat{N}_{0,Q} (\boldsymbol{\theta}_0) \ne 0$ при некотором $\boldsymbol{\theta}_0 \in \mathbb{S}^{d-1}$. Тогда не существует положительной функции $\mathcal{C}(\tau)$ такой, что $\lim_{\tau \to \infty} \mathcal{C}(\tau)/ |\tau| = 0$ и выполнена оценка
	\begin{equation*}
	\bigl\| J(\varepsilon;\tau) \mathcal{R} (\varepsilon)^{3/2} \bigr\|_{L_2(\mathbb{R}^d) \to L_2(\mathbb{R}^d)} \le \mathcal{C}(\tau) \varepsilon
	\end{equation*}
	при всех $\tau \in \mathbb{R}$ и всех достаточно малых $\varepsilon > 0$.
\end{thrm}
\begin{thrm}
	\label{sndw_A_exp_shrp_thrm_2}
	Пусть $\widehat{N}_{0,Q} (\boldsymbol{\theta}) = 0$ при всех $\boldsymbol{\theta} \in \mathbb{S}^{d-1}$ и пусть $\widehat{\mathcal{N}}_Q^{(q)} (\boldsymbol{\theta}_0) \ne 0$  при некотором $q \in \{1,\ldots,p(\boldsymbol{\theta}_0)\}$ и $\boldsymbol{\theta}_0 \in \mathbb{S}^{d-1}$. Тогда не существует положительной функции $\mathcal{C}(\tau)$ такой, что $\lim_{\tau \to \infty} \mathcal{C}(\tau)/ |\tau|^{1/2} = 0$ и выполнена оценка
	\begin{equation*}
	\bigl\| J(\varepsilon;\tau) \mathcal{R} (\varepsilon) \bigr\|_{L_2(\mathbb{R}^d) \to L_2(\mathbb{R}^d)} \le \mathcal{C}(\tau) \varepsilon
	\end{equation*}
	при всех $\tau \in \mathbb{R}$ и всех достаточно малых $\varepsilon > 0$.
\end{thrm}

\part{Задачи усреднения для нестационарных уравнений \\ типа Шрёдингера}
\label{main_results_part}

\section{Усреднение оператора $e^{-i\tau\mathcal{A}_\varepsilon}$}
\label{main_results_exp_section}

\subsection{Операторы $\widehat{\mathcal{A}}_\varepsilon$, $\mathcal{A}_\varepsilon$. Масштабное преобразование}
Если $\psi(\mathbf{x})$~--- измеримая $\Gamma$-периодическая функция в $\mathbb{R}^d$, условимся использовать обозначение $\psi^{\varepsilon}(\mathbf{x}) \coloneqq \psi(\varepsilon^{-1} \mathbf{x}), \; \varepsilon > 0$. \emph{Наши основные объекты}~--- операторы $\widehat{\mathcal{A}}_\varepsilon$, $\mathcal{A}_\varepsilon$, действующие в $L_2 (\mathbb{R}^d; \mathbb{C}^n)$, формально заданные выражениями
\begin{align}
\label{Ahat_eps}
\widehat{\mathcal{A}}_\varepsilon &\coloneqq b(\mathbf{D})^* g^{\varepsilon}(\mathbf{x}) b(\mathbf{D}), \\
\label{A_eps}
\mathcal{A}_\varepsilon &\coloneqq (f^{\varepsilon}(\mathbf{x}))^* b(\mathbf{D})^* g^{\varepsilon}(\mathbf{x}) b(\mathbf{D}) f^{\varepsilon}(\mathbf{x}).
\end{align}
Строгие определения даются через соответствующие квадратичные формы (ср.~п.~\ref{A_section}).

Пусть $T_{\varepsilon}$~--- \emph{унитарный в $L_2 (\mathbb{R}^d; \mathbb{C}^n)$ оператор масштабного преобразования}: 
$(T_{\varepsilon} \mathbf{u})(\mathbf{x}) = \varepsilon^{d/2} \mathbf{u} (\varepsilon \mathbf{x})$, $\varepsilon > 0$.
Тогда справедливо тождество $\mathcal{A}_\varepsilon = \varepsilon^{-2}T_{\varepsilon}^* \mathcal{A} T_{\varepsilon}$. Следовательно,
\begin{equation}
\label{exp_scale_transform}
e^{-i\tau\mathcal{A}_{\varepsilon}} = T_{\varepsilon}^* e^{-i\tau\varepsilon^{-2} \mathcal{A}} T_{\varepsilon}.
\end{equation}
Аналогичные соотношения выполнены и для оператора $\widehat{\mathcal{A}}_{\varepsilon}$. Применяя масштабное преобразование к резольвенте оператора $\mathcal{H}_0 = - \Delta$ и используя обозначение~(\ref{R(epsilon)}), получаем
\begin{equation}
\label{H0_resolv_scale_transform}
(\mathcal{H}_0 + I)^{-1} = \varepsilon^2 T_\varepsilon^* (\mathcal{H}_0 + \varepsilon^2 I)^{-1} T_\varepsilon = T_\varepsilon^* \mathcal{R} (\varepsilon) T_\varepsilon.
\end{equation}
Наконец, если $\psi(\mathbf{x})$~--- $\Gamma$-периодическая функция, то $[\psi^{\varepsilon}] = T_\varepsilon^* [\psi] T_\varepsilon$.

\subsection{Усреднение оператора $e^{-i\tau\widehat{\mathcal{A}}_\varepsilon}$}
Начнём с более простого оператора~(\ref{Ahat_eps}). Пусть $\widehat{\mathcal{A}}^0$~--- эффективный оператор~(\ref{hatA0}). Применяя соотношения вида~(\ref{exp_scale_transform}) (для операторов $\widehat{\mathcal{A}}_\varepsilon$ и $\widehat{\mathcal{A}}^0$), а также~(\ref{H0_resolv_scale_transform}), получаем тождество
\begin{equation}
\label{hatB_eps_exps_scale_transform}
(e^{-i\tau\widehat{\mathcal{A}}_{\varepsilon}} - e^{-i\tau\widehat{\mathcal{A}}^0}) (\mathcal{H}_0 + I)^{-s/2} =  T_{\varepsilon}^* \widehat{J}(\varepsilon; \tau) \mathcal{R} (\varepsilon)^{s/2} T_{\varepsilon}, \qquad \varepsilon > 0.
\end{equation}

Используя теорему~\ref{hatA_exp_general_thrm} и~(\ref{hatB_eps_exps_scale_transform}), можно получить следующий результат, ранее доказанный в~\cite[теорема~12.2]{BSu2008}
\begin{thrm}[\cite{BSu2008}]
	\label{hatA_eps_exp_general_thrm}  
	Пусть $\widehat{\mathcal{A}}_{\varepsilon}$~--- оператор~\emph{(\ref{Ahat_eps})} и  $\widehat{\mathcal{A}}^0$~---  эффективный оператор~\emph{(\ref{hatA0})}. Тогда при $0 \le s \le 3$ и $\tau \in \mathbb{R}$, $\varepsilon > 0$ справедлива оценка
	\begin{equation*}
	\| e^{-i\tau\widehat{\mathcal{A}}_{\varepsilon}} - e^{-i\tau\widehat{\mathcal{A}}^0} \|_{H^s (\mathbb{R}^d) \to L_2 (\mathbb{R}^d)} \le  \widehat{\mathfrak{C}}_1 (s) (1 +  |\tau|)^{s/3} \varepsilon^{s/3},  
	\end{equation*}
	где $\widehat{\mathfrak{C}}_1 (s) = 2^{1-s/3} \widehat{\mathcal{C}}_1^{s/3} $. Константа $\widehat{\mathcal{C}}_1$ зависит только от $\alpha_0$, $\alpha_1$, $\|g\|_{L_\infty}$, $\|g^{-1}\|_{L_\infty}$ и $r_0$. 
\end{thrm}

Этот результат может быть усилен при дополнительных предположениях. Из теоремы~\ref{hatA_exp_enchcd_thrm_1} выводится следующий результат.
\begin{thrm}
	\label{hatA_eps_exp_enchcd_thrm_1}
	Пусть выполнены условия теоремы~\emph{\ref{hatA_eps_exp_general_thrm}}. Пусть оператор $\widehat{N}(\boldsymbol{\theta})$ определён в~\emph{(\ref{hatN(theta)})}. Предположим, что $\widehat{N}(\boldsymbol{\theta})=0$ при всех $\boldsymbol{\theta} \in \mathbb{S}^{d-1}$. Тогда при $0 \le s \le 2$ и $\tau \in \mathbb{R}$, $\varepsilon > 0$ справедлива оценка
	\begin{equation}
	\label{hatA_eps_exp_enchcd_est_1}
	\| e^{-i\tau\widehat{\mathcal{A}}_{\varepsilon}} - e^{-i\tau\widehat{\mathcal{A}}^0} \|_{H^s (\mathbb{R}^d) \to L_2 (\mathbb{R}^d)} \le  \widehat{\mathfrak{C}}_2 (s) (1 +  |\tau|^{1/2})^{s/2} \varepsilon^{s/2},  
	\end{equation}
	где $\widehat{\mathfrak{C}}_2 (s) = 2^{1-s/2} \widehat{\mathcal{C}}_2^{s/2} $. Константа $\widehat{\mathcal{C}}_2$ зависит только от $\alpha_0$, $\alpha_1$, $\|g\|_{L_\infty}$, $\|g^{-1}\|_{L_\infty}$ и $r_0$.
\end{thrm}
\begin{proof}
	Ввиду унитарности оператора $T_\varepsilon$ и~(\ref{hatB_eps_exps_scale_transform}) из теоремы~\ref{hatA_exp_enchcd_thrm_1} следует оценка
	\begin{equation}
	\label{hatA_eps_exp_enchcd_est_1_L2L2}
	\| ( e^{-i\tau\widehat{\mathcal{A}}_{\varepsilon}} - e^{-i\tau\widehat{\mathcal{A}}^0}) (\mathcal{H}_0 + I)^{-1} \|_{L_2(\mathbb{R}^d) \to L_2(\mathbb{R}^d)} \le \widehat{\mathcal{C}}_2 (1+ |\tau|^{1/2}) \varepsilon.
	\end{equation}
	Очевидно,
	\begin{equation}
	\label{hatA_exp_trivial_est}
	\| e^{-i\tau\widehat{\mathcal{A}}_{\varepsilon}} - e^{-i\tau\widehat{\mathcal{A}}^0} \|_{L_2 (\mathbb{R}^d) \to L_2 (\mathbb{R}^d)} \le 2.
	\end{equation}
	Интерполируя между (\ref{hatA_exp_trivial_est}) и~(\ref{hatA_eps_exp_enchcd_est_1_L2L2}), при $ 0 \le s \le 2$ получаем
	\begin{equation}
	\label{hatA_exp_enchcd_est_1_interp}
	\| (e^{-i\tau\widehat{\mathcal{A}}_{\varepsilon}} - e^{-i\tau\widehat{\mathcal{A}}^0}) (\mathcal{H}_0 + I)^{-s/2} \|_{L_2 (\mathbb{R}^d) \to L_2 (\mathbb{R}^d)}
	\le 2^{1-s/2} \widehat{\mathcal{C}}_2^{s/2} (1 +  |\tau|^{1/2})^{s/2} \varepsilon^{s/2}. 
	\end{equation}
	Оператор $(\mathcal{H}_0 + I)^{s/2}$ осуществляет изометрический изоморфизм пространства Соболева $H^s(\mathbb{R}^d; \mathbb{C}^n)$ на 
	$L_2(\mathbb{R}^d; \mathbb{C}^n)$. Поэтому оценка~(\ref{hatA_exp_enchcd_est_1_interp}) эквивалентна~(\ref{hatA_eps_exp_enchcd_est_1}).
\end{proof}
Аналогично, применяя теорему~\ref{hatA_exp_enchcd_thrm_2}, получаем следующую теорему.
\begin{thrm}
	\label{hatA_eps_exp_enchcd_thrm_2}
	Пусть выполнены условия теоремы~\emph{\ref{hatA_eps_exp_general_thrm}}. Кроме того, пусть выполнено условие~\emph{\ref{cond1}} \emph{(}или более сильное условие~\emph{\ref{cond2})}. Тогда при $0 \le s \le 2$ и $\tau \in \mathbb{R}$, $\varepsilon > 0$ справедлива оценка
	\begin{equation*}
	\| e^{-i\tau\widehat{\mathcal{A}}_{\varepsilon}} - e^{-i\tau\widehat{\mathcal{A}}^0} \|_{H^s (\mathbb{R}^d) \to L_2 (\mathbb{R}^d)} \le  \widehat{\mathfrak{C}}_3 (s) (1 +  |\tau|^{1/2})^{s/2} \varepsilon^{s/2},  
	\end{equation*}
	где $\widehat{\mathfrak{C}}_3 (s) = 2^{1-s/2} \widehat{\mathcal{C}}_3^{s/2} $. Константа $\widehat{\mathcal{C}}_3$ зависит от $\alpha_0$, $\alpha_1$, $\|g\|_{L_\infty}$, $\|g^{-1}\|_{L_\infty}$, $r_0$, а также от $n$ и $\widehat{c}^{\circ}$.
\end{thrm}
Теоремы~\ref{hatA_eps_exp_enchcd_thrm_1} и~\ref{hatA_eps_exp_enchcd_thrm_2} усиливают результаты теорем~13.2 и~13.4 из~\cite{Su2017} в отношении зависимости оценок от $\tau$.

Применение теорем~\ref{hatA_exp_smooth_shrp_thrm_1}, \ref{hatA_exp_smooth_shrp_thrm_2} позволяет подтвердить точность теорем~\ref{hatA_eps_exp_general_thrm}, \ref{hatA_eps_exp_enchcd_thrm_1}, \ref{hatA_eps_exp_enchcd_thrm_2} в отношении типа операторной нормы.
\begin{thrm}[\cite{Su2017}]
	\label{hatA_eps_exp_smooth_shrp_thrm_1}
	Пусть $\widehat{N}_0 (\boldsymbol{\theta}_0) \ne 0$ при некотором $\boldsymbol{\theta}_0 \in \mathbb{S}^{d-1}$. Пусть $\tau \ne 0$ и $0 \le s < 3$. Тогда не существует такой константы $\mathcal{C}(\tau) >0$, чтобы оценка
	\begin{equation*}
	\| e^{-i\tau\widehat{\mathcal{A}}_{\varepsilon}} - e^{-i\tau\widehat{\mathcal{A}}^0} \|_{H^s (\mathbb{R}^d) \to L_2 (\mathbb{R}^d)} \le \mathcal{C}(\tau) \varepsilon
	\end{equation*}
	выполнялась при всех достаточно малых $\varepsilon > 0$.
\end{thrm}

\begin{thrm}
	Пусть $\widehat{N}_0 (\boldsymbol{\theta}) = 0$ при всех $\boldsymbol{\theta} \in \mathbb{S}^{d-1}$ и пусть $\widehat{\mathcal{N}}^{(q)} (\boldsymbol{\theta}_0) \ne 0$  при некоторых $q \in \{1,\ldots,p(\boldsymbol{\theta}_0)\}$ и $\boldsymbol{\theta}_0 \in \mathbb{S}^{d-1}$. 
	Пусть $\tau \ne 0$ и $0 \le s < 2$. Тогда не существует такой константы $\mathcal{C}(\tau) >0$, чтобы оценка
	\begin{equation*}
	\| e^{-i\tau\widehat{\mathcal{A}}_{\varepsilon}} - e^{-i\tau\widehat{\mathcal{A}}^0} \|_{H^s (\mathbb{R}^d) \to L_2 (\mathbb{R}^d)} \le \mathcal{C}(\tau) \varepsilon
	\end{equation*}
	выполнялась при всех достаточно малых $\varepsilon > 0$.
\end{thrm}
Теорема~\ref{hatA_eps_exp_smooth_shrp_thrm_1} была доказана в~\cite[теорема~13.6]{Su2017}.

Наконец, применение теорем~\ref{hatA_exp_shrp_thrm_1}, \ref{hatA_exp_shrp_thrm_2} позволяет подтвердить точность теорем~\ref{hatA_eps_exp_general_thrm}, \ref{hatA_eps_exp_enchcd_thrm_1}, \ref{hatA_eps_exp_enchcd_thrm_2} в отношении зависимости оценки от времени.
\begin{thrm}
	Пусть $\widehat{N}_0 (\boldsymbol{\theta}_0) \ne 0$ при некотором $\boldsymbol{\theta}_0 \in \mathbb{S}^{d-1}$. Тогда не существует положительной функции $\mathcal{C}(\tau)$ такой, что $\lim_{\tau \to \infty} \mathcal{C}(\tau)/ |\tau| = 0$ и выполнена оценка
	\begin{equation*}
	\| e^{-i\tau\widehat{\mathcal{A}}_{\varepsilon}} - e^{-i\tau\widehat{\mathcal{A}}^0} \|_{H^3 (\mathbb{R}^d) \to L_2 (\mathbb{R}^d)} \le \mathcal{C}(\tau) \varepsilon
	\end{equation*}
	при всех $\tau \in \mathbb{R}$ и всех достаточно малых $\varepsilon > 0$.
\end{thrm}

\begin{thrm}
	Пусть $\widehat{N}_0 (\boldsymbol{\theta}) = 0$ при всех $\boldsymbol{\theta} \in \mathbb{S}^{d-1}$ и пусть $\widehat{\mathcal{N}}^{(q)} (\boldsymbol{\theta}_0) \ne 0$  при некотором $q \in \{1,\ldots,p(\boldsymbol{\theta}_0)\}$ и $\boldsymbol{\theta}_0 \in \mathbb{S}^{d-1}$. Тогда не существует положительной функции $\mathcal{C}(\tau)$ такой, что $\lim_{\tau \to \infty} \mathcal{C}(\tau)/ |\tau|^{1/2} = 0$ и выполнена оценка
	\begin{equation*}
	\| e^{-i\tau\widehat{\mathcal{A}}_{\varepsilon}} - e^{-i\tau\widehat{\mathcal{A}}^0} \|_{H^2 (\mathbb{R}^d) \to L_2 (\mathbb{R}^d)} \le \mathcal{C}(\tau) \varepsilon
	\end{equation*}
	при всех $\tau \in \mathbb{R}$ и всех достаточно малых $\varepsilon > 0$.
\end{thrm}

\subsection{Усреднение окаймлённого оператора $e^{-i\tau\mathcal{A}_\varepsilon}$}

Рассмотрим теперь более общий оператор $\mathcal{A}_{\varepsilon}$ (см.~(\ref{A_eps})).
Пусть оператор $\mathcal{A}^0$ определён в~(\ref{A0}). Применяя соотношения вида~(\ref{exp_scale_transform}) (для операторов $\mathcal{A}_\varepsilon$ и $\mathcal{A}^0$), а также~(\ref{H0_resolv_scale_transform}), получаем тождество
\begin{equation}
\label{sndw_exps_scale_transform}
\bigl(f^\varepsilon e^{-i \tau \mathcal{A}_{\varepsilon}} (f^\varepsilon)^{-1} - f_0 e^{-i \tau \mathcal{A}^0} f_0^{-1} \bigr) (\mathcal{H}_0 + I)^{-s/2} =  T_{\varepsilon}^* J(\varepsilon; \tau) \mathcal{R} (\varepsilon)^{s/2} T_{\varepsilon}, \qquad \varepsilon > 0.
\end{equation}

Из теоремы~\ref{sndw_A_exp_general_thrm} и тождества~(\ref{sndw_exps_scale_transform}) можно получить следующий результат, ранее доказанный в~\cite[теорема~12.4]{BSu2008}.
\begin{thrm}[\cite{BSu2008}]
	\label{sndw_A_eps_exp_general_thrm}
	Пусть $\mathcal{A}_{\varepsilon}$ и $\mathcal{A}^0$~--- операторы, определённые выражениями~\emph{(\ref{A_eps})} и~\emph{(\ref{A0})}. Тогда при $0 \le s \le 3$ и $\tau \in \mathbb{R}$, $\varepsilon > 0$ справедлива оценка
	\begin{equation*}
	\| f^\varepsilon e^{-i \tau \mathcal{A}_{\varepsilon}} (f^\varepsilon)^{-1} - f_0 e^{-i \tau \mathcal{A}^0} f_0^{-1} \|_{H^s (\mathbb{R}^d) \to L_2 (\mathbb{R}^d)} \le  \mathfrak{C}_1 (s) (1 +  |\tau|)^{s/3} \varepsilon^{s/3},  
	\end{equation*}
	где $\mathfrak{C}_1 (s) = (2 \|f\|_{L_\infty} \|f^{-1}\|_{L_\infty})^{1-s/3} \mathcal{C}_1^{s/3}$. Константа $\mathcal{C}_1$ зависит только от $\alpha_0$, $\alpha_1$, $\|g\|_{L_\infty}$, $\|g^{-1}\|_{L_\infty}$, $\|f\|_{L_{\infty}}$, $\|f^{-1}\|_{L_{\infty}}$ и $r_0$.
\end{thrm}
Этот результат может быть усилен при дополнительных предположениях. Применяя теорему~\ref{sndw_A_exp_enchcd_thrm_1} c учётом~(\ref{sndw_exps_scale_transform}) и очевидной оценки
\begin{equation*}
\| f^\varepsilon e^{-i \tau \mathcal{A}_{\varepsilon}} (f^\varepsilon)^{-1} - f_0 e^{-i \tau \mathcal{A}^0} f_0^{-1} \|_{L_2(\mathbb{R}^d) \to L_2(\mathbb{R}^d)} \le 2 \|f\|_{L_\infty} \|f^{-1}\|_{L_\infty},
\end{equation*}
получаем следующий результат.
\begin{thrm}
	\label{sndw_A_eps_exp_enchcd_thrm_1}
	Пусть выполнены условия теоремы~\emph{\ref{sndw_A_eps_exp_general_thrm}}. Пусть оператор $\widehat{N}_Q (\boldsymbol{\theta})$, определённый в~\emph{(\ref{N_Q(theta)})}, равен нулю\emph{:} $\widehat{N}_Q (\boldsymbol{\theta}) = 0$ при всех $\boldsymbol{\theta} \in \mathbb{S}^{d-1}$. Тогда при $0 \le s \le 2$ и $\tau \in \mathbb{R}$, $\varepsilon > 0$ справедлива оценка
	\begin{equation*}
	\| f^\varepsilon e^{-i \tau \mathcal{A}_{\varepsilon}} (f^\varepsilon)^{-1} - f_0 e^{-i \tau \mathcal{A}^0} f_0^{-1} \|_{H^s (\mathbb{R}^d) \to L_2 (\mathbb{R}^d)} \le  \mathfrak{C}_2 (s) (1 +  |\tau|^{1/2})^{s/2} \varepsilon^{s/2},  
	\end{equation*}
	где $\mathfrak{C}_2 (s) = (2 \|f\|_{L_\infty} \|f^{-1}\|_{L_\infty})^{1-s/2} \mathcal{C}_2^{s/2}$. Константа $\mathcal{C}_2$ зависит только от $\alpha_0$, $\alpha_1$, $\|g\|_{L_\infty}$, $\|g^{-1}\|_{L_\infty}$, $\|f\|_{L_{\infty}}$, $\|f^{-1}\|_{L_{\infty}}$ и $r_0$.
\end{thrm}
Аналогично, применяя теорему~\ref{sndw_A_exp_enchcd_thrm_2}, получаем следующее утверждение.
\begin{thrm}
	\label{sndw_A_eps_exp_enchcd_thrm_2}
	Пусть выполнены условия теоремы~\emph{\ref{sndw_A_eps_exp_general_thrm}}. Кроме того, пусть выполнено условие~\emph{\ref{sndw_cond1}} \emph{(}или более сильное условие~\emph{\ref{sndw_cond2})}. Тогда при $0 \le s \le 2$ и $\tau \in \mathbb{R}$, $\varepsilon > 0$ справедлива оценка
	\begin{equation*}
	\| f^\varepsilon e^{-i \tau \mathcal{A}_{\varepsilon}} (f^\varepsilon)^{-1} - f_0 e^{-i \tau \mathcal{A}^0} f_0^{-1} \|_{H^s (\mathbb{R}^d) \to L_2 (\mathbb{R}^d)} \le  \mathfrak{C}_3 (s) (1 +  |\tau|^{1/2})^{s/2} \varepsilon^{s/2}, 
	\end{equation*}
	где $\mathfrak{C}_3 (s) = (2 \|f\|_{L_\infty} \|f^{-1}\|_{L_\infty})^{1-s/2}  \mathcal{C}_3^{s/2}$. Константа $\mathcal{C}_3$ зависит от $\alpha_0$, $\alpha_1$, $\|g\|_{L_\infty}$, $\|g^{-1}\|_{L_\infty}$, $\|f\|_{L_{\infty}}$, $\|f^{-1}\|_{L_{\infty}}$, $r_0$, а также от $n$ и $c^{\circ}$. 
\end{thrm}
Теоремы~\ref{sndw_A_eps_exp_enchcd_thrm_1} и~\ref{sndw_A_eps_exp_enchcd_thrm_2} усиливают результаты теорем~13.8 и~13.10 из~\cite{Su2017} в отношении зависимости оценок от $\tau$.

Применение теорем~\ref{sndw_A_exp_smooth_shrp_thrm_1}, \ref{sndw_A_exp_smooth_shrp_thrm_2} позволяет подтвердить точность теорем~\ref{sndw_A_eps_exp_general_thrm}, \ref{sndw_A_eps_exp_enchcd_thrm_1}, \ref{sndw_A_eps_exp_enchcd_thrm_2} в отношении типа операторной нормы.
\begin{thrm}[\cite{Su2017}]
	\label{sndw_A_eps_exp_smooth_shrp_thrm_1}
	Пусть $\widehat{N}_{0,Q} (\boldsymbol{\theta}_0) \ne 0$ при некотором $\boldsymbol{\theta}_0 \in \mathbb{S}^{d-1}$. Пусть $\tau \ne 0$ и $0 \le s < 3$. Тогда не существует такой константы $\mathcal{C}(\tau) >0$, чтобы оценка
	\begin{equation*}
	\| f^\varepsilon e^{-i \tau \mathcal{A}_{\varepsilon}} (f^\varepsilon)^{-1} - f_0 e^{-i \tau \mathcal{A}^0} f_0^{-1} \|_{H^s (\mathbb{R}^d) \to L_2 (\mathbb{R}^d)} \le \mathcal{C}(\tau) \varepsilon
	\end{equation*}
	выполнялась при всех достаточно малых $\varepsilon > 0$.
\end{thrm}

\begin{thrm}
	Пусть $\widehat{N}_{0,Q} (\boldsymbol{\theta}) = 0$ при всех $\boldsymbol{\theta} \in \mathbb{S}^{d-1}$ и пусть $\widehat{\mathcal{N}}_Q^{(q)} (\boldsymbol{\theta}_0) \ne 0$  при некоторых $q \in \{1,\ldots,p(\boldsymbol{\theta}_0)\}$ и $\boldsymbol{\theta}_0 \in \mathbb{S}^{d-1}$. 
	Пусть $\tau \ne 0$ и $0 \le s < 2$. Тогда не существует такой константы $\mathcal{C}(\tau) >0$, чтобы оценка
	\begin{equation*}
	\| f^\varepsilon e^{-i \tau \mathcal{A}_{\varepsilon}} (f^\varepsilon)^{-1} - f_0 e^{-i \tau \mathcal{A}^0} f_0^{-1} \|_{H^s (\mathbb{R}^d) \to L_2 (\mathbb{R}^d)} \le \mathcal{C}(\tau) \varepsilon
	\end{equation*}
	выполнялась при всех достаточно малых $\varepsilon > 0$.
\end{thrm}
Теорема~\ref{sndw_A_eps_exp_smooth_shrp_thrm_1} была доказана в~\cite[теорема~13.12]{Su2017}.

Применение теоремы~\ref{sndw_A_exp_shrp_thrm_1} позволяет подтвердить точность теоремы~\ref{sndw_A_eps_exp_general_thrm} в отношении зависимости оценки от времени.
\begin{thrm}
	Пусть $\widehat{N}_{0,Q} (\boldsymbol{\theta}_0) \ne 0$ при некотором $\boldsymbol{\theta}_0 \in \mathbb{S}^{d-1}$. Тогда не существует положительной функции $\mathcal{C}(\tau)$ такой, что $\lim_{\tau \to \infty} \mathcal{C}(\tau)/ |\tau| = 0$ и выполнена оценка
	\begin{equation*}
	\| f^\varepsilon e^{-i \tau \mathcal{A}_{\varepsilon}} (f^\varepsilon)^{-1} - f_0 e^{-i \tau \mathcal{A}^0} f_0^{-1} \|_{H^3 (\mathbb{R}^d) \to L_2 (\mathbb{R}^d)} \le \mathcal{C}(\tau) \varepsilon
	\end{equation*}
	при всех $\tau \in \mathbb{R}$ и всех достаточно малых $\varepsilon > 0$.
\end{thrm}
Аналогично, применение теоремы~\ref{sndw_A_exp_shrp_thrm_2} позволяет подтвердить точность теорем~\ref{sndw_A_eps_exp_enchcd_thrm_1}, \ref{sndw_A_eps_exp_enchcd_thrm_2}.
\begin{thrm}
	Пусть $\widehat{N}_{0,Q} (\boldsymbol{\theta}) = 0$ при всех $\boldsymbol{\theta} \in \mathbb{S}^{d-1}$ и пусть $\widehat{\mathcal{N}}_Q^{(q)} (\boldsymbol{\theta}_0) \ne 0$  при некотором $q \in \{1,\ldots,p(\boldsymbol{\theta}_0)\}$ и $\boldsymbol{\theta}_0 \in \mathbb{S}^{d-1}$. Тогда не существует положительной функции $\mathcal{C}(\tau)$ такой, что $\lim_{\tau \to \infty} \mathcal{C}(\tau)/ |\tau|^{1/2} = 0$ и выполнена оценка
	\begin{equation*}
	\| f^\varepsilon e^{-i \tau \mathcal{A}_{\varepsilon}} (f^\varepsilon)^{-1} - f_0 e^{-i \tau \mathcal{A}^0} f_0^{-1} \|_{H^2 (\mathbb{R}^d) \to L_2 (\mathbb{R}^d)} \le \mathcal{C}(\tau) \varepsilon
	\end{equation*}
	при всех $\tau \in \mathbb{R}$ и всех достаточно малых $\varepsilon > 0$.
\end{thrm}

\section{Усреднение задачи Коши для уравнения типа Шрёдингера}
\label{main_results_Cauchy_section}
\subsection{Задача Коши для уравнения с оператором $\widehat{\mathcal{A}}_\varepsilon$}
Пусть $\mathbf{u}_\varepsilon (\mathbf{x}, \tau)$~--- решение следующей задачи Коши:
\begin{equation}
\label{Cauchy_hatA_eps}
\left\{
\begin{aligned}
&i\frac{\partial \mathbf{u}_\varepsilon (\mathbf{x}, \tau)}{\partial \tau} = b(\mathbf{D})^*g^{\varepsilon}(\mathbf{x}) b(\mathbf{D}) \mathbf{u}_\varepsilon (\mathbf{x}, \tau) + \mathbf{F} (\mathbf{x}, \tau), \qquad \mathbf{x} \in \mathbb{R}^d, \; \tau \in \mathbb{R}, \\
& \mathbf{u}_\varepsilon (\mathbf{x}, 0) = \boldsymbol{\phi} (\mathbf{x}),
\end{aligned}
\right.
\end{equation}
где $\boldsymbol{\phi} \in L_2 (\mathbb{R}^d; \mathbb{C}^n)$, $ \mathbf{F} \in L_{1, \mathrm{loc}} (\mathbb{R}; L_2 (\mathbb{R}^d; \mathbb{C}^n))$. Справедливо представление 
\begin{equation*}
\mathbf{u}_\varepsilon (\cdot, \tau) = e^{-i\tau\widehat{\mathcal{A}}_{\varepsilon}} \boldsymbol{\phi} - i \int_{0}^{\tau} e^{-i(\tau-\tilde{\tau})\widehat{\mathcal{A}}_{\varepsilon}} \mathbf{F} (\cdot, \tilde{\tau}) \, d \tilde{\tau}.
\end{equation*}
Пусть $\mathbf{u}_0 (\mathbf{x}, \tau)$~--- решение \textquotedblleft усреднённой\textquotedblright \ задачи:
\begin{equation}
\label{Cauchy_hatA0}
\left\{
\begin{aligned}
&i\frac{\partial \mathbf{u}_0 (\mathbf{x}, \tau)}{\partial \tau} = b(\mathbf{D})^* g^{0} b(\mathbf{D}) \mathbf{u}_0 (\mathbf{x}, \tau) + \mathbf{F} (\mathbf{x}, \tau), \qquad \mathbf{x} \in \mathbb{R}^d, \; \tau \in \mathbb{R}, \\
& \mathbf{u}_0 (\mathbf{x}, 0) = \boldsymbol{\phi} (\mathbf{x}), \qquad \mathbf{x} \in \mathbb{R}^d.
\end{aligned}
\right.
\end{equation}
Тогда 
\begin{equation*}
\mathbf{u}_0 (\cdot, \tau) = e^{-i\tau\widehat{\mathcal{A}}^0} \boldsymbol{\phi} - i\int_{0}^{\tau} e^{-i(\tau-\tilde{\tau})\widehat{\mathcal{A}}^0} \mathbf{F} (\cdot, \tilde{\tau}) \, d \tilde{\tau}.
\end{equation*}

Из теоремы~\ref{hatA_eps_exp_general_thrm} непосредственно вытекает следующий результат (доказанный ранее в~\cite[теорема~14.2]{BSu2008}).
\begin{thrm}[\cite{BSu2008}]
	\label{hatA_eps_Cauchy_general_thrm}
	Пусть $\mathbf{u}_\varepsilon$~--- решение задачи~\emph{(\ref{Cauchy_hatA_eps})} и $\mathbf{u}_0$~--- решение задачи~\emph{(\ref{Cauchy_hatA0})}.
	\begin{enumerate}[label=\emph{\arabic*$^{\circ}.$}, ref=\arabic*$^{\circ}$, leftmargin=2.5\parindent]
		\item Если $\boldsymbol{\phi} \in H^s (\mathbb{R}^d; \mathbb{C}^n)$, $\mathbf{F} \in L_{1,\mathrm{loc}}(\mathbb{R}; H^s (\mathbb{R}^d; \mathbb{C}^n))$, где $0 \le s \le 3$, то при $\tau \in \mathbb{R}$ и $\varepsilon > 0$ выполнена оценка
		\begin{equation*}
		\| \mathbf{u}_\varepsilon (\cdot, \tau) - \mathbf{u}_0 (\cdot, \tau) \|_{L_2(\mathbb{R}^d)} \le \varepsilon^{s/3} (1 +  |\tau|)^{s/3}  \widehat{\mathfrak{C}}_1 (s) \left(\| \boldsymbol{\phi} \|_{H^s(\mathbb{R}^d)} + \|\mathbf{F} \|_{L_1((0,\tau);H^s(\mathbb{R}^d))} \right).
		\end{equation*}
		При дополнительном предположении $\mathbf{F} \in L_p (\mathbb{R}_\pm; H^s (\mathbb{R}^d, \mathbb{C}^n))$, где $p \in [1, \infty]$, и при $\tau = \pm \varepsilon^{-\alpha}$, $0 < \varepsilon \le 1$, $0 < \alpha < s(s + 3/p')^{-1}$ справедлива оценка
		\begin{multline*}
		\| \mathbf{u}_\varepsilon (\cdot, \pm \varepsilon^{-\alpha}) - \mathbf{u}_0 (\cdot, \pm \varepsilon^{-\alpha}) \|_{L_2(\mathbb{R}^d)} \le \\ \le \varepsilon^{s(1-\alpha)/3} \cdot 2^{s/3}\widehat{\mathfrak{C}}_1 (s) \left(\| \boldsymbol{\phi} \|_{H^s(\mathbb{R}^d)} + \varepsilon^{-\alpha/p'} \|\mathbf{F} \|_{L_p(\mathbb{R}_\pm;H^s(\mathbb{R}^d))} \right).
		\end{multline*}
		Величина $\widehat{\mathfrak{C}}_1 (s)$ определена в~теореме \emph{\ref{hatA_eps_exp_general_thrm}}. Здесь $p^{-1} + (p')^{-1} = 1$.
		\item Если $\boldsymbol{\phi} \in L_2 (\mathbb{R}^d; \mathbb{C}^n)$ и $ \mathbf{F} \in L_{1, \mathrm{loc}} (\mathbb{R}; L_2 (\mathbb{R}^d; \mathbb{C}^n) )$, то
		\begin{equation*}
		\lim\limits_{\varepsilon \to 0} \| \mathbf{u}_\varepsilon (\cdot, \tau) - \mathbf{u}_0 (\cdot, \tau) \|_{L_2(\mathbb{R}^d)} = 0, \qquad \tau \in \mathbb{R}.
		\end{equation*}
		При дополнительном предположении $\mathbf{F} \in L_1 (\mathbb{R}_\pm; L_2 (\mathbb{R}^d, \mathbb{C}^n))$ справедливо
		\begin{equation*}
		\lim\limits_{\varepsilon \to 0} \| \mathbf{u}_\varepsilon (\cdot, \pm\varepsilon^{-\alpha}) - \mathbf{u}_0 (\cdot, \pm\varepsilon^{-\alpha}) \|_{L_2(\mathbb{R}^d)} = 0, \qquad 0 < \alpha < 1.
		\end{equation*}
	\end{enumerate}
\end{thrm}
Результат теоремы~\ref{hatA_eps_Cauchy_general_thrm} можно усилить при дополнительных предположениях. Применяя теорему~\ref{hatA_eps_exp_enchcd_thrm_1}, получаем следующее утверждение.
\begin{thrm}
	\label{hatA_eps_Cauchy_enchcd_thrm_1}
	Пусть выполнены условия теоремы~\emph{\ref{hatA_eps_Cauchy_general_thrm}}. Пусть оператор  $\widehat{N} (\boldsymbol{\theta})$ определён в~\emph{(\ref{hatN(theta)})}. Предположим, что $\widehat{N} (\boldsymbol{\theta}) = 0$ при всех $\boldsymbol{\theta} \in \mathbb{S}^{d-1}$.
	\begin{enumerate}[label=\emph{\arabic*$^{\circ}.$}, ref=\arabic*$^{\circ}$, leftmargin=2.5\parindent]
	\item	 Если $\boldsymbol{\phi} \in H^s (\mathbb{R}^d; \mathbb{C}^n)$, $\mathbf{F} \in L_{1,\mathrm{loc}}(\mathbb{R}; H^s (\mathbb{R}^d; \mathbb{C}^n))$, где $0 \le s \le 2$, то при $\tau \in \mathbb{R}$ и $\varepsilon > 0$ выполнена оценка
	\begin{equation*}
	\| \mathbf{u}_\varepsilon (\cdot, \tau) - \mathbf{u}_0 (\cdot, \tau) \|_{L_2(\mathbb{R}^d)} \le \varepsilon^{s/2} (1 +  |\tau|^{1/2})^{s/2} \widehat{\mathfrak{C}}_2 (s) \left(\| \boldsymbol{\phi} \|_{H^s(\mathbb{R}^d)} + \|\mathbf{F} \|_{L_1((0,s);H^s(\mathbb{R}^d))} \right).
	\end{equation*}
	При дополнительном предположении $\mathbf{F} \in L_p (\mathbb{R}_\pm; H^s (\mathbb{R}^d, \mathbb{C}^n))$, где $p \in [1, \infty]$, и при $\tau = \pm \varepsilon^{-\alpha}$, $0 < \varepsilon \le 1$, $0 < \alpha < 2s(s + 4/p')^{-1}$ справедлива оценка
	\begin{multline*}
	\| \mathbf{u}_\varepsilon (\cdot, \pm \varepsilon^{-\alpha}) - \mathbf{u}_0 (\cdot, \pm \varepsilon^{-\alpha}) \|_{L_2(\mathbb{R}^d)} \le \\ \le \varepsilon^{s(1-\alpha/2)/2} \cdot 2^{s/2}\widehat{\mathfrak{C}}_2 (s) \left(\| \boldsymbol{\phi} \|_{H^s(\mathbb{R}^d)} + \varepsilon^{-\alpha/p'} \|\mathbf{F} \|_{L_p(\mathbb{R}_\pm;H^s(\mathbb{R}^d))} \right).
	\end{multline*}
	Величина $\widehat{\mathfrak{C}}_2 (s)$ определена в~теореме \emph{\ref{hatA_eps_exp_enchcd_thrm_1}}. Здесь $p^{-1} + (p')^{-1} = 1$.
	\item Если $\boldsymbol{\phi} \in L_2 (\mathbb{R}^d; \mathbb{C}^n)$ и $\mathbf{F} \in L_1 (\mathbb{R}_\pm; L_2 (\mathbb{R}^d, \mathbb{C}^n))$, то
	\begin{equation*}
	\lim\limits_{\varepsilon \to 0} \| \mathbf{u}_\varepsilon (\cdot, \pm\varepsilon^{-\alpha}) - \mathbf{u}_0 (\cdot, \pm\varepsilon^{-\alpha}) \|_{L_2(\mathbb{R}^d)} = 0, \qquad 0 < \alpha < 2.
	\end{equation*}
	\end{enumerate}
\end{thrm}

Аналогично, применяя теорему~\ref{hatA_eps_exp_enchcd_thrm_2}, получаем следующее утверждение.
\begin{thrm}
	Пусть выполнены условия теоремы~\emph{\ref{hatA_eps_Cauchy_general_thrm}}. Предположим, что выполнено условие~\emph{\ref{cond1}} \emph{(}или более сильное условие~\emph{\ref{cond2})}.
	\begin{enumerate}[label=\emph{\arabic*$^{\circ}.$}, ref=\arabic*$^{\circ}$, leftmargin=2.5\parindent]
		\item	 Если $\boldsymbol{\phi} \in H^s (\mathbb{R}^d; \mathbb{C}^n)$, $\mathbf{F} \in L_{1,\mathrm{loc}}(\mathbb{R}; H^s (\mathbb{R}^d; \mathbb{C}^n))$, где $0 \le s \le 2$, то при $\tau \in \mathbb{R}$ и $\varepsilon > 0$ выполнена оценка
		\begin{equation*}
		\| \mathbf{u}_\varepsilon (\cdot, \tau) - \mathbf{u}_0 (\cdot, \tau) \|_{L_2(\mathbb{R}^d)} \le \varepsilon^{s/2} (1 +  |\tau|^{1/2})^{s/2} \widehat{\mathfrak{C}}_3 (s) \left(\| \boldsymbol{\phi} \|_{H^s(\mathbb{R}^d)} + \|\mathbf{F} \|_{L_1((0,\tau);H^s(\mathbb{R}^d))} \right).
		\end{equation*}
		При дополнительном предположении $\mathbf{F} \in L_p (\mathbb{R}_\pm; H^s (\mathbb{R}^d, \mathbb{C}^n))$, где $p \in [1, \infty]$, и при $\tau = \pm \varepsilon^{-\alpha}$, $0 < \varepsilon \le 1$, $0 < \alpha < 2s(s + 4/p')^{-1}$ справедлива оценка
		\begin{multline*}
		\| \mathbf{u}_\varepsilon (\cdot, \pm \varepsilon^{-\alpha}) - \mathbf{u}_0 (\cdot, \pm \varepsilon^{-\alpha}) \|_{L_2(\mathbb{R}^d)} \le \\ \le \varepsilon^{s(1-\alpha/2)/2} \cdot 2^{s/2}\widehat{\mathfrak{C}}_3 (s) \left(\| \boldsymbol{\phi} \|_{H^s(\mathbb{R}^d)} + \varepsilon^{-\alpha/p'} \|\mathbf{F} \|_{L_p(\mathbb{R}_\pm;H^s(\mathbb{R}^d))} \right).
		\end{multline*}
		Величина $\widehat{\mathfrak{C}}_3 (s)$ определена в~теореме \emph{\ref{hatA_eps_exp_enchcd_thrm_2}}. Здесь $p^{-1} + (p')^{-1} = 1$.
		\item Если $\boldsymbol{\phi} \in L_2 (\mathbb{R}^d; \mathbb{C}^n)$ и $\mathbf{F} \in L_1 (\mathbb{R}_\pm; L_2 (\mathbb{R}^d, \mathbb{C}^n))$, то
		\begin{equation*}
		\lim\limits_{\varepsilon \to 0} \| \mathbf{u}_\varepsilon (\cdot, \pm\varepsilon^{-\alpha}) - \mathbf{u}_0 (\cdot, \pm\varepsilon^{-\alpha}) \|_{L_2(\mathbb{R}^d)} = 0, \qquad 0 < \alpha < 2.
		\end{equation*}
	\end{enumerate}
\end{thrm}

\subsection{Задача Коши для уравнения с оператором $\mathcal{A}_\varepsilon$}

Рассмотрим более общую задачу Коши для уравнения с оператором $\mathcal{A}_\varepsilon$:
\begin{equation*}
\left\{
\begin{aligned}
&i\frac{\partial \mathbf{u}_\varepsilon (\mathbf{x}, \tau)}{\partial \tau} = (f^{\varepsilon}(\mathbf{x}))^* b(\mathbf{D})^* g^{\varepsilon}(\mathbf{x}) b(\mathbf{D}) f^{\varepsilon}(\mathbf{x}) \mathbf{u}_\varepsilon (\mathbf{x}, \tau) +  (f^\varepsilon (\mathbf{x}))^{-1} \mathbf{F} (\mathbf{x}, \tau), \qquad \mathbf{x} \in \mathbb{R}^d, \; \tau \in \mathbb{R},\\
& f^\varepsilon (\mathbf{x}) \mathbf{u}_\varepsilon (\mathbf{x}, 0) = \boldsymbol{\phi}(\mathbf{x}), \qquad \mathbf{x} \in \mathbb{R}^d,
\end{aligned}
\right.
\end{equation*}
где $\boldsymbol{\phi} \in L_2 (\mathbb{R}^d; \mathbb{C}^n)$, $ \mathbf{F} \in L_{1, \mathrm{loc}} (\mathbb{R}; L_2 (\mathbb{R}^d; \mathbb{C}^n) )$. Справедливо представление 
\begin{equation*}
\mathbf{u}_\varepsilon (\cdot, \tau) = e^{-i\tau\mathcal{A}_{\varepsilon}} (f^\varepsilon)^{-1} \boldsymbol{\phi} - i \int_{0}^{\tau} e^{-i(\tau-\tilde{\tau})\mathcal{A}_{\varepsilon}} (f^\varepsilon)^{-1} \mathbf{F} (\cdot, \tilde{\tau}) \, d \tilde{\tau}.
\end{equation*}
Пусть $\mathbf{u}_0 (\mathbf{x}, s)$~--- решение \textquotedblleft усреднённой\textquotedblright \ задачи:
\begin{equation*}
\left\{
\begin{aligned}
&i\frac{\partial \mathbf{u}_\varepsilon (\mathbf{x}, \tau)}{\partial \tau} = f_0 b(\mathbf{D})^* g^0 b(\mathbf{D}) f_0 \mathbf{u}_\varepsilon (\mathbf{x}, \tau) +  f_0^{-1} \mathbf{F} (\mathbf{x}, \tau), \qquad \mathbf{x} \in \mathbb{R}^d, \; \tau \in \mathbb{R},\\
& f_0 \mathbf{u}_\varepsilon (\mathbf{x}, 0) = \boldsymbol{\phi}(\mathbf{x}), \qquad \mathbf{x} \in \mathbb{R}^d.
\end{aligned}
\right.
\end{equation*}
Тогда 
\begin{equation*}
\mathbf{u}_0 (\cdot, \tau) = e^{-i\tau\mathcal{A}^0} f_0^{-1} \boldsymbol{\phi} - i \int_{0}^{\tau} e^{-i(\tau-\tilde{\tau})\mathcal{A}^0} f_0^{-1} \mathbf{F} (\cdot, \tilde{\tau}) \, d \tilde{\tau}.
\end{equation*}

Из теоремы~\ref{sndw_A_eps_exp_general_thrm} непосредственно вытекает следующий результат (доказанный ранее в~\cite[теорема~14.5]{BSu2008}).
\begin{thrm}[\cite{BSu2008}]
	\label{A_eps_Cauchy_general_thrm}
	Пусть $\mathbf{u}_\varepsilon$~--- решение задачи~\emph{(\ref{Cauchy_hatA_eps})} и $\mathbf{u}_0$~--- решение задачи~\emph{(\ref{Cauchy_hatA0})}.
	\begin{enumerate}[label=\emph{\arabic*$^{\circ}.$}, ref=\arabic*$^{\circ}$, leftmargin=2.5\parindent]
		\item Если $\boldsymbol{\phi} \in H^s (\mathbb{R}^d; \mathbb{C}^n)$, $\mathbf{F} \in L_{1,\mathrm{loc}}(\mathbb{R}; H^s (\mathbb{R}^d; \mathbb{C}^n))$, где $0 \le s \le 3$, то при $\tau \in \mathbb{R}$ и $\varepsilon > 0$ выполнена оценка
		\begin{equation*}
		\|f^\varepsilon \mathbf{u}_\varepsilon (\cdot, \tau) - f_0 \mathbf{u}_0 (\cdot, \tau) \|_{L_2(\mathbb{R}^d)} \le \varepsilon^{s/3} (1 +  |\tau|)^{s/3}  \mathfrak{C}_1 (s) \left(\| \boldsymbol{\phi} \|_{H^s(\mathbb{R}^d)} + \|\mathbf{F} \|_{L_1((0,\tau);H^s(\mathbb{R}^d))} \right).
		\end{equation*}
		При дополнительном предположении $\mathbf{F} \in L_p (\mathbb{R}_\pm; H^s (\mathbb{R}^d, \mathbb{C}^n))$, где $p \in [1, \infty]$, и при $\tau = \pm \varepsilon^{-\alpha}$, $0 < \varepsilon \le 1$, $0 < \alpha < s(s + 3/p')^{-1}$ справедлива оценка
		\begin{multline*}
		\| f^\varepsilon \mathbf{u}_\varepsilon (\cdot, \pm \varepsilon^{-\alpha}) - f_0 \mathbf{u}_0 (\cdot, \pm \varepsilon^{-\alpha}) \|_{L_2(\mathbb{R}^d)} \le \\ \le \varepsilon^{s(1-\alpha)/3} \cdot 2^{s/3}\mathfrak{C}_1 (s) \left(\| \boldsymbol{\phi} \|_{H^s(\mathbb{R}^d)} + \varepsilon^{-\alpha/p'} \|\mathbf{F} \|_{L_p(\mathbb{R}_\pm;H^s(\mathbb{R}^d))} \right).
		\end{multline*}
		Величина $\mathfrak{C}_1 (s)$ определена в~теореме \emph{\ref{sndw_A_eps_exp_general_thrm}}. Здесь $p^{-1} + (p')^{-1} = 1$.
		\item Если $\boldsymbol{\phi} \in L_2 (\mathbb{R}^d; \mathbb{C}^n)$ и $ \mathbf{F} \in L_{1, \mathrm{loc}} (\mathbb{R}; L_2 (\mathbb{R}^d; \mathbb{C}^n) )$, то
		\begin{equation*}
		\lim\limits_{\varepsilon \to 0} \|  f^\varepsilon \mathbf{u}_\varepsilon (\cdot, \tau) - f_0 \mathbf{u}_0 (\cdot, \tau) \|_{L_2(\mathbb{R}^d)} = 0, \qquad \tau \in \mathbb{R}.
		\end{equation*}
		При дополнительном предположении $\mathbf{F} \in L_1 (\mathbb{R}_\pm; L_2 (\mathbb{R}^d, \mathbb{C}^n))$ справедливо
		\begin{equation*}
		\lim\limits_{\varepsilon \to 0} \| f^\varepsilon \mathbf{u}_\varepsilon (\cdot, \pm\varepsilon^{-\alpha}) - f_0 \mathbf{u}_0 (\cdot, \pm\varepsilon^{-\alpha}) \|_{L_2(\mathbb{R}^d)} = 0, \qquad 0 < \alpha < 1.
		\end{equation*}
	\end{enumerate}
\end{thrm}

Результаты теоремы~\ref{A_eps_Cauchy_general_thrm} можно усилить при дополнительных предположениях. Применяя теорему~\ref{sndw_A_eps_exp_enchcd_thrm_1}, получаем следующее утверждение.
\begin{thrm}
	Пусть выполнены условия теоремы~\emph{\ref{A_eps_Cauchy_general_thrm}}. Пусть оператор  $\widehat{N}_Q (\boldsymbol{\theta})$ определён в~\emph{(\ref{N_Q(theta)})}. Предположим, что $\widehat{N}_Q (\boldsymbol{\theta}) = 0$ при всех $\boldsymbol{\theta} \in \mathbb{S}^{d-1}$.
	\begin{enumerate}[label=\emph{\arabic*$^{\circ}.$}, ref=\arabic*$^{\circ}$, leftmargin=2.5\parindent]
		\item Если $\boldsymbol{\phi} \in H^s (\mathbb{R}^d; \mathbb{C}^n)$, $\mathbf{F} \in L_{1,\mathrm{loc}}(\mathbb{R}; H^s (\mathbb{R}^d; \mathbb{C}^n))$, где $0 \le s \le 2$, то при $\tau \in \mathbb{R}$ и $\varepsilon > 0$ выполнена оценка
		\begin{equation*}
		\| f^\varepsilon \mathbf{u}_\varepsilon (\cdot, \tau) - f_0 \mathbf{u}_0 (\cdot, \tau) \|_{L_2(\mathbb{R}^d)} \le \varepsilon^{s/2} (1 +  |\tau|^{1/2})^{s/2} \mathfrak{C}_2 (s) \left(\| \boldsymbol{\phi} \|_{H^s(\mathbb{R}^d)} + \|\mathbf{F} \|_{L_1((0,s);H^s(\mathbb{R}^d))} \right).
		\end{equation*}
		При дополнительном предположении $\mathbf{F} \in L_p (\mathbb{R}_\pm; H^s (\mathbb{R}^d, \mathbb{C}^n))$, где $p \in [1, \infty]$, и при $\tau = \pm \varepsilon^{-\alpha}$, $0 < \varepsilon \le 1$, $0 < \alpha < 2s(s + 4/p')^{-1}$ справедлива оценка
		\begin{multline*}
		\| f^\varepsilon \mathbf{u}_\varepsilon (\cdot, \pm \varepsilon^{-\alpha}) - f_0 \mathbf{u}_0 (\cdot, \pm \varepsilon^{-\alpha}) \|_{L_2(\mathbb{R}^d)} \le \\ \le \varepsilon^{s(1-\alpha/2)/2} \cdot 2^{s/2}\mathfrak{C}_2 (s) \left(\| \boldsymbol{\phi} \|_{H^s(\mathbb{R}^d)} + \varepsilon^{-\alpha/p'} \|\mathbf{F} \|_{L_p(\mathbb{R}_\pm;H^s(\mathbb{R}^d))} \right).
		\end{multline*}
		Величина $\mathfrak{C}_2 (s)$ определена в~теореме \emph{\ref{sndw_A_eps_exp_enchcd_thrm_1}}. Здесь $p^{-1} + (p')^{-1} = 1$.
		\item Если $\boldsymbol{\phi} \in L_2 (\mathbb{R}^d; \mathbb{C}^n)$ и $\mathbf{F} \in L_1 (\mathbb{R}_\pm; L_2 (\mathbb{R}^d, \mathbb{C}^n))$, то
		\begin{equation*}
		\lim\limits_{\varepsilon \to 0} \| f^\varepsilon \mathbf{u}_\varepsilon (\cdot, \pm \varepsilon^{-\alpha}) - f_0 \mathbf{u}_0 (\cdot, \pm \varepsilon^{-\alpha}) \|_{L_2(\mathbb{R}^d)} = 0, \qquad 0 < \alpha < 2.
		\end{equation*}
	\end{enumerate}
\end{thrm}

Аналогично, применяя теорему~\ref{sndw_A_eps_exp_enchcd_thrm_2}, получаем следующее утверждение.
\begin{thrm}
	Пусть выполнены условия теоремы~\emph{\ref{A_eps_Cauchy_general_thrm}}. Предположим, что выполнено условие~\emph{\ref{sndw_cond1}} \emph{(}или более сильное условие~\emph{\ref{sndw_cond2})}.
	\begin{enumerate}[label=\emph{\arabic*$^{\circ}.$}, ref=\arabic*$^{\circ}$, leftmargin=2.5\parindent]
		\item	 Если $\boldsymbol{\phi} \in H^s (\mathbb{R}^d; \mathbb{C}^n)$, $\mathbf{F} \in L_{1,\mathrm{loc}}(\mathbb{R}; H^s (\mathbb{R}^d; \mathbb{C}^n))$, где $0 \le s \le 2$, то при $\tau \in \mathbb{R}$ и $\varepsilon > 0$ выполнена оценка
		\begin{equation*}
		\| f^\varepsilon \mathbf{u}_\varepsilon (\cdot, \tau) - f_0 \mathbf{u}_0 (\cdot, \tau) \|_{L_2(\mathbb{R}^d)} \le \varepsilon^{s/2} (1 +  |\tau|^{1/2})^{s/2} \mathfrak{C}_3 (s) \left(\| \boldsymbol{\phi} \|_{H^s(\mathbb{R}^d)} + \|\mathbf{F} \|_{L_1((0,\tau);H^s(\mathbb{R}^d))} \right).
		\end{equation*}
		При дополнительном предположении $\mathbf{F} \in L_p (\mathbb{R}_\pm; H^s (\mathbb{R}^d, \mathbb{C}^n))$, где $p \in [1, \infty]$, и при $\tau = \pm \varepsilon^{-\alpha}$, $0 < \varepsilon \le 1$, $0 < \alpha < 2s(s + 4/p')^{-1}$ справедлива оценка
		\begin{multline*}
		\| f^\varepsilon \mathbf{u}_\varepsilon (\cdot, \pm \varepsilon^{-\alpha}) - f_0 \mathbf{u}_0 (\cdot, \pm \varepsilon^{-\alpha}) \|_{L_2(\mathbb{R}^d)} \le \\ \le \varepsilon^{s(1-\alpha/2)/2} \cdot 2^{s/2} \mathfrak{C}_3 (s) \left(\| \boldsymbol{\phi} \|_{H^s(\mathbb{R}^d)} + \varepsilon^{-\alpha/p'} \|\mathbf{F} \|_{L_p(\mathbb{R}_\pm;H^s(\mathbb{R}^d))} \right).
		\end{multline*}
		Величина $\mathfrak{C}_3 (s)$ определена в~теореме \emph{\ref{sndw_A_eps_exp_enchcd_thrm_2}}. Здесь $p^{-1} + (p')^{-1} = 1$.
		\item Если $\boldsymbol{\phi} \in L_2 (\mathbb{R}^d; \mathbb{C}^n)$ и $\mathbf{F} \in L_1 (\mathbb{R}_\pm; L_2 (\mathbb{R}^d, \mathbb{C}^n))$, то
		\begin{equation*}
		\lim\limits_{\varepsilon \to 0} \| f^\varepsilon \mathbf{u}_\varepsilon (\cdot, \pm \varepsilon^{-\alpha}) - f_0 \mathbf{u}_0 (\cdot, \pm \varepsilon^{-\alpha}) \|_{L_2(\mathbb{R}^d)} = 0, \qquad 0 < \alpha < 2.
		\end{equation*}
	\end{enumerate}
\end{thrm}

Полученные общие результаты можно применить к конкретным уравнениям математической физики (см.~\cite[разделы~15, 16]{Su2017}). Так, для уравнения Шрёдингера с вещественной матрицей $g(\mathbf{x})$ (см.~\cite[п.~15.1]{Su2017}) справедлива теорема~\ref{hatA_eps_Cauchy_enchcd_thrm_1}. Для магнитного уравнения Шрёдингера с малым потенциалом (см.~\cite[п.~15.4]{Su2017}) и для двумерного волнового уравнения Паули (см.~\cite[п.~16.3]{Su2017}) выполнена теорема~\ref{A_eps_Cauchy_general_thrm}. Эти результаты являются точными как в отношении гладкости начальных данных, так и в отношении зависимости оценок от времени.

\end{document}